\documentclass{amsart}
\usepackage[utf8]{inputenc}

\usepackage[foot]{amsaddr}
\usepackage{amsmath}
\usepackage{amssymb}
\usepackage[all]{xy}
\usepackage{enumerate}
\usepackage{tikz}
\usepackage{longtable}
\usepackage{xcolor}
\usepackage{dsfont}
\usepackage{tensor}
\usepackage{hyperref}
\usetikzlibrary{decorations.markings}

\usepackage{ulem} 
\usepackage{upgreek} 

\newcommand\Z{{\mathbb Z}}
\newcommand\Q{{\mathbb Q}}
\newcommand\Qbar{{\overline{\mathbb Q}}}

\newcommand\C{{\mathbb C}}
\renewcommand\P{{\mathbb P}}
\newcommand\D{{\mathbb D}}

\renewcommand\H{{\mathbb H}}
\newcommand\Le{{\mathbb L}}
\newcommand\V{{\mathbb V}}
\newcommand\bE{{\mathbb E}}
\newcommand\J{{\mathbb J}}
\newcommand\A{{\mathbb A}}

\newcommand\E{{\mathcal E}}

\newcommand\M{{\mathcal M}}
\newcommand\cO{{\mathcal O}}

\newcommand\cG{{\mathcal G}}
\newcommand\cH{{\mathcal H}}
\newcommand\K{{\mathcal K}}

\newcommand\cP{{\mathcal P}}
\newcommand\cV{{\mathcal V}}

\newcommand{\cJ}{\mathcal{J}}
\newcommand\cA{{\mathcal A}}

\DeclareMathAlphabet\mathbfcal{OMS}{cmsy}{b}{n}

\newcommand\g{{\mathfrak g}}
\newcommand\h{{\mathfrak h}}
\renewcommand\k{{\mathfrak k}}
\newcommand\p{{\mathfrak p}}
\renewcommand\r{{\mathfrak r}}

\newcommand\m{{\mathfrak m}}

\newcommand\e{{\epsilon}}
\newcommand\w{{\omega}}

\newcommand\G{{\Gamma}}

\newcommand\zetabar{{\overline{\zeta}}}

\newcommand\vv{{\vec{\mathsf v}}}
\newcommand\ww{{\vec{\mathsf w}}}
\newcommand\uu{{\vec{\mathsf u}}}

\newcommand\bP{{\boldsymbol{\cP}}}

\newcommand\Pol{{\mathbf{Pol}}}

\newcommand\ee{{\mathbf e}}

\newcommand\bt{{\mathbf t}}

\newcommand\bv{\mathbf{v}}

\newcommand\bx{\mathbf{x}}
\newcommand\by{\mathbf{y}}
\newcommand\bX{\mathbf{X}}
\newcommand\bY{\mathbf{Y}}
\newcommand\bL{\mathbf{L}}
\renewcommand\v{\mathbf{v}}

\newcommand\bmu{{\boldsymbol{\mu}}}

\newcommand\bsigma{{\boldsymbol{\sigma}}}
\newcommand\bsigmabar{{\overline{\boldsymbol{\sigma}}}}

\renewcommand\Vec{{\mathsf{Vec}}}
\newcommand\Rep{{\mathsf{Rep}}}
\newcommand\MHS{{\mathsf{MHS}}}

\newcommand\MTM{{\mathsf{MTM}}}

\renewcommand\sl{\mathfrak{sl}}

\newcommand\SL{{\mathrm{SL}}}

\newcommand\Gm{{\mathbb{G}_m}}

\newcommand\betti{{\mathrm{B}}}
\newcommand\DR{{\mathrm{DR}}}

\newcommand\op{\mathrm{op}}

\renewcommand\ss{\mathrm{ss}}

\newcommand\elp{\mathrm{ell}}
\newcommand\cyc{\mathrm{cyc}}

\newcommand\fte{\mathrm{fte}}
\newcommand\mot{\mathrm{mot}}
\newcommand\dch{\mathrm{dch}}

\newcommand\bdot{{\bullet}}
\newcommand\bs{\backslash}
\newcommand\bbs{{\bs\negthickspace \bs}}

\renewcommand\ll{\langle\langle}
\newcommand\rr{\rangle\rangle}

\newcommand\KZB{\mathrm{KZB}}
\newcommand\KZ{\mathrm{KZ}}
\renewcommand\Pol{{\mathbf{Pol}}}
\newcommand\Gal{{\mathrm{Gal}}}
\newcommand{\piun}{\pi_1^{\mathrm{un}}}

\newcommand\id{\operatorname{id}}
\newcommand\ad{\operatorname{ad}}
\newcommand\Ad{\operatorname{Ad}}

\newcommand\Span{\operatorname{span}}
\newcommand\Spec{\operatorname{Spec}}
\newcommand\Hom{\operatorname{Hom}}

\newcommand\Ext{\operatorname{Ext}}
\newcommand\End{\operatorname{End}}
\newcommand\Aut{\operatorname{Aut}}
\newcommand\Der{\operatorname{Der}}
\newcommand\SDer{\operatorname{SDer}}
\newcommand\SAut{\operatorname{SAut}}
\newcommand\Gr{\operatorname{Gr}}
\newcommand\Isom{\operatorname{Isom}}
\newcommand\Res{\operatorname{Res}}

\newcommand\Sym{\operatorname{Sym}}
\newcommand\Lie{\operatorname{Lie}}
\newcommand\Bl{\operatorname{Bl}}
\newcommand\Li{\mathrm{Li}}
\newcommand\topo{\mathrm{top}}
\newcommand\Ob{\operatorname{Ob}}
\newcommand\per{\operatorname{per}}

\renewcommand\Im{\operatorname{Im}}

\numberwithin{equation}{section}

\newtheorem{theorem}{Theorem}[section]
\newtheorem{lemma}[theorem]{Lemma}
\newtheorem{prop}[theorem]{Proposition}
\newtheorem{cor}[theorem]{Corollary}
\newtheorem{bigtheorem}{Theorem}

\theoremstyle{definition}
\newtheorem{definition}[theorem]{Definition}
\newtheorem{example}[theorem]{Example}

\theoremstyle{remark}
\newtheorem{remark}[theorem]{Remark}

\usepackage{color}

\begin{document}

\tikzset{->-/.style={decoration={
  markings,
  mark=at position #1 with {\arrow{>}}},postaction={decorate}}}
\tikzset{-<-/.style={decoration={
  markings,
  mark=at position #1 with {\arrow{<}}},postaction={decorate}}}

\title[Polylogarithms and Extensions of $\Q$ by $\Q(m)$]{Polylogarithm variations and motivic extensions of $\Q$ by ${\Q(m)}$}
\author{Eric Hopper}
\email{eric@hoppermath.com}

\begin{abstract}
Deligne and Goncharov constructed a neutral tannakian category of mixed Tate motives unramified over $\Z[\bmu_N,1/N]$. Brown and Hain--Matsumoto computed the depth 2 quadratic relations of the motivic Galois group of this category for $N = 1$. We take the first steps in generalizing their results to all $N \ge 1$ by realizing the generators of the motivic Galois group by derivations on the Lie algebra of the unipotent fundamental group of a restriction of the Tate elliptic curve. 

This representation is compatible with a natural identification of the odd rational $K$-groups of the rings $\Z[\bmu_N,1/N]$ with spaces of $\G_1(N)$ Eisenstein series, thus inducing a natural action of the prime to $N$ part of the Hecke algebra on the $K$-groups.

We establish these results by first showing the inclusion of $\P^1 - \{0,\bmu_N,\infty\}$ into the nodal elliptic curve with a cyclic subgroup of order $N$ removed induces a morphism of mixed Tate motives on unipotent fundamental groups and then by computing the periods of the limit mixed Hodge structure of an elliptic polylogarithm variation of MHS over the universal elliptic curve of $Y_1(N)$. 
\end{abstract}

\maketitle

\setcounter{tocdepth}{1}

\section{Introduction}

Let $N$ be a positive integer. Denote the $N$th roots of unity by $\bmu_N$,\label{not:bmu} and let $\MTM_N$ be the category of mixed Tate motives unramified over the ring $\Z[\bmu_N,1/N]$. In this paper, we study the action of the motivic Galois group $\pi_1(\MTM_N)$ on the unipotent fundamental group of the restriction of the Tate elliptic curve to the tangent vector $\partial/\partial q$ of the $q$-disk with $\bmu_N$ removed. In particular, we realize generators of $\pi_1(\MTM_N)$ by derivations on the Lie algebra of the unipotent fundamental group of this first order Tate curve. This representation is a generalization of the $N = 1$ case computed by Hain and Matsumoto \cite{hain:MEM}. Their work, together with results of Pollack \cite{pollack} and Brown \cite{brown:MMV}, provides an explanation of the Ihara--Takao \cite{ihara} relations in the depth 2 graded quotient of $\pi_1(\MTM_1)$. Goncharov \cite{gonch:sym} has shown non-trivial quadratic relations in the depth 2 graded quotient also exist when $N \ge 5$ is prime. Our result is the first step in computing these relations for all $N \ge 1$. 


\subsection{Mixed Tate motives}

Let $k$ be a number field, $S$ a set of primes, and $\cO_{k,S}$ the ring of $S$-integers in $k$. Using the work of Borel \cite{borel}, Beilinson \cite{beilinson}, and Levine \cite{levine}, Deligne and Goncharov \cite{DG} proved the existence of a $\Q$-tannakian category $\MTM(\cO_{k,S})$ of {\em mixed Tate motives} over $\cO_{k,S}$. One may view these motives as those arising from cohomology groups of genus 0 curves and their moduli spaces \cite{brown:mzv}.

The category $\MTM(\cO_{k,S})$ is equivalent to the category of finite dimensional representations of an affine group scheme $\cG^\DR_{k,S}$ over $\Q$, which is an extension of the multiplicative group $\Gm$\label{not:Gm} by a free prounipotent group $\K^\DR_{k,S}$:
$$
\xymatrix{
1 \ar[r] & \K^\DR_{k,S} \ar[r] & \cG^\DR_{k,S} \ar[r] & \Gm \ar[r] & 1.
}
$$
This extension splits canonically, which implies the Lie algebra of the kernel $\K^\DR_{k,S}$ is canonically isomorphic to the completion of a graded free Lie algebra $\k_{k,S}$. Thus, the category $\MTM(\cO_{k,S})$ is also equivalent to the category of finite dimensional graded representations of $\k_{k,S}$.

In the case $\cO_N := \Z[\bmu_N,1/N]$, Deligne and Goncharov proved the coordinate ring of the unipotent path torsor of $\P^1 - \{0,\bmu_N,\infty\}$ with suitable tangential base points is an ind-object of $\MTM(\cO_N)$. The periods of this object are $\Q$-linear combinations of powers of $2\pi i$ and the {\em multiple polylogarithms}\label{not:Li}
\begin{equation}
    \label{eqn:Lisum}
    \Li_{n_1,\ldots,n_m}(z_1,\ldots,z_m) := \sum_{0<k_1<k_2<\cdots<k_m} \frac{z_1^{k_1}z_2^{k_2}\cdots z_m^{k_m}}{k_1^{n_1}k_2^{n_2}\cdots k_m^{n_m}}
\end{equation}
evaluated at $N$th roots of unity. These values are referred to as {\em $N$-cyclotomic multiple zeta values (MZVs)}. They were first studied in the context of mixed Tate motives by Goncharov in \cite{gonch:mzv}. The integer $m$ is called the {\em depth} of an MZV. The $\Q$-vector space generated by the cyclotomic MZVs is filtered, but not graded, by depth.\footnote{The classical MZV $\zeta(n_1,\dots,n_m)$ is equal to $\Li_{n_1,\ldots,n_m}(1,\dots,1)$, and for example, $\zeta(4) = 4\zeta(1,3)=\zeta(1,1,2)=\frac43 \zeta(2,2)$.} We will prove the following (see Theorem \ref{thm:obj}).
\begin{bigtheorem}
\label{bthm:obj}
Let $E_{\partial/\partial q}$ denote the restriction of the Tate elliptic curve over the tangent vector $\partial/\partial q$. The Lie algebra of the unipotent fundamental group of $E_{\partial/\partial q} - \bmu_N$ with a suitable choice of tangential base point is a pro-object of $\MTM(\cO_N)$. Its periods are $\Q(2\pi i)$-linear combinations of $N$-cyclotomic MZVs. 
\end{bigtheorem}

The proof is tannakian in nature. We then show the inclusion $\Gm - \bmu_N \hookrightarrow E_{\partial/\partial q} - \bmu_N$ induces an an action of the motivic Galois group of $\MTM(\cO_N)$ on the Lie algebra of the unipotent fundamental group of $E_{\partial/\partial q} - \bmu_N$ and that this action is canonical. 



\subsection{The depth filtration on {{$\k_N$}}} 

A question of interest is whether the $\Q$-algebra generated by $2\pi i$ and the $N$-cyclotomic MZVs includes all periods of all objects of $\MTM(\cO_N)$. Deligne \cite{deligne:23468} showed this to be the case when $N = 2$, 3, 4, 6, or 8. Other values of $N$, including the $N = 1$ case proven by Brown \cite{brown:MTM}, are more difficult due to the existence of relations in the associated ``depth graded'' of the motivic Lie algebra $\k_N$. Our work is a first step in understanding such relations in the depth 2 graded quotient for general $N$. 

We now define the {\em depth filtration} $D^\bullet$ on the Lie algebra $\k_N$. Let $\bL(\ee_0,\ee_\zeta \mid \zeta \in \bmu_N)^\wedge$ be the completed free Lie algebra on $\{\ee_0,\ee_\zeta \mid \zeta \in \bmu_N\}$. The Knizhnik--Zamolodchikov equations induce a canonical isomorphism 
$$
\Lie \piun(\P^1 - \{0,\bmu_N,\infty\},\vv_1)^\DR \cong \bL(\ee_0,\ee_\zeta \mid \zeta \in \bmu_N)^\wedge. 
$$
Define a decreasing filtration on $\Lie \piun(\P^1 - \{0,\bmu_N,\infty\},\vv_1)^\DR$ corresponding to the degree in the generators $\ee_\zeta$. The depth filtration $D^\bullet$ on $\k_N$ is the pullback of this filtration under the motivic Galois representation 
$$
\k_N \to \Der \Lie \piun(\P^1 - \{0,\bmu_N,\infty\},\vv_1)^\DR.
$$
The filtration $D^\bullet$ of $\k_N$ is dual to an increasing filtration of the periods of $\MTM(\cO_N)$ compatible with the depth filtration of the $\Q$-algebra generated by the $N$-cyclotomic MZVs.

As mentioned earlier, although the Lie algebra of $\k_N$ is free, its associated depth graded Lie algebra $\Gr_D^\bullet \k_N$ is not free in general. When $N = 1$, there are well-known relations \cite{ihara, pollack, brown:MMV,hain:MEM} in $\Gr_D^2 \k_1$ associated to cusp forms of $\SL_2(\Z)$. Dual relations between MZVs were established in \cite{GKZ}. While Deligne \cite{deligne:23468} has shown $\Gr_D^\bullet \k_N$ is free when $N = 2$, 3, 4, 6, and 8, these values are expected to be exceptional. Goncharov \cite{gonch:sym} has proven the existence of relations in $\Gr^2_D \k_N$ when $N$ is prime and at least 5. 


In Section 11, we realize for each $N \ge 1$ the generators of $\Gr_D^1 \k_N$ by derivations on the Lie algebra of the unipotent fundamental group of $E_{\partial/\partial q} - \bmu_N$. More precisely, we compute the motivic Galois representation 
$$
\phi_\elp : \Gr_D^1\k_N \longrightarrow \Der \left[\Lie \piun(E_{\partial/\partial q} - \bmu_N,\vv_1)^\DR/D^2\right].
$$
The Lie algebra $\Lie \piun(E_{\partial/\partial q} - \bmu_N,\vv_1)^\DR$ is canonically isomorphic to the completed Lie algebra
$$
\bL(\bX,\bY,\bt_\zeta \mid \zeta \in \bmu_N)^\wedge/\Big([\bX,\bY] = \sum_{\zeta\in\bmu_N} \bt_\zeta\Big),
$$ 
and its depth filtration is given by the degree in the letters $\bt_\zeta$:
$$
D^m = \{\text{Lie words with degree in $\bt_\zeta \ge m$}\}. 
$$
Define the derivations
$$
\e^\op_{m+1,\zeta} \equiv \ad (\bY^{m-1} \cdot (\bt_\zeta + (-1)^{m+1}\bt_\zetabar)) \bmod D^2
$$
(full formula in \cite[\S7]{hopper}). 
The following is a weak version of Theorem \ref{thm:head}. 

\begin{bigtheorem}
\label{bthm:head}
The action of $\k_N$ on $\Lie \piun(E_{\partial/\partial q} - \bmu_N,\vv_1)/D^2$ factors through the abelianization map $\k_N \to \Gr_D^1\k_N$. The quotient $\Gr_D^1\k_N$ has canonical generators $\bsigma_{m,\zeta}$ indexed by integers $m \ge 2$ and primitive $N$th roots of unity $\zeta \in \bmu_N$. Complex conjugation induces the relations $\bsigma_{m,\zeta} = (-1)^{m+1}\bsigma_{m,\zetabar}$. The action of $\bsigma_{m,\zeta} \in \Gr_D^1\k_N$ on $\Lie \piun(E_{\partial/\partial q} - \bmu_N,\vv_1)/D^2$ is given by
$$
\phi_\elp(\bsigma_{m,\zeta}) \equiv \e^\op_{m+1,\zeta} + \sum_{\substack{\eta \in \bmu_N \\ \text{not primitive}}} c_\eta \e^\op_{m+1,\eta} \bmod D^2
$$
where $c_\eta \in \Q$. 
\end{bigtheorem}
The result follows from the observation that $\Gm - \bmu_N$ includes into the degeneration of the an elliptic curve with a cyclic subgroup of order $N$ removed. This inclusion induces a degeneration of the elliptic and classical (cyclotomic) polylogarithm variations. This degeneration is well-known in the $N = 1$ case. For example, it was used by Hain to study the KZB connection \cite[\S3]{hain:kzb} and by Huber--Kings to study the Bloch--Kato conjectures for the Riemann zeta function \cite{HK}. Our proof considers the degeneration for general $N \ge 1$ to calculate the periods of the limit MHS of the elliptic polylog variation. 

The derivations $\e^\op_{m+1,\zeta}$ appearing in Theorem \ref{bthm:head} are cyclotomic generalizations of derivations $\e_{m}$ first defined by Tsunogai \cite[\S3--4]{tsunogai} and appear in the universal elliptic KZB connection \cite{CEE,LR,hain:kzb}. We explicitly compute the constants $c_\eta$ when $N$ is a prime power and when $N = 6$. Theorem \ref{bthm:head} is a generalization of the $N = 1$ case proven in \cite[Thm 29.4]{hain:MEM}. Importantly, since the map $\phi_\elp$ is injective on $\Gr_D^1 \k_N$, the depth two relations between the generators $\bsigma_{m,\zeta}$ correspond exactly to quadratic relations between the derivations $\e_{m+1,\zeta}$.

\subsection{Hecke action on $K$-groups}

A second consequence of the injectivity of $\phi_\elp$ is the computation of a natural Hecke action on the groups $K_{2m-3}(\cO_N) \otimes \Q$ when $m \ge 3$. When $m \ge 3$, the KZB connection canonically identifies $\Ext^1_{\MTM(\cO_N)}(\Q,\Q(m-1))$ with the dual space of the $\Q$-span of the Eisenstein series 
$$
G_{m,\zeta}(\tau)  = \sum_{\substack{k,\ell \in \Z \\ (k,\ell) \neq (0,0)}} \frac{\zeta^k}{(k\tau + \ell)^m}
$$
with $\zeta \in \bmu_N$ primitive. When $p \nmid N$, this space is $T_p$-invariant. Therefore, Deligne and Goncharov's isomorphism $\Ext^1_{\MTM(\cO_N)}(\Q,\Q(m-1)) \cong K_{2m-3}(\cO_N) \otimes \Q$ implies the following.
\begin{bigtheorem}
    \label{bthm:hecke}
    When $m \ge 3$, there is a natural identification of $K_{2m-3}(\cO_N) \otimes \Q$ with the dual of a subspace of $\G_1(N)$ Eisenstein series of weight $m$, which induces an action of the prime to $N$ Hecke operators on $K_{2m - 3}(\cO_N) \otimes \Q$.
\end{bigtheorem}
For $p \nmid N$ prime, we give an explicit formula of the action of $T_p$ on the Hodge realizations of extensions in $\Ext^1_{\MTM(\cO_N)}(\Q,\Q(m-1))$. The details can be found in \S\ref{sec:hecke}.

\subsection{Outline}

We begin in Sections 2 and 3 by reviewing definitions and setting notation for working in the category of mixed Tate motives. This includes a discussion of the Betti and de~Rham realization of the motivic Galois group of $\MTM_N$ and their actions on motivic periods. In Section 4, we review from \cite{DG} the motivic unipotent path torsor of $\P^1 - \{0 , \mu_N, \infty\}$. We also introduce Brown's category of generalized Hodge structures \cite{brown:periods} and Hain's construction of canonical mixed Hodge structures on fundamental groups \cite{hain:bowdoin}. Section 5 begins with a summary of the level $N$ KZB connection derived by Calaque and Gonzalez \cite{CG} and the level $N$ generalization of the Hain map relating the KZ and KZB connections along the Tate curve $E_{\partial/\partial q}$ \cite{hain:kzb, hopper}. These facts are gathered in Section 6 to prove Theorem \ref{bthm:obj}. 

In Section 7, we define the depth filtrations on the KZ and KZB local systems, define the cyclotomic and elliptic polylogarithm quotients, and show their compatibility with the Hain map. Section 8 applies the Hain map to compute the limit MHS of the polylogarithm variations at the identity of the Tate curve. It follows from Theorem \ref{bthm:obj} that these MHS are Hodge realizations of mixed Tate motives. In Sections 9 and 10, we proceed to compute the motivic Galois actions on these motivic MHS. These formulas allow us to prove Theorems \ref{bthm:head} and \ref{bthm:hecke} in Sections 11 and 12, respectively. 

The appendices include relevant background information on tannakian categories, admissible variations of MHS, iterated integrals, and transport of linearized connections. Appendix E is an index of notation. \\

\noindent {\em Acknowledgements:} Many of the results in this paper appeared in my Ph.D. thesis \cite{thesis} completed under the supervision of Richard Hain at Duke University. I am grateful for his guidance in approaching these problems and many suggestions in the writing of this manuscript. Further results are from my time at the Department of Mathematics at the University of Rochester.

\section{Notation and conventions}

We work in the category of complex analytic varieties unless otherwise noted.

Throughout, we denote by $\gamma$ the matrix $\begin{pmatrix} a & b \cr c & d \end{pmatrix} \in \SL_2(\Z)$, and $a$, $b$, $c$, and $d$ will refer to its entries. The group $\SL_2(\Z)$ acts on the upper half plane $\h$\label{not:h} in the standard way 
$$
\gamma : \tau \longmapsto \frac{a \tau + b}{c \tau + d}.
$$

We will use the topologist's convention for composition of paths. If $X$ is a topological space, $\alpha,\beta : [0,1] \to X$, and $\alpha(1) = \beta(0)$, then $\alpha\beta$ denotes the path by first proceeding along $\alpha$ and then along $\beta$.

The one-dimensional pure $\Q$-Hodge structure of type $(-n,-n)$ will be denoted by $\Q(n)$. Its $\Q$-Betti and $\Q$-de~Rham generators are $\Q\ee^\betti$ and $\Q\ee^\DR$, respectively. The Betti to de~Rham comparison isomorphism takes $\ee^\betti$ to $(2\pi i)^n \ee^\DR$.

Suppose that $F$ is a field of characteristic 0 and that $V$ is a finite dimensional vector space over $F$. Denote by $\bL(V)$ the free Lie algebra on $V$. Recall that the universal enveloping algebra of $\bL(V)$ is the tensor algebra $T(V)$. Let $\bL(V)^\wedge$ be the completion of $\bL(V)$ with respect to its lower central series. Let $T(V)^\wedge$ be the completion of $T(V)$ with respect to powers of the augmentation ideal $I = \ker \psi$ where $\psi : T(V) \to F$ and $\psi(v) = 0$ for all $v \in V$.

The adjoint action of an element $x \in T(V)$ on $y \in \bL(V)$ will be denoted $x \cdot y$. This notation also applies to completions. If $a_n \in F$ and $x\in V$, then 
$$
\left(\sum_{n = 0}^\infty a_nx^n\right)\cdot y := \sum_{n = 0}^\infty a_n \ad_x^n(y).
$$
Also, when it is clear we are working in the derivation algebra $\Der \g$ of a Lie algebra $\g$ with trivial center, such as a free Lie algebra of rank $>1$, we will view $\g$ as a subalgebra of $\Der \g$ via the adjoint action $\ad : \g \to \Der \g$.

If $X$ is a Riemann surface, we let $\piun(X,x)$\label{not:piun} denote the unipotent completion of the fundamental group $\pi_1(X,x)$ over $\Q$. We also set $\p(X,x)$\label{not:p} to be the $\Q$-Lie algebra of $\piun(X,x)$. For background on unipotent completion, see \cite[\S5]{hopper}.


\section{Review of mixed Tate motives}

We first review some basic definitions and results of Deligne and Goncharov \cite{DG}. We assume familiarity with neutral tannakian categories, but a brief review is included in Appendix \ref{sec:tannaka}.




\subsection{Betti and de~Rham realizations}


The category $\MTM(\cO_{k,S})$ of mixed Tate motives over the ring $\cO_{k,S}$\label{not:OkS} of $S$-integers of a number field $k$ is neutral tannakian with respect to the natural fiber functors \label{not:ff}
$$
\omega^\DR : \MTM(\cO_{k,S}) \to \Vec_k \quad \text{and} \quad \omega^\betti : \MTM(\cO_{k,S}) \to \Vec_\Q.
$$
The images $\omega^\DR(V)$ and $\omega^\betti(V)$ of an object $V$ of $\MTM(\cO_{k,S})$ are called the de~Rham and Betti realizations of $V$, respectively. We will often use the shorthand $V^\DR_k$\label{not:VDRk} and $V^\betti$\label{not:VB} for $\omega^\DR(V)$ and $\omega^\betti(V)$, respectively. Each object $V$ of $\MTM(\cO_{k,S})$ has a weight filtration $M_\bullet$ by subobjects of $V$ in $\MTM(\cO_{k,S})$. It induces filtrations on $V^\DR_{k}$ and $V^\betti$. The de~Rham realization has an additional filtration $F^\bullet$ called the Hodge filtration. For each embedding $\sigma : k \hookrightarrow \C$, there is a comparison isomorphism of filtered vector spaces
\begin{equation}
    \label{eqn:csigma}
    c_{V,\sigma}^{\DR,\betti} : (V^\DR_{k},M_\bullet) \otimes_\sigma \C \xrightarrow{\sim} (V^\betti,M_\bullet) \otimes_\Q \C.
\end{equation}

\subsection{Simple objects of $\MTM(\cO_{k,S})$}
The simple object $\Q(-1)$ is defined to be the motive of $H^1(\Gm)$. The de~Rham realization $\Q(-1)^\DR_{k}$ is the algebraic de~Rham cohomology $H_\DR^1(\Gm) \cong k[\frac{dz}{z}]$ with coefficients in $k$. The Betti realization $\Q(-1)^\betti$ is the singular cohomology with rational coefficients $H^1_\mathrm{sing}(\Gm,\Q)$. The comparison isomorphism
$$
c_{\Q(-1),\sigma}^{\DR,\betti} : H^1_\DR(\Gm) \otimes \C \to H_\betti^1(\Gm) \otimes \C
$$
maps $[\frac{dz}{z}] \mapsto 2\pi i \gamma^\vee$, where $\gamma$ is the positive generator of $H_1^{\mathrm{sing}}(\Gm,\Z)$ and $\gamma^\vee$ denotes its Poincar\'e dual. 

We then set $\Q(-n) := \Q(-1)^{\otimes n}$ for any $n \in \Z$. The Hodge and weight filtrations of $V = \Q(n)$ are given by 
$$
V^\DR = F^{-n}V^\DR \supset F^{-n + 1}V^\DR = \{0\}
$$
and 
$$
\{0\} = M_{-2n - 1}V \subset M_{-2n}V = V.
$$

The de~Rham and Hodge realizations of $V = \Q(n)$ are given by $V^\DR = k$ and $V^\betti = \Q$. The comparison isomorphism $c_{V,\sigma}^{\DR,\betti} : V^\DR \otimes \C \to V^\betti \otimes \C$ is multiplication by $(2\pi i)^{-n}$. 


\subsection{Canonical $\Q$-structure of $V^\DR$}
\label{sec:vDRQ}

Define a functor $\w : \MTM(\cO_{k,S}) \to \Vec_\Q$ by\label{not:DGw}
$$
\omega(V) := \bigoplus_{m \in \Z} \Hom_{\MTM(\cO_{k,S})}(\Q(m),\Gr_{-2m}^M V).
$$
It is a fiber functor and $\omega(V)$ has Hodge and weight filtrations defined by $$
F^{-m}W_{-2m} \omega(V) = \Hom_{\MTM(\cO_{k,S})}(\Q(m),\Gr_{-2m}^M V).
$$

The de~Rham realization $V^\DR_{k}$ of every object $V$ of $\MTM(\cO_{k,S})$ has a canonical splitting
\begin{equation}
    \label{eqn:DRsplit}
    V^\DR_{k} \cong \bigoplus F^{-m}V^\DR_{k} \cap M_{-2m}V^\DR_{k} \cong \bigoplus \Gr_{-2m}^MV^\DR_{k}
\end{equation}
as $k$-vector spaces. By \cite[Prop. 2.10]{DG}, for each $V$ in $\MTM(\cO_{k,S})$, there is a canonical isomorphism
\begin{equation}
    \label{eqn:wisom}
    V^\DR_{k} \cong \omega(V) \otimes_\Q k.
\end{equation}
that preserves the Hodge and weight filtrations. Thus, the $\Q$-vector space $\omega(V)$ is a rational form of $V^\DR_k$. We will denote the $\Q$ vector space $\w(V)$ by $V^\DR$.\label{not:VDR} With this notation $V^\DR_{k} \cong V^\DR\otimes_\Q k$. The comparison isomorphism \eqref{eqn:csigma} becomes 
$$
c_V^{\DR,\betti} : V^\DR \otimes_\Q \C \xrightarrow{\sim} V^\betti \otimes_\Q \C.
$$
Note that this isomorphism still depends on the choice of embedding $\sigma : k \hookrightarrow \C$. 

\begin{example}
\label{ex:wQ}

Let $U = \A^1 - D$ where $D$ is a finite set of points defined over $k$. For each $s \in D$, the map $z \mapsto z - s$ induces an isomorphism $H^1(\Gm) \to H^1(\A^1-\{s\})$ between objects of $\MTM(\cO_{k,S})$. The inclusion $U \hookrightarrow \A^1 - \{s\}$ induces a morphism $H^1(\A^1 - \{s\}) \to H^1(U)$ between objects of $\MTM(\cO_{k,S})$. We then have the commutative diagram below where the horizontal maps are each the canonical isomorphism \eqref{eqn:wisom}.
$$
\xymatrix{
H^1_\DR(\Gm) \ar[r]^-\cong \ar[d] & \omega(H^1(\Gm)) \otimes k \ar[d] \cr
H^1_\DR(\A^1 - \{s\}) \ar[r]^-\cong & \omega(H^1(\A^1 - \{s\})) \otimes k \cr
H^1_\DR(U) \ar[r]^-\cong \ar[u] & \omega(H^1(U)) \otimes k \ar[u]
}
$$
Since the class $\frac{dw}{w} \in H^1_\DR$ is rational in $\omega(H^1(\Gm)) \otimes k$, it follows that $\frac{dw}{w-s}$ is rational (in the sense of \eqref{eqn:wisom}) in both $H^1_\DR(\A^1 - \{s\})$ and $H^1_\DR(U)$. 

For each embedding $\sigma : k \hookrightarrow \C$ we  have the commutative diagram
$$
\xymatrix{
H^1_\DR(U) \otimes_\sigma \C \ar[r]^{c^{\DR,\betti}_{V,\sigma}} & H^1_B(U) \otimes \C \ar@{=}[d] \cr
\omega(H^1(U)) \otimes_\sigma \C \ar[u]^\cong \ar[r]^-{c^{\DR,\betti}_V} & H^1_B(U) \otimes \C.
}
$$
The comparison isomorphism $c^{\DR,\betti}_{V,\sigma}$ maps the rational generator $\frac{dw}{w-s}$ of $H^1_\DR(U)$ to $ 2\pi i\gamma_{\sigma(s)}^\vee$, where $\gamma_{\sigma(s)}$ is the homology class of a small positive loop around $\sigma(s)$.  
\end{example}

Taken together, the de~Rham realization $V^\DR$ over $\Q$, the Betti realization $V^\betti$, the filtrations $M_\bullet$ and $F^\bullet$, and the comparison isomorphism $c_V^{\DR,\betti}$ form a mixed Hodge structure that we call the ``Hodge realization'' of $V$. More precisely, the $\Q$-vector space $(c_V^{\DR,\betti})^{-1}(V^\betti)$ underlies a MHS whose complexification is $V^\DR \otimes \C$. This MHS has weight graded quotients of Tate type, and its periods are the matrix coefficients of $c^{\DR,\betti}_V$. These numbers are also called the {\em periods} of the mixed Tate motive $V$.

\subsection{The $\ell$-adic realization}

There are also fiber functors $\omega_\ell : \MTM(\cO_{k,S}) \to \Vec_{\Q_\ell}$ for each prime $\ell$ that does not divide $N$. The $\Q_\ell$-vector space $V_\ell := \omega_\ell(V)$ is a $\Gal(\Qbar/k)$-representation filtered by $M_\bullet$, unramified outside $\ell$, and crystalline at $\ell$. It is called the $\ell$-adic \'etale realization of $V$. There is also a comparison isomorphism of filtered $\Q_\ell$-vector spaces
$$
c_V^{B,\ell} : (V^\betti,M_\bullet) \otimes \Q_\ell \xrightarrow{\sim} (V_\ell,M_\bullet).
$$
We only mention these realizations for the sake of completeness as the Hodge realization of $\MTM(\cO_{k,S})$ is fully faithful \cite{DG}, and thus the $\ell$-adic realizations are completely determined by the de~Rham and Betti realizations together with their filtrations and comparison isomorphism $c^{\DR,\betti}_V$. In this paper, we focus on the de~Rham and Betti fiber functors.

\subsection{Extensions}
\label{sec:MTMext}

Deligne and Goncharov \cite{DG} showed the extensions of simple objects in $\MTM(\cO_{k,S})$ have the property 
\begin{equation}
    \label{eqn:K}
    \Ext^j_{\MTM(\cO_{k,S})}(\Q,\Q(n)) \cong 
    \begin{cases}
        \Q & j = n = 0 \cr
        K_{2n - 1}(\cO_{k,S}) \otimes \Q & j = 1, n \ge 0 \cr
        0 & \text{otherwise},
    \end{cases}
\end{equation}
where $K_\bullet(\cO_N)$ are the algebraic $K$-groups of $\cO_N$. 

 In the later sections, we will focus on the case where $k$ is the cyclotomic field $\Q(\bmu_N)$ and $S$ is the set of primes dividing $N$. In this setting, $\cO_{k,S} = \Z[\bmu_N,1/N]$, which we shall denote by $\cO_N$.\label{not:ON} We further abbreviate $\MTM(\cO_N)$ to $\MTM_N$.\label{not:MTMN} Restricting to this case, it follows from Borel \cite{borel:rank} that
\begin{equation}
\label{eqn:dim}
\dim_\Q \Ext^1_{\MTM_N}(\Q,\Q(n)) = \begin{cases}
    1 & N = 1, \text{$n$ odd $\ge 3$} \cr
    0 & N = 1, \text{$n$ otherwise} \cr
    1 & N = 2, \text{$n$ odd} \cr
    0 & N = 2, \text{$n$ even} \cr
    \varphi(N)/2 - 1 + \upomega(N) & N \ge 3, n = 1 \cr
    \varphi(N)/2 & N \ge 3, n \ge 2,
\end{cases}
\end{equation}
where $\upomega(N) = \sum_{p\mid N} 1$ is the prime omega function and $\varphi(N)$ is the Euler phi function. In \S \ref{sec:gen}, we will specify a natural basis of $\Ext^1_{\MTM_N}(\Q,\Q(n))$.  

\subsection{Tannakian fundamental group} 

Let $\MTM(\cO_{k,S})^\ss$ denote the category of semi-simple mixed Tate motives over $\cO_N$. This category is tannakian with respect to the de~Rham and Betti fiber functors. For $\bullet \in \{\DR, B\}$, the action of $\pi_1(\MTM(\cO_{k,S})^\ss,\omega^\bullet)$ on $\omega^\bullet(\Q(-1))$ determines the action on any object of $\MTM(\cO_{k,S})^\ss$. The complex points of $\pi_1(\MTM(\cO_{k,S})^\ss,\omega^\bullet)$ are given by
\begin{align*}
    \pi_1(\MTM(\cO_{k,S})^\ss,\omega^\bullet)(\C) &= \Aut^\otimes(\omega^\bullet|_{\MTM(\cO_{k,S})^\ss})(\C) \cr
    &= \Aut^\otimes(\omega^\bullet|_{\Q(-1)})(\C) \cr
    &= \Aut_\C(\Q \otimes \C) \cr
    &= \Gm.
\end{align*}
The simple object $\Q(m)$ corresponds to the $m$th power of the standard character of $\Gm$. 

Let $\cG_{k,S}^\bullet$\label{not:GbulletN} denote the fundamental group of $\MTM(\cO_{k,S})$ with respect to the fiber functor $\omega^\bullet$ for $\bullet \in \{\DR,\betti\}$. The inclusion $\MTM(\cO_{k,S})^\ss \to \MTM(\cO_{k,S})$ induces a surjection $\chi^\bullet: \cG_{k,S}^\bullet \to \Gm$. The kernel $\mathcal{K}_{k,S}^\bullet$\label{not:KNS} acts trivially on objects of $\MTM(\cO_{k,S})^\ss$ and thus is prounipotent. Thus, for each fiber functor, we have the short exact sequence
\begin{equation}
    \label{eqn:Gses}
    \xymatrix{1 \ar[r] & \mathcal{K}_{k,S}^\bullet \ar[r] & \cG_{k,S}^\bullet \ar[r]^-{\chi^\bullet} & \Gm \ar[r] & 1.}
\end{equation}
In light of \S\ref{sec:vDRQ}, the de~Rham realization $V^\DR$ of every object $V$ of $\MTM(\cO_{k,S})$ has a canonical splitting
$$
V^\DR \cong \bigoplus F^mV^\DR \cap M_{2m}V^\DR
$$
as $\Q$-vector spaces. Thus, the functor $\Gr_\bullet^M : \MTM(\cO_{k,S}) \to \MTM(\cO_{k,S})^\ss$ induces a splitting of $\cG_{k,S}^\DR \to \Gm$ and defines an isomorphism 
$$
\cG_{k,S}^\DR \cong \mathcal{K}_{k,S}^\DR \rtimes \Gm.
$$

The Lie algebra of $\mathcal{K}_{k,S}^\DR$ is the de~Rham realization of a pro-object $\k^\wedge_{k,S}$ of $\MTM(\cO_{k,S})$. It is therefore the degree completion of its de~Rham realization $\k_{k,S}^\DR := \w(\k_{k,S}^\wedge)$\label{not:kOkS}. It is referred to as the {\em motivic Lie algebra} of $\MTM(\cO_{k,S})$. It is free. The de~Rham realization $V^\DR$ of each object $V$ of $\MTM(\cO_{k,S})$ is a graded representation of $\k_{k,S}^\DR$, and $\MTM(\cO_{k,S})$ is equivalent to the category of graded representations of $\k_{k,S}^\DR$ on finite dimensional rational vector spaces. 





\subsection{Motivic periods}
\label{sec:mperiods}

We define the rings of {\em de~Rham} and {\em Betti periods} of $\MTM(\cO_{k,S})$ to be the coordinate rings of $\cG_{k,S}^\DR$ and $\cG^\betti_{k,S}$, respectively 
$$
\cP^\bullet_{k,S} := \cO(\cG^\bullet_{k,S}) \quad \text{for $\bullet \in \{\DR, B\}$}. 
$$
We also define the ring of {\em motivic periods} of $\MTM(\cO_{k,S})$ by 
$$
\cP^\m_{k,S} := \cO(\Isom^\otimes(\omega,\omega^\betti)). \label{not:POkS}
$$
The left and right actions of $\cG_{k,S}^\DR$ and $\cG_{k,S}^\betti$ on $\Isom^\otimes(\omega,\omega^\betti)$ induce right and left coactions 
\begin{equation}
    \label{eqn:coaction}
    \Delta^\DR : \cP^\m_{k,S} \to \cP^\m_{k,S} \otimes \cP_{k,S}^\DR \quad \text{and} \quad \Delta^\betti : \cP^\m_{k,S} \to \cP^\betti_{k,S} \otimes \cP^\m_{k,S}.
\end{equation}
The action of $\cG^\DR_{k,S}$ is well-defined because the isomorphism \eqref{eqn:wisom} is canonical. 

Let $R$ be a subring of $\C$. Then $\cG_{k,S}^\bullet(R) = \Hom(\cP^\bullet_{k,S},R)$ and the coactions $\Delta^\bullet$ induce {\em Galois actions} of $\cG_{k,S}^\bullet(R)$ on $\cP^\m_{k,S}$. The de~Rham action is given by 
\begin{align*}
\cG^\DR(R) \times \cP^\m_{k,S} &\longrightarrow \cP^\m_{k,S} \otimes R \cr
(g, p^\m) & \longmapsto ((\id \otimes g) \circ \Delta^\DR)(p^\m).
\end{align*}
The Betti Galois action is defined analogously. 

The motivic periods of $\cP_{k,S}^\m$ may be written more concretely as symbols $[V,\eta,\gamma]^\m$, where $V \in \Ob(\MTM(\cO_{k,S}))$, $\eta \in V^\DR$, and $\gamma \in (V^\betti)^\vee$. In many settings, these symbols may be viewed as formal unevaluated integrals, which satisfy bilinearity relations and respect tensor products and morphisms in $\MTM(\cO_{k,S})$ \cite[\S 2]{brown:periods}. For each object $V$, there is a motivic comparison $c_V^\m$ 
\begin{equation}
\begin{aligned}
    \label{eqn:cVm}
    c_V^\m : V^\DR &\longrightarrow V^\betti \otimes \cP_{k,S}^\m \cr
    \eta &\longmapsto \sum_j[V,\eta,e_j^\vee] e_j,
\end{aligned}
\end{equation}
where $\{e_j\}$ is a basis of $V^\betti$. This comparison may be viewed as a $\cP^\m_{k,S}$-rational point of $\Isom^\otimes(\omega,\omega^\betti)$ and thus is natural with respect to the morphisms in $\MTM(\cO_{k,S})$. 

The {\em period map} $\per : \cP^\m_{k,S} \to \C$\label{not:per} sends 
$$
\per : [V,\gamma,\eta] \mapsto \gamma(c_V^{\DR,\betti}(\eta)),
$$
where $c_V^{\DR,\betti} : V^\DR \otimes \C \to V^\betti\otimes B$ is the comparison isomorphism of $V$. This map may be interpreted as the evaluation of the integral of the ``form'' $\eta$ against the ``cycle'' $\gamma$. It is straightforward to confirm that the diagram
$$
\xymatrix{V^\DR \ar[r]^-{c_V^\m} \ar[dr]_{c_V^{\DR,\betti}} & V^\betti \otimes \cP_{k,S}^\m \ar[d]^{\id \otimes \per} \cr & V^\betti \otimes \C}
$$
commutes. In the case $V^\DR$ and $V^\betti$ are the de~Rham and Betti cohomology of a variety defined over $\Q$, the image of $\per$ contains periods in the sense of Kontsevich and Zagier \cite{KZ}. Restricted to mixed Tate motives, Grothendieck's {\em period conjecture} claims the map $\per : \cP^\m_{k,S} \to \C$ is injective \cite{grothendieck}. An equivalent claim is that all polynomial relations between periods of objects of $\MTM(\cO_{k,S})$ are of ``geometric origin'' (i.e. arise from algebraic change of variables and Stokes' theorem) \cite{andre}. 
\begin{example}
The {\em Lefshetz period} is the simplest non-trivial example of a motivic period. Recall that $\Q(-1) = H^1(\Gm)$. Let $\sigma \in H^1_B(\Gm,\Q)^\vee$ be dual to a simple closed positively-oriented loop about the origin. Then the Lefshetz period, denoted $\Le$,\label{not:Le} is given by
$$
\Le = [H^1(\Gm),[dz/z],\sigma]^\m.
$$
Evaluating gives $\per \Le = \sigma(c_{\Q(-1)}^{\DR,\betti}([dz/z])) = 2\pi i$.
\end{example}

The Galois actions of $\cG^\bullet_{k,S}(R)$ on $\cP^\m_{k,S}$ may be described in terms of matrix entries. Every $g \in \cG_{k,S}^\bullet(R)$ is a natural isomorphism of $\omega^\bullet \otimes R$. Thus, for every realization $\omega^\bullet(V)$, the element $g$ defines an automorphism of $\omega^\bullet(V) \otimes R$. The left action of $\cG^\DR_{k,S}$ is given by 
$$
g \cdot [V,\eta,\gamma]^\m = [V,g \cdot \eta,\gamma]^\m \in \cP_{k,S}^\m \otimes R
$$ 
for $g \in \cG^\DR_{k,S}$. The right action of $\cG_B^N$ is analogous:
$$
[V,\eta,\gamma]^\m \cdot g = [V,\eta,\gamma \cdot g]^\m \in \cP_{k,S}^\m \otimes R
$$
for $g \in \cG^\betti_{k,S}$. This action is equivalent to the coaction \eqref{eqn:coaction} \cite[\S 2.3]{brown:periods}.

\begin{example}
We can compute the action of $\cG_{k,S}^\DR$ and $\cG_{k,S}^\betti$ on the Lefschetz period $\Le$. Since $\Q(-1)$ is semi-simple, the action of $\cG^\bullet_{k,S}(\C)$ factors through $\chi^\bullet : \cG^\bullet_{k,S} \to \Gm$. Thus, if $g \in \cG^\DR_{k,S}(\C)$, we have 
$$
g \cdot \Le = [H^1(\Gm),g \cdot [dz/z], \gamma] = [H^1(\Gm),\chi^\DR(g)^{-1}[dz/z],\gamma] = \chi^\DR(g)^{-1}\Le.
$$
Similarly, if $g \in \cG^\betti_{k,S}$, then 
$$
\Le \cdot g = [H^1(\Gm), [dz/z], \gamma \cdot g] = [H^1(\Gm),[dz/z],\chi^\betti(g)^{-1}\gamma] = \chi^\betti(g)^{-1}\Le.
$$
\end{example}

\section{The motivic path torsor of \texorpdfstring{$\P^1 - \{0,\bmu_N,\infty\}$}{TEXT}}
\label{sec:KZ}

Deligne and Goncharov \cite{DG} showed the coordinate ring of the motivic path torsor of $\P^1 - \{0,\bmu_N,\infty\}$ with appropriate choice of base points is an object of $\MTM(\cO_N)$ where $\cO_N = \Z[\bmu_N,1/N]$. This object is a rich source of motivic periods. These periods will appear as coefficients of the motivic version of the cyclotomic Drinfeld associator.


\subsection{Brown's category of generalized Hodge realizations}

Brown \cite[\S 3]{brown:periods} defines a category $\cH$\label{not:HR} of generalized Hodge realizations. An object $V$ of $\cH$ is a triple $(V^\betti,V^\DR,c)$ consisting of \begin{enumerate}
    \item A finite-dimensional $\Q$-vector space $V^\betti$ with a finite increasing filtration $M_\bullet$,
    \item A finite-dimensional $\Q$-vector space $V^\DR$ with finite increasing filtration $M_\bullet$ and finite decreasing filtration $F^\bullet$,
    \item An isomorphism 
    $$
    c : V^\DR \otimes \C \xrightarrow{\sim} V^\betti \otimes \C,
    $$
    \item An involution $F_\infty : V^\betti \xrightarrow{\sim} V^\betti$ called the {\em real Frobenius}. 
\end{enumerate}
These are also required to satisfy the following two conditions.
\begin{itemize}
    \item If $c_\DR$ (resp. $c_B$) is the $\C$-antilinear involution on $V^\DR \otimes \C$ (resp. $V^\betti \otimes \C$), then $c \circ c_\DR = (F_\infty \otimes c_B) \circ c$. 
    \item The filtrations $M_\bullet V^\betti$ and $cF^\bullet(V^\DR \otimes \C)$ equip $V^\betti$ with a graded-polarizable $\Q$-MHS. 
\end{itemize}
There are fiber functors $\omega^\betti_\cH$ and $\omega^\DR_\cH$\label{not:HRff} from $\cH$ to $\Vec_\Q$ taking $V$ to its corresponding realization. Deligne \cite[\S 1]{deligne:P1} showed $\cH$ is tannakian with respect to both $\omega^\betti_\cH$ and $\omega^\DR_\cH$.

The category $\MTM(\cO_{k,S})$ of mixed Tate motives over $\cO_{k,S}$ embeds as a full subcategory of $\cH$ via the functor
$$
\omega^\cH : V \longmapsto (V^\betti,V^\DR,c^{\DR,\betti}_V)
$$
\cite[\S 1]{DG}. Also, the discussion of motivic periods in \S\ref{sec:mperiods} extends analogously to the category $\cH$, where the ring of periods $\cP^\m_\cH$ is the coordinate ring $\cO(\Isom^\otimes(\omega^\DR_\cH,\omega^\betti_\cH))$. Thus, since $\omega^\cH$ is fully faithful, it induces an inclusion 
$$
\cP^\m_{k,S} \hookrightarrow \cP^\m_\cH.
$$

In the next section, we review the construction of the canonical MHS on fundamental groups. These MHS will be viewed as objects of $\cH$. 

\subsection{The MHS of fundamental groups}
\label{sec:piMHS}

There is a canonical MHS on the fundamental group of any smooth pointed complex algebraic variety $X$ \cite{hain:bowdoin, hain:drt}. We recall the construction in the case $U = \P^1 - (D \cup \{\infty\})$, where $D \subset \A^1(k)$ is a finite set of points defined over a number field $k$. 

Let $I$ be the augmentation ideal of the group algebra $\Q\pi_1(U,x)$. Denote the $I$-adic completion by $\Q\pi_1(U,x)^\wedge$ (see \cite[\S 5]{hopper}). We also define the $I$-adic completion of the path torsor $\pi(U;x,y)$ by 
\begin{equation}
    \label{eqn:torsor}
    \Q\pi(U;x,y)^\wedge := \varprojlim_n \Q\pi(U;x,y)/I^n\pi(U;x,y).
\end{equation}


Denote the 1-form $\frac{dw}{w - s}$ by $\omega_s$.\label{not:ws} Let $V$ be the $\Q$-vector space
$$
V = \omega(H^1(U)) = \bigoplus_{s \in D} \Q \omega_s
$$
(see Example \ref{ex:wQ}). The dual space $V^\vee = \Hom_\Q(V,\Q)$ has basis $\{\ee_s \mid s \in D\}$.\label{not:J} Define the tensor algebra $A := T(V^\vee)$. It has the structure of a Hopf algebra with coproduct generated by $\Delta \ee_s = \ee_s \otimes 1 + 1 \otimes \ee_s$ for each $s \in D$. There is also an augmentation ideal $J$ of $A$ generated by $V^\vee$. Denote the $J$-adic completion of $A$ by $A^\wedge$. There are natural Hodge and weight filtrations on $A^\wedge$ since each $\ee_s$ has type $(-1,-1)$. 

Define the universal connection form 
$$
\Omega = \sum_{\substack{s \in D}}\omega_s\ee_s \in V \otimes V^\vee.
$$
Chen's transport of $\Omega$ \eqref{eqn:chenT} defines a map 
\begin{equation}
    \label{eqn:chen}
    \Theta_{x,y} : \C\pi(U;x,y)^\wedge \to A^\wedge \otimes \C.
\end{equation}
This is an isomorphism of complete Hopf algebras \cite[(3.5)]{chen}. It induces a MHS on $\Q\pi(U;x,y)^\wedge$. The periods are iterated integrals of the 1-forms $\omega_s$ over paths from $x$ to $y$. The triple 
$$
(\Q\pi(U;x,y)^\wedge,A^\wedge,\Theta_{x,y})
$$ 
is a pro-object of Brown's category $\cH$ \cite[\S 5]{hain:bowdoin}.  

\begin{remark}
In the case $x = y$, the completion $\Q\pi_1(U,x)^\wedge$ is the universal enveloping algebra of Lie algebra $\p(U,x)$ over $\Q$. The natural inclusion $\p(U,x) \hookrightarrow \Q\pi_1(U,x)^\wedge$ induces a MHS on $\p(U,x)$. The comparison isomorphism $\p(U,x) \to \bL(V^\vee)^\wedge$ is also Chen's transport of $\Omega$ but each $\ee_s$ acts by left adjoint. 
\end{remark}

The coordinate ring $\cO(A^\wedge \otimes \C)$ is the $\C$-vector space spanned by the iterated integrals
$$
\int \omega_{s_1} \cdots \omega_{s_r}.
$$
This is the tensor coalgebra on $H^0(\P^1,\Omega_{\P^1}^1(\log D))$ with the shuffle product \eqref{eqn:shuffle}. Moreover, $\cO(A^\wedge \otimes \C)$ is endowed with a $\Q$-mixed Hodge structure. Since each $\omega_s$ is of type $(1,1)$, the Hodge and weight filtrations are 
\begin{align*}
    F^p &= \Span\left\{\int \omega_{s_1} \cdots \omega_{s_r} \middle| r \ge p\right\} \cr
    W_{2m+1} = W_{2m} &= \Span\left\{\int \omega_{s_1} \cdots \omega_{s_r} \middle| r \le m \right\}.
\end{align*}
The $\Q$-structure is spanned by the integrals which take values in $\Q$ over paths in $\pi(U;x,y)$. 

By the construction above, the path torsors define a variation of MHS (VMHS) over $U \times U$ 
$$
\{\Q\pi(U;x,y)^\wedge\}_{(x,y)} \to U \times U.
$$
The connection form $\Omega$ is in $F^{-1}W_{-2}$ and has pro-nilpotent residues. Monodromy $\Theta_{x,x}$ is the identity on the weight graded quotients and thus is unipotent. Thus, the relative weight filtration at each point in $(D\times U) \cup (U \times D)$ is equal to the weight filtration. Hence, the variation is admissible \cite{hainzucker}. See Appendix \ref{sec:avmhs} for background on VMHS and admissibility. 

Restricting the variation to the diagonal, the fibers are each completed fundamental groups $\Q\pi_1(U,x)^\wedge$. Transport from $x$ to $y$ is conjugation by $\Theta_{y,x}$. The inclusion $\p(U,x) \hookrightarrow \Q\pi_1(U,x)^\wedge$ generates a sub-VMHS over $U$
\begin{equation}
    \label{eqn:pvar}
    \{\p(U,x)\}_x \to U.
\end{equation}
Transport from $x$ to $y$ is left adjoint by $\Theta_{y,x}$.

\subsection{Tangential base points} 

Let $X$ be a Riemann surface and $D$ a finite set of points on $X$. The {\em real oriented blow-up} $\Bl_D^oX$ of $X$ at $D$ is a Riemann surface with boundary circle replacing each point of $D$. There exists a continuous projection $\Bl_D^oX \to X$, where the preimage of a point $P \in D$ is a circle viewed as the quotient of $T_PX - \{0\}$ by the positive real numbers.  The restriction  of the projection $(\Bl_D^oX - \partial \Bl_D^oX) \to X - D$ is a biholomorphism. 

Each tangent vector $\vec{v} \in T_PX -\{0\}$ determines a point $[\vec{v}] \in \Bl_D^oX$. Deligne \cite{deligne:P1} defines the fundamental group $\pi_1(X - D,\vec{v})$ of $X-D$ with tangential base point $\vec{v}$ to be $\pi_1(\Bl_D^oX,[\vec{v}])$. When $r$ is a positive real number, the biholomorphism $(\Bl_D^oX - \partial \Bl_D^oX) \to X - D$ induces a canonical isomorphism between $\pi_1(X - D,\vec{v})$ and $\pi_1(X - D,r\vec{v})$. However, the limit MHS on the unipotent completions $\piun(X - D,\vec{v})$ and $\piun(X - D,r\vec{v})$ are {\em not} necessarily isomorphic unless the local monodromy at $P$ is trivial (see Appendix \ref{sec:avmhs}).

\subsection{Realizations of \texorpdfstring{{{$\tensor*[_x]{\Pi}{^\mot_y}$}}}{TEXT}}

We can now describe the Hodge realization and motivic periods of Deligne and Goncharov's object in the category $\MTM_N$ of mixed Tate motives over $\cO_N = \Z[\bmu_N,1/N]$. These are discussed in greater detail in \cite[\S 5]{DG} and \cite[\S 2]{brown:P1}. 

For convenience, we set 
$$
U_N = \P^1 - \{0,\bmu_N,\infty\} = \Gm - \bmu_N.
$$
\label{not:UN}Observe that $U_N$ is a scheme over $\Q$
$$
U_N = \Spec \Q[Z]/(Z^N - 1).
$$
Fix tangential base points $\ww_\zeta = \zeta\partial/\partial w \in T_0U_N$\label{not:wz} and $\vv_\zeta = \zeta\partial/\partial w \in T_\zeta U_N$\label{not:vz} for all $\zeta \in \bmu_N$ (see Figure \ref{fig:bp}). 
\begin{figure}
\centering
    \begin{tikzpicture}[scale=1]
        \draw[thick,fill=gray!20] (0, 0) circle (2cm);
        
        \draw[red,thick,<->] (.201,.618) -- (.417,1.284);
        \draw[red,thick,<->] (-.526,.382) -- (-1.092,.794);
        \draw[red,thick,<->] (-.526,-.382) -- (-1.092,-.794);
        \draw[red,thick,<->] (.201,-.618) -- (.417,-1.284);
        \draw[red,thick,<->] (.65,0) -- (1.35,0);
        
        \draw[blue,thick,->] (0,0) -- (.108,.333);
        \draw[blue,thick,->] (0,0) -- (-.283,.206);
        \draw[blue,thick,->] (0,0) -- (.108,-.333);
        \draw[blue,thick,->] (0,0) -- (-.283,-.206);
        \draw[blue,thick,->] (0,0) -- (.35,0);
        
        \draw[fill=black] (0, 0) circle (.07cm);
        \draw[fill=black] (1, 0) circle (.07cm);
        \draw[fill=black] (.309, .951) circle (.07cm);
        \draw[fill=black] (-.809, .588) circle (.07cm);
        \draw[fill=black] (-.809, -.588) circle (.07cm);
        \draw[fill=black] (.309, -.951) circle (.07cm);
        
        \draw (.417,1.434) node {{\color{red} $\vv_\zeta$}};
        \draw (.108,.483) node {{\color{blue} $\ww_\zeta$}};
        
    \end{tikzpicture}
    \hspace{.8cm}
    \begin{tikzpicture}[scale=1]
        \draw[thick,fill=gray!20] (0, 0) circle (2cm);
        \draw[fill=white] (0, 0) circle (.15cm);
        \draw[fill=white] (1, 0) circle (.15cm);
        \draw[fill=white] (.309, .951) circle (.15cm);
        \draw[fill=white] (-.809, .588) circle (.15cm);
        \draw[fill=white] (-.809, -.588) circle (.15cm);
        \draw[fill=white] (.309, -.951) circle (.15cm);
        
        \draw[fill=red,red] (1.15, 0) circle (.05cm);
        \draw[fill=red,red] (.355, 1.094) circle (.05cm);
        \draw[fill=red,red] (-.93, .676) circle (.05cm);
        \draw[fill=red,red] (-.93, -.676) circle (.05cm);
        \draw[fill=red,red] (.355, -1.094) circle (.05cm);
        
        \draw[fill=red,red] (.85, 0) circle (.05cm);
        \draw[fill=red,red] (.263, .808) circle (.05cm);
        \draw[fill=red,red] (-.688, .5) circle (.05cm);
        \draw[fill=red,red] (-.688, -.5) circle (.05cm);
        \draw[fill=red,red] (.263, -.808) circle (.05cm);
        
        \draw[fill=blue,blue] (.15, 0) circle (.05cm);
        \draw[fill=blue,blue] (.046, .143) circle (.05cm);
        \draw[fill=blue,blue] (-.121, .088) circle (.05cm);
        \draw[fill=blue,blue] (-.121, -.088) circle (.05cm);
        \draw[fill=blue,blue] (.046, -.143) circle (.05cm);
        
        \draw (.417,1.384) node {{\color{red} $\vv_\zeta$}};
        \draw (.108,.333) node {{\color{blue} $\ww_\zeta$}};

    \end{tikzpicture}
    \caption{The tangential base points $\pm\vv_\zeta$ and $\ww_\zeta$ in $U_N$ and $\Bl^o_{\{0,\bmu_N\}}U_N$}
    \label{fig:bp}
\end{figure}
Set $\tensor*[_x]{\Pi}{^\betti_y} := \Q\pi(U_N;x,y)^{\wedge}$, the $I$-adic completion of the path torsor \eqref{eqn:torsor}. Likewise, set $\tensor*[_x]{\Pi}{^\DR_y}$ to be the tensor coalgebra $A^\wedge$ as defined in \S\ref{sec:piMHS} with $U = U_N$ and $D = \{0\} \cup \bmu_N$. There is a comparison isomorphism 
\begin{equation}
    \label{eqn:cBdR}
    c^{B,\DR} : \tensor*[_x]{\Pi}{^\betti_y} \otimes_\Q \C \to \tensor*[_x]{\Pi}{^\DR_y} \otimes_\Q \C
\end{equation}
as defined in \eqref{eqn:chen}. 
\begin{theorem}[\cite{DG}]
\label{thm:DG}
Suppose $x,y,z \in \{\pm\vv_\zeta, \ww_\zeta \mid \zeta \in \bmu_N\}$. The coordinate rings $\cO(\tensor*[_x]{\Pi}{^\betti_y})$ and $\cO(\tensor*[_x]{\Pi}{^\DR_y})$ are the Betti and de~Rham realizations, respectively, of an ind-object $\cO(\tensor*[_x]{\Pi}{^\mot_y})$ in the category $\MTM_N$. The comparison isomorphism is induced by \eqref{eqn:cBdR}. The coproducts $\cO(\tensor*[_x]{\Pi}{^\mot_z}) \to \cO(\tensor*[_x]{\Pi}{^\mot_y}) \otimes \cO(\tensor*[_y]{\Pi}{^\mot_z})$ are morphisms in $\MTM_N$.
\end{theorem}

\subsection{Periods of \texorpdfstring{{{$\tensor*[_x]{\Pi}{^\mot_y}$}}}{TEXT}}

The universal 1-form $\Omega$ is given by 
$$
\Omega = \omega_0\ee_0 + \sum_{\zeta \in \bmu_N} \omega_\zeta\ee_\zeta.
$$
By the theorem, the transport $c^{B,\DR}$ of $\Omega$ is the transpose of the comparison isomorphism of $\cO(\tensor*[_x]{\Pi}{^\mot_y})$. Let $\cP_N^\m$\label{not:PN} denote the motivic periods of $\MTM_N$. The matrix coefficients of $c^{B,\DR}$ are in $\cP_N^\m$, and thus $c^{B,\DR}$ restricts to a homomorphism
$$
c^{B,\DR} : \Q\pi_1(U_N;x,y)^\wedge \to \cP^\m_N\ll\ee_0,\ee_\zeta \mid \zeta \in \bmu_N\rr
$$
of cocommutative Hopf algebras. The image is generated as an algebra by two distinct types of elements: exponentials and Drinfeld associators. 

\subsubsection{Exponentials}

In this case, suppose $x,y \in \{\pm\vv_\zeta, \ww_\zeta \mid \zeta \in \bmu_N\}$ are anchored at the same point $r \in \{0\} \cup \bmu_N$. Let $\gamma_r^{x,y}$ denote the counterclockwise path about $r$ from $x$ to $y$. If $x = y$, then $\gamma_r^{x,x}$ is a simple closed loop about $r$ based at $x$. Then
\begin{align*}
    c^{B,\DR}(\gamma_r^{x,y}) &= \sum_{[\omega_{r_1}|\cdots|\omega_{r_n}] \in \cO(\tensor*[_x]{\Pi}{^\DR_y})}  [\cO(\tensor*[_x]{\Pi}{^\mot_y}),[\omega_{r_1}|\cdots|\omega_{r_n}],(\gamma_r^{x,y})^\vee]^\m\ee_{r_1}\cdots \ee_{r_n} \cr
    &= \sum_{n \ge 0} [\cO(\tensor*[_x]{\Pi}{^\mot_y}),[\underbrace{\omega_{r}|\cdots|\omega_{r}}_{n \text{ times}}],(\gamma_r^{x,y})^\vee]^\m\ee_r^n \cr
    &= \sum_{n \ge 0} \frac{(k\Le/N)^n\ee_r^n}{n!} \cr
    &= \exp(k\Le\ee_r/N),
\end{align*}
where $k \in \{1,2,\ldots, N\}$ such that $y = e^{2\pi i k/N}x$. 

\subsubsection{Cyclotomic Drinfeld associators}
\label{sec:drinfeld}

Fix $x = \ww_1$ and $y = -\vv_1$. Let $\dch \in \pi_1(U_N;\ww_1,-\vv_1)$\label{not:dch} denote the straight line path, or ``droit chemin,'' from $\ww_1$ and $-\vv_1$. Then 
\begin{equation}
    \label{eqn:drinfeld}
    c^{B,\DR}(\dch) = \sum_{[\omega_{r_1}|\cdots|\omega_{r_n}] \in \cO(\tensor*[_x]{\Pi}{^\DR_y})} [\cO(\tensor*[_x]{\Pi}{^\mot_y}),[\omega_{r_1}|\cdots|\omega_{r_n}], \dch^\vee]^\m\ee_{r_1}\cdots \ee_{r_n},
\end{equation}
where 
$$
\per [\cO(\tensor*[_x]{\Pi}{^\mot_y}),[\omega_{r_1}|\cdots|\omega_{r_n}], \dch^\vee]^\m = \int_\dch \omega_{r_1} \cdots \omega_{r_n}. 
$$
To account for the limit MHS at $\vv_0$ and $-\vv_1$, the integrals above are regularized at $\vv_0$ if $r_1 = 0$ and regularized at $-\vv_1$ if $r_n \neq 0$ (see \S\ref{sec:reg}). By \eqref{eqn:Li}, each integral evaluates to an $N$-cyclotomic MZV (see \S\ref{sec:Nmzv}) \cite{furusho}. Thus, we shall refer to these unevaluated motivic periods as {\em motivic $N$-cyclotomic MZVs}. The right hand side of \eqref{eqn:drinfeld} is called the {\em motivic Drinfeld associator} and denoted $\Phi^\m_{01}$.\label{not:drin} It is group-like in $\cP^\m_N\ll \ee_0,\ee_\zeta\rr$ \cite{drinfeld}. The image of the straight line path from $\ww_\eta$ to $-\vv_\eta$ under $c^{B,\DR}$ is obtained by the change of variables $w \mapsto w\eta$. This is simply the image of $\Phi_{01}^\m$ under the action $\ee_r \mapsto \ee_{r\eta}$ for all $r \in \{0\} \cup \bmu_N$. We denote this associator by $\Phi_{0\eta}^\m$. Its periods are also the motivic $N$-cyclotomic MZVs.



\section{The KZ and KZB local systems}

Suppose $X$ is a Riemann surface. Recall the notation $\piun(X,x)$ for the unipotent completion of $\pi_1(X,x)$ and $\p(X,x)$ for its Lie algebra.  

\subsection{The KZ local system}

Denote by $\bP_N^{\KZ,\topo}$ the $\Q$-local system over $U_N$ whose fiber over $w \in U_N$ is $\p(U_N,w)$. The completed group ring $\Q \pi_1(U_N,w)^\wedge$ is canonically isomorphic to the universal enveloping algebra of the completed free Lie algebra $\p(U_N,w)$. Thus, the discussion in \S\ref{sec:piMHS} gives $\bP_N^{\KZ,\topo}$ the structure of an admissible VMHS over $U_N$. After tensoring with $\C$, each fiber is isomorphic to the completed free Lie algebra \label{not:KZ}
\begin{align*}
\p_N^\KZ :&= \bL(\ee_0,\ee_\infty,\ee_\zeta \mid \zeta \in \bmu_N)^\wedge \big/ \left(\ee_0 + \ee_\infty + \sum_\zeta \ee_\zeta = 0\right) \cr
&\cong \bL(\ee_0,\ee_\zeta \mid \zeta \in \bmu_N)^\wedge,
\end{align*}
Define $\bP^\KZ_N$ to be the trivial bundle $\p^\KZ_N \times U_N$ over $U_N$ with flat connection
\begin{equation}
\label{eqn:KZ}
\nabla_{\KZ_N} = d + \Omega_{\KZ_N} = d + \ee_0\frac{dw}{w} + \sum_{\zeta \in \bmu_N} \ee_\zeta \frac{dw}{w - \zeta},
\end{equation}
where $\ee_0$ and $\ee_\zeta$ act on fibers $\p_N^\KZ$ via left adjoint. This is a well-known cyclotomic generalization of the Knizhnik--Zamolodchikov (KZ) connection \cite{KZe,gonch:poly}. The variation $\bP^\KZ_N$ is a special case of \eqref{eqn:pvar}. In particular, it is a unipotent admissible VMHS and is isomorphic as a flat vector bundle to $\bP^{\KZ,\topo}_N$.

\begin{prop}
\label{prop:pobj}
At each tangential base point $\vec{v} \in \{\ww_\zeta,\pm\vv_\zeta \mid \zeta \in \bmu_N\}$, the limit MHS of $\bP^\KZ_N$ is the Hodge realization of a pro-object of $\MTM_N$.
\end{prop}
\begin{proof}
The Lie algebra $\p(U_N,\vec{v})$ is canonically isomorphic to the Zariski tangent space of $\tensor*[_{\vec{v}}]{\Pi}{^\DR_{\vec{v}}}$ and thus dual to a subquotient of $\cO(\tensor*[_{\vec{v}}]{\Pi}{^\DR_{\vec{v}}})$. The result follows from Theorem \ref{thm:DG}.
\end{proof}
\begin{cor}
The {\em inverse} transport of the KZ connection along the straight line path $\dch$ from $\p(U_N,\ww_1)$ to $\p(U_N,-\vv_1)$ is the {\em left} adjoint of the motivic Drinfeld associator $\Phi^\m_{01}$.
\end{cor}
\begin{proof}
By the discussion at the end of \S\ref{sec:piMHS}, transport of $\bP^\KZ_N$ from -$\vv_1$ to $\ww_1$ along the inverse of $\dch$ is left adjoint by 
$$
\Theta_{\ww_1,-\vv_1}^\KZ = T(\dch)^{-1} = c^{B,\DR}(\dch) = \Phi_{01}^\m.
$$
\end{proof}
\begin{remark}
\label{rem:eQ}
As dual elements of the algebraic de~Rham cohomology $H^1_\DR(U_N)$, the basis elements $\ee_\zeta$ are defined over $\Q(\bmu_N)$. However, by the argument in Example \ref{ex:wQ}, they form a $\Q$-basis of $\omega(\p(U_N,\vv_1)) = \p(U_N,\vv_1)^\DR$.
\end{remark}




\subsection{The KZB local system}
\label{sec:kzb}

The KZB local system is a genus 1 analogue of the KZ local system. It is defined over a universal elliptic curve, which we will briefly describe now. For any congruence subgroup $\G \subset \SL_2(\Z)$ of level $N$, define the modular curve $Y_\G := \G \bbs \h$.\label{not:YG} Define the universal elliptic curve $\E_\G \to Y_\G$\label{not:EG} by the quotient $\E_\G = (\G \ltimes \Z^2) \bbs (\h \times \C)$. The fiber of $\E_\G$ above $[\tau] \in Y_\G$ is the elliptic curve $E_\tau := \C/\Lambda_\tau$, where $\Lambda_\tau := \Z \oplus \tau\Z$. Both $Y_\G$ and $\E_\G$ will be regarded as complex analytic varieties (or orbifolds when $N < 3$). The modular curve $Y_\G$ has natural compactification $X_\G$\label{not:XG} by gluing in disks at each of the cusps in $\G \bs \P^1(\Q)$. The universal curve $\E_\G \to Y_\G$ has a compactification $\overline{\E}_\G \to X_\G$ by gluing in nodal cubic curves and polygons on $\P^1$'s above each cusp of $X_\G$. 

Let $\E_\h$ denote the pullback of $\E_\G$ to $\h$. Denote the set of $N$-torsion sections of $\E_\G$ and $\E_\h$ by $\E_\G[N]$\label{not:EGN} and $\E_\h[N]$, respectively. If $\alpha \in \E_\G[N]$, it pulls back to a unique section $\tilde{\alpha} \in \E_\h[N]$. The lift $\tilde{\alpha}$ has coordinates $(x_\alpha,y_\alpha) \in (N^{-1}\Z/\Z)^2$ such that 
\begin{equation}
\label{eqn:lift}
    \tilde{\alpha} : \tau \longmapsto x_\alpha+y_\alpha\tau \bmod \Lambda_\tau.
\end{equation}
This defines a group isomorphism $\E_\h[N] \to (N^{-1}\Z/\Z)^2$, and the natural right action of $\SL_2(\Z)$ on $(\Z/N\Z)^2$ induces a right action on $\E_\h[N]$. Then 
$$
\E_\G[N] = \E_\h[N]^\G \cong \left((N^{-1}\Z/\Z)^2\right)^\G.
$$
The closure of a section $\alpha \in \E_\G[N]$ intersects the $\P^1$ and $E_0$ components of the singular fibers of $\overline{\E}_\G$ at an $N$th root of unity. 

For the purposes of this paper, we will focus on the congruence subgroup 
$$
\G_1(N) = \left\{ \begin{pmatrix} a & b \cr c & d \end{pmatrix} \equiv \begin{pmatrix} 1 & \ast \cr 0 & 1 \end{pmatrix} \bmod N \right\} \subset \SL_2(\Z). 
$$
The $N$-torsion sections of $\E_{\G_1(N)}$ form a group isomorphic to $\bmu_N$ since the $N$ sections of $\E_{\G_1(N)}[N]$ intersect the nodal cubic $E_0$ over $\tau = i \infty$ at the roots of unity $\bmu_N \subset \Gm \subset E_0$. Set $\E_{\G_1(N)}' := \E_{\G_1(N)} - \E_{\G_1(N)}[N]$\label{not:EGminus}. The fiber of $\E_{\G_1(N)}'$ above $[\tau]$ is $E_\tau' := E_\tau - \langle 1/N \rangle$. 

The $\G_1(N)$ KZB local system $\bP_{\G_1(N)}^\topo$ is defined over $\E_{\G_1(N)}'$, where the fiber above $[\tau,x]$ is $\p(E_\tau',x)$. After tensoring with $\C$, each fiber is non-canonically isomorphic to the completed free Lie algebra\label{not:pN}\label{not:XY}
$$
\p_N := \bL\left(\bX,\bY,\bt_\alpha \mid \zeta \in \bmu_N \right)^\wedge \big/ \left(\sum_{\alpha} \bt_\alpha = [\bX,\bY]\right). 
$$
The $\G_1(N)$ KZB connection $\nabla_{\G_1(N)}^\KZB = d + \Omega_{\G_1(N)}^\KZB$, where 
$$
\Omega_{\G_1(N)}^\KZB \in \Omega^1(\h \times \C, \log \Lambda_\tau) \otimes \Der \p_N,
$$
is a flat meromorphic connection on the trivial bundle $\p_N \times \h \times \C \to \h \times \C$ that is invariant with respect to $\G_1(N) \ltimes \Z^2$. The full formula of the connection form can be found in \cite[\S7,\S12]{hopper}, but for the purposes of this paper we only need its formula in a neighborhood of the singular fiber $E_0$ of $\E_{\G_1(N)}$ above $\tau = i\infty$. In this neighborhood, the connection form simplifies to
\begin{multline}
\label{eqn:KZBrest}
    \Omega_{\G_1(N)}' = \left(\bY \frac{\partial }{\partial\bX} + \frac12 \sum_{\substack{\zeta \in \bmu_N \\ m \ge 0}} \frac{B_{2m+2}}{(2m)!(2m+2)} \e_{2m+2,\zeta}\right) \frac{dq}{q} \cr
    + \frac{\bX}{e^\bX - 1} \cdot \bY \, \frac{dw}{w} + \sum_{\zeta \in \bmu_N} \bt_\zeta \, \frac{dw}{w - \zeta},
\end{multline}
where $B_{2m+2}$ are Bernoulli numbers, $q = e^{2\pi i \tau}$, $w = e^{2\pi i z}$ and $\e_{2m + 2,\zeta} \in \Der \p_N$. This formula is adapted from the bilevel KZB equations of Calaque and Gonzalez \cite{CG}, which are a generalization of the level 1 KZB connection first derived by Calaque, Enriquez, and Etinghof \cite{CEE} and Levin and Racinet \cite{LR}.

Since the connection on the trivial bundle $\p_N \times \h \times \C \to \h \times \C$ is invariant with respect to $\G_1(N) \ltimes \Z^2$, it descends to a meromorphic connection on the quotient bundle over $\E_{\G_(N)}'$, which we will call $\bP_{\G_1(N)}$.
\begin{theorem}[{\cite[Thm 9.2]{hopper}}]
Parallel transport of the KZB connection induces an isomorphism between $\bP_{\G_1(N)}$ and the flat vector bundle $\bP_{\G_1(N)}^\topo \otimes \cO_{\E_{\G_1(N)}'}$.
\end{theorem}

Furthermore, we may define Hodge, weight, and relative weight filtrations on the completed Lie algebra $\p_N$:
\begin{align}
\begin{split}
    \label{eqn:filtrations}
    F^{-p} &= \{x \in \p_N \mid \deg_\bY(x) + \sum\deg_{\bt_\zeta}(x) \leq p\} \cr
    W_{-m} &= \{x \in \p_N \mid \deg_\bY(x) + \deg_\bX(x) + 2 \sum\deg_{\bt_\zeta}(x) \geq m\} \cr 
    M_{-m} &= \{x \in \p_N \mid 2\deg_\bY(x) + 2\sum\deg_{\bt_\zeta}(x) \geq m\}.
\end{split}
\end{align}
The KZB connection takes values in $F^{-1}W_0 \Der \p_N$, and thus $F^\bullet$ and $W_\bullet$ filter the bundle $\bP_{\G_1(N)}$ with connection $\nabla^\KZB_{\G_1(N)}$. 
\begin{theorem}[{\cite[\S11]{hopper}}]
\label{thm:avmhs}
Together with these filtrations, $\mathbfcal{P}_{\G_1(N)}^\topo \to \E_{\G_1(N)}'$ is a pro-object of the category of admissible variations of MHS. The MHS on each fiber $\p(E',x)$ is the canonical MHS on $\p(E',x)$ in terms of \cite{hain:bowdoin,hain:drt}.
The relative weight filtration of the limit MHS at $\partial/\partial q + \partial/\partial w$ anchored at identity of the singular fiber $E_0$ (the point $(q = 0, w = 1)$) is $M_\bullet$.
\end{theorem}




\subsection{The Hain map}
\label{sec:hain}

The key observation we will exploit is that he KZB connection along the singular fiber $E_0$ of $\bP_{\G_1(N)}$ ``degenerates'' to the $N$-cyclotomic KZ connection. Hain \cite{hain:kzb} Any elliptic curve $E_\tau = \C/\Lambda_\tau$ may be written as $E_q := \C^\ast/q^\Z$ where $q = e^{2\pi i\tau}$. The usefulness of this observation and the resulting map $\p_N^\KZ \to \p_N$ were made in level 1 by Hain \cite{hain:kzb}. Following the same notation, we will generalize this map for the level $N$ setting. 

Set 
$$
\cA_{|q|} := \{w \in \C^\ast \mid |q|^{1/2} \leq |w| \leq |q|^{-1/2}\}.
$$
Denote the inner and outer boundaries of this annulus by $\partial_-\cA_{|q|}$ and $\partial_+\cA_{|q|}$, respectively. Then $E_q$ is the quotient of $\cA_{|q|}$ with $w \in \partial_+\cA_{|q|}$ identified with $qw \in \partial_-\cA_{|q|}$. The real oriented blow-up $\Bl_{\bmu_N}^oE_q$ is a similar quotient of $\Bl_{\bmu_N}^o \cA_{|q|}$ (see Figure \ref{fig:E}).

\begin{figure}
    \centering
    \begin{tikzpicture}[scale=1.3]
        \draw[very thick,fill=gray!20] (0, 0) circle (2cm);
        \draw[very thick,fill=white] (0, 0) circle (.25cm);
        \draw[fill=white] (1, 0) circle (.15cm);
        \draw[fill=white] (.309, .951) circle (.15cm);
        \draw[fill=white] (-.809, .588) circle (.15cm);
        \draw[fill=white] (-.809, -.588) circle (.15cm);
        \draw[fill=white] (.309, -.951) circle (.15cm);
        
        \draw[red] plot [smooth cycle, tension=.7] coordinates {(1.15, 0) (1.4, .6) (1.202, 1.202) (0, 1.7) (-1.202, 1.202) (-1.7, 0) (-1.202, -1.202) (0, -1.7) (1.202, -1.202) (1.4, -.6)};
        \draw[->, red] (0, 1.7) arc (90:91:1.7);
        
        \draw[blue] (1.15, 0) .. controls (1, -.6) and (.7, 0) .. (.25, 0);
        \draw[blue] (1.15, 0) -- (2, 0);
        \draw[<-, blue] (1.6, 0) -- (2, 0);
        
        
        
        \draw[fill=black] (1.15, 0) circle (.05cm);
        \draw[fill=black] (-2, 0) circle (.05cm);
        \draw[fill=black] (-.25, 0) circle (.05cm);
       
        \draw (-2, 0) node[left] {$w$};
        \draw (-.2, 0) node[left] {$qw$};
        \draw (1.05, .1) node[above] {$\vv_1$};
        \draw (-1.2, 1.2) node[right] {{\color{red} $\alpha$}};
        \draw (1.6, 0) node[below] {{\color{blue} $\beta$}};
        
    \end{tikzpicture}
    \qquad
    \begin{tikzpicture}[scale=1.3]
        \draw[very thick,fill=gray!20] (0, 0) circle (2cm);
        \draw[very thick,fill=white] (0, 0) circle (.25cm);
        \draw[fill=white] (1, 0) circle (.15cm);
        \draw[fill=white] (.309, .951) circle (.15cm);
        \draw[fill=white] (-.809, .588) circle (.15cm);
        \draw[fill=white] (-.809, -.588) circle (.15cm);
        \draw[fill=white] (.309, -.951) circle (.15cm);
        
        
        
        \draw[black!60!green,->] plot [smooth] coordinates {(1.15, 0) (1, -.2) (.8, 0) (.559, .951)};
        \draw[black!60!green] (.559, .951) arc (0:225:.25);
        \draw[black!60!green] plot [smooth] coordinates {(.132, .774) (.9, -.2) (1.15, 0)};
        
        \draw[black!60!green,->] plot [smooth] coordinates {(1.15, 0) (.9, -.25) (-.1, .5) (-.986, -.411)};
        \draw[black!60!green] (-.986, -.411) arc (135:315:.25);
        \draw[black!60!green] plot [smooth] coordinates {(-.632, -.765) (-.1, .4) (.9, -.3) (1.15, 0)};
        
        \draw[fill=black] (1.15, 0) circle (.05cm);
       
        \draw (1.05, .1) node[above] {$\vv_1$};
        \draw (.559, .951) node[right] {{\color{black!60!green} $\gamma_\zeta$}};
        
    \end{tikzpicture}
    \caption{$\Bl_{\bmu_N}^oE_q$ seen as quotient of $\Bl_{\bmu_N}^o \cA_{|q|}$ with generators of $\pi_1(E_q - \bmu_N,\vv_1)$}
    \label{fig:E}
\end{figure}

Roughly speaking, as $q \to 0$, the space $\cA_{|q|}$ converges to $\Bl_{\{0,\infty\}}^o\P^1$ and $\Bl_{\bmu_N}^o\cA_{|q|}$ converges to $\Bl_{\{0,\bmu_N,\infty\}}^o\P^1$. Thus, as $R \to 0$, the elliptic curve $E_{Re^{i\theta}}$ converges to the quotient of $\Bl_{\{0,\infty\}}^o\P^1$ with the boundary circles at 0 and $\infty$ identified by multiplication by $e^{i\theta}$. We denote this space by $E_{e^{i\theta}\partial/\partial q}$.\label{not:Eu} It is the ``first order smoothing'' of the nodal elliptic curve in the direction of $e^{i\theta}\partial/\partial q$. The scheme $U_N$ includes into $E_{e^{i\theta}\partial/\partial q}$ in the obvious way. 

Fix $\uu = \partial/\partial q$ at the origin of the $q$-disk.\label{not:u} We observe that the inclusion $U_N \hookrightarrow E_\uu$ is compatible with the $\KZ_N$ and $\G_1(N)$ KZB connections. That is, upon a change of basis, the linearization of the KZB bundle \eqref{eqn:KZBrest} over $E_\uu$ ``pulls back" to the $\KZ_N$ bundle $\bP^\KZ_N$ over $U_N$. 

\begin{prop} 
\label{prop:hain}
Let $E_\uu' := E_\uu - \bmu_N$.\label{not:Eup} The inclusion $U_N \hookrightarrow E_\uu'$ induces the commutative diagram
\begin{equation}
\label{eqn:haindiag}
\xymatrix{\p(U_N, \vv_1) \ar[r] \ar[d]^{\KZ_N} & \p(E_\uu', \vv_1) \ar[d]^{\KZB_{\G_1(N)}} \\ \p_N^\KZ \ar[r]^{\Psi_N} & \p_N.}
\end{equation}
The vertical maps are the comparison isomorphisms of the canonical MHS on $\p(U_N,\vv_1)$ and $\p(E_\uu',\vv_1)$ induced by transport of the the $\KZ_N$ and $\G_1(N)$ KZB connections, respectively. The lower horizontal map $\Psi_N$ is the level $N$ {\em ``Hain map"}
\begin{equation}
    \label{eqn:hain}
    \Psi_N : \left\{ 
    \begin{array}{lll}
        \ee_0 & \longmapsto & \frac{\bX}{e^{\bX} - 1} \cdot \bY \\
        \ee_\infty & \longmapsto & \frac{\bX}{e^{-\bX} - 1}\cdot \bY \\
        \ee_\zeta & \longmapsto & \bt_\zeta.
    \end{array}
    \right.
\end{equation}
Furthermore, $\Psi_N$ is a morphism of MHS that preserves relative weight filtrations.
\end{prop}


\begin{proof}
Recall the restriction of $\G_1(N)$ $\KZB$ to a first order neighborhood of the nodal cubic \eqref{eqn:KZBrest}. Pulling back along the inclusion $U_N \hookrightarrow E_\uu'$ yields the connection
$$
d + \frac{\bX}{e^\bX - 1} \cdot \bY \, \frac{dw}{w} + \sum_{\zeta \in \bmu_N} \bt_\zeta \, \frac{dw}{w - \zeta}.
$$
This is precisely the image of the cyclotomic KZ connection \eqref{eqn:KZ} under the given formula \eqref{eqn:hain} for $\Psi_N$.
 
For the last claim, recall each $\ee_r$ is of type $(-1,-1)$ and $M_\bullet = W_\bullet$ on $\p^\KZ_N$. It follows that $\Psi_N$ is compatible with the filtrations on $\p_N$ defined in \S\ref{sec:kzb}. It remains to show that $\Psi_N$ preserves rational structures. Since the Hain map is induced by the relationship between $\KZ_N$ and $\G_1(N)$ KZB, this problem is equivalent to showing the MHS on $\p(E_\uu',\vv_1)$ induced by the upper horizontal map in \eqref{eqn:haindiag} is the same as the canonical MHS induced by the $\G_1(N)$ KZB connection. This follows from Theorem \ref{thm:hainCan}. 
\end{proof}

\section{The motivic limit MHS of \texorpdfstring{$\bP_{\G_1(N)}$}{TEXT}}
\label{sec:pEv}


The Lie algebra $\p(E_\uu',\vv_1)$ is endowed with the limit MHS of the elliptic KZB variation on $\bP_{\G_1(N)}$ at $\partial/\partial q + \partial/\partial w$ above the cusp $q = 0$. We now argue this MHS, like $\p(U_N,\v_1)$, is in fact the Hodge realization of a pro-object of $\MTM_N$. This will build on the results about the cyclotomic KZ connection from the previous section.

Define loops $\alpha$ and $\beta$ in $\pi_1(E_q - \bmu_N,\vv_1)$ to be the standard generators of the fundamental group of a torus (see Figure \ref{fig:E}). 
For each $\zeta \in \bmu_N$, there is a generator $\gamma_\zeta \in \pi_1(E_q-\bmu_N,\vv_1)$ chosen to be the homotopy class of the loop traveling clockwise half way around the boundary circle at 1, counter clockwise around the inner boundary circle $|w| = |q|^{1/2}$ (as necessary) and then counterclockwise around the boundary circle at $\zeta$ (see green loops in Figure \ref{fig:E}). Dividing $\beta$ into two paths, one can easily verify the expected relation
$$
\alpha\beta\alpha^{-1}\beta^{-1} = \gamma_1\gamma_\zeta \cdots \gamma_{\zeta^{N-1}},
$$
where here $\zeta = e^{2\pi i/N}$. 

Letting $q \to 0$ in the direction of $\uu = \partial/\partial q$, the homotopy classes of $\alpha$, $\beta$, and $\gamma_\zeta$ for $\zeta \in \bmu_N$ converge to generators of the topological fundamental group $\pi_1(E_\uu',\vv_1)$. We will construct a MHS on $\p(E_\uu',\vv_1)$ by transporting the $\KZ_N$ connection form along these generators of $\pi_1(E_\uu,\vv_1)$ as in \S\ref{sec:piMHS} and then applying the Hain map. 

Recall from \S\ref{sec:drinfeld} the regularized transport of $\KZ_N$ from $\ww_\eta \in T_0U_N$ to $-\vv_\eta \in T_\eta U_N$ is the motivic Drinfeld associator $\Phi_{0\eta}^\m \in \cP_N^\m\ll\ee_0,\ee_\zeta\mid\zeta \in \bmu_N\rr$. The involution $w \mapsto 1/w$ allows us to define a motivic transport from $\vv_1$ to $\infty$, denoted $\Phi_{1\infty}^\m$ by exchanging $\ee_0$ for $\ee_1$, $\ee_1$ for $\ee_\infty = -\ee_0 - \ee_1 - \sum_\zeta \ee_\zeta$, and $\ee_\zeta$ for $\ee_\zetabar$ in the formula of $\Phi_{01}^\m$ \eqref{eqn:drinfeld}. 

For the remainder of the section, denote by $A_{E_\uu'}$ the group algebra 
$$
A_{E_\uu'} := \Q\ll\bX,\bY,\bt_\zeta\rr/\Bigg(\sum_{\zeta \in \bmu_N} \bt_\zeta = [\bX,\bY]\Bigg).
$$
Following the paths in Figure \ref{fig:E}, {\em inverse} monodromy 
\begin{equation}
    \label{eqn:Pmono}
    \Theta : \pi_1(E_\uu',\vv_1) \to \cP^\m_N \otimes A_{E_\uu'}
\end{equation}
is given by the formulas
\begin{align*}
    \alpha & \longmapsto \Phi_{\infty 1}^\m e^{\Le \ee_\infty} \Phi_{1\infty}^\m \cr
    \beta & \longmapsto e^{\Le \ee_1/2} \Phi_{01}^\m e^{-\bX} \Phi_{1\infty}^\m \cr
    \gamma_\zeta & \longmapsto e^{\Le \ee_1/2} \Phi_{01}^\m e^{k\Le \ee_0/N}\Phi_{\zeta 0}^\m e^{-\Le \ee_\zeta}\Phi_{0\zeta }^\m e^{-k\Le  \ee_0/N}\Phi_{10}^\m e^{-\Le \ee_1/2},
\end{align*}
where $\zeta = e^{2\pi ik/N}$ for $k \in \{0,1,\ldots,N-1\}$ and the motivic Drinfeld associators $\Phi^\m$ actually indicate their image under the Hain map. The $e^{-\bX}$ in the formula for $\beta$ corresponds to the factor of automorphy of the KZB local system associated with identifying the inner and outer boundaries of the annulus (see \cite[\S6.2]{hopper} and Figure \ref{fig:kzE}). Note that $\Theta(\gamma_1)$ simplifies to $e^{-\Le\ee_1}$. 

\begin{figure}[!ht]
    \centering
    \begin{tikzpicture}[scale=1.3]
        \draw[red, thick,fill=gray!20] (0, 0) circle (2cm);
        \draw[very thick,fill=white] (0, 0) circle (.25cm);
        \draw[fill=white] (1, 0) circle (.15cm);
        \draw[fill=white] (.309, .951) circle (.15cm);
        \draw[fill=white] (-.809, .588) circle (.15cm);
        \draw[fill=white] (-.809, -.588) circle (.15cm);
        \draw[fill=white] (.309, -.951) circle (.15cm);
       
        \draw[thick, red, ->-=.4, -<-=.7] (1.15, 0) -- (2, 0);
        \draw[<-, thick, red] (-1.414, -1.414) arc (225:230:2);
        \draw[<-, thick, red] (-1.414, 1.414) arc (135:180:2);
        \draw[thick, blue, -<-=.55] (.85, 0) -- (.25, 0);
        \draw[thick, blue, ->-=.7] (.85, 0) arc (180:360:.15);
        \draw[thick, blue, ->-=.6] plot [smooth, tension=1.2] coordinates {(1.15, 0) (1.625,-.2) (2, 0)};
        \draw[thick, blue, dashed, -<-=.3, -<-=.7] plot [smooth, tension=1.2] coordinates {(.25, 0) (1.225, -1.9) (2, 0)};
        
        \draw (-2, 0) node[right] {$e^{\Le \ee_\infty}$};
        \draw (1, -.15) node[below] {$e^{\Le \ee_1/2}$};
        \draw (.55, 0) node[above] {$\Phi_{01}$};
        \draw (1.575, 0) node[above] {$\Phi_{1\infty}^{\pm 1}$};
        \draw (1.85,-1.75) node {$e^{-\bX}$};
        
        \draw[fill=black] (1.15, 0) circle (.05cm);

    \end{tikzpicture}
    \qquad
    \begin{tikzpicture}[scale=1.3]
        \draw[black, thick,fill=gray!20] (0, 0) circle (2cm);
        \draw[fill=white] (0, 0) circle (.25cm);
        \draw[fill=white] (1, 0) circle (.15cm);
        \draw[fill=white] (.309, .951) circle (.15cm);
        \draw[fill=white] (-.809, .588) circle (.15cm);
        \draw[thick, black!60!green, fill=white] (-.809, -.588) circle (.15cm);
        \draw[fill=white] (.309, -.951) circle (.15cm);
        
        \draw[thick, black!60!green, ->-=.4, -<-=.7] (.85, 0) -- (.25, 0);
        \draw[thick, black!60!green, ->-=.4, -<-=.7] (.85, 0) arc (180:360:.15);
        \draw[thick, black!60!green, ->-=.4, -<-=.7] (.25, 0) arc (0:216:.25);
        \draw[thick, black!60!green, ->-=.4, -<-=.7] (-.202, -.147) -- (-.688, -.5);
        \draw[thick, black!60!green, <-] (-.930, -.676) arc (216:270:.15);
        
        \draw (1, -.15) node[below] {$e^{\mp\Le \ee_1/2}$};
        \draw (.55, 0) node[above] {$\Phi_{10}^{\pm 1}$};
        \draw (-.445, -.323) node[above left] {$\Phi_{0\zeta}^{\pm 1}$};
        \draw[->] plot [smooth, tension=1.2] coordinates {(-.2, 1.1) (-.25, .7) (-.2, .35)}; 
        \draw (-.2, 1.25) node {$e^{\pm k\Le\ee_0/N}$};
        \draw (-.809, -.738) node[below] {$e^{-\Le\ee_\zeta}$};

        \draw[fill=black] (1.15, 0) circle (.05cm);

    \end{tikzpicture}
    
    \caption{Inverse transport operators for $\alpha$, $\beta$, and $\gamma_\zeta$}
    \label{fig:kzE}
\end{figure}
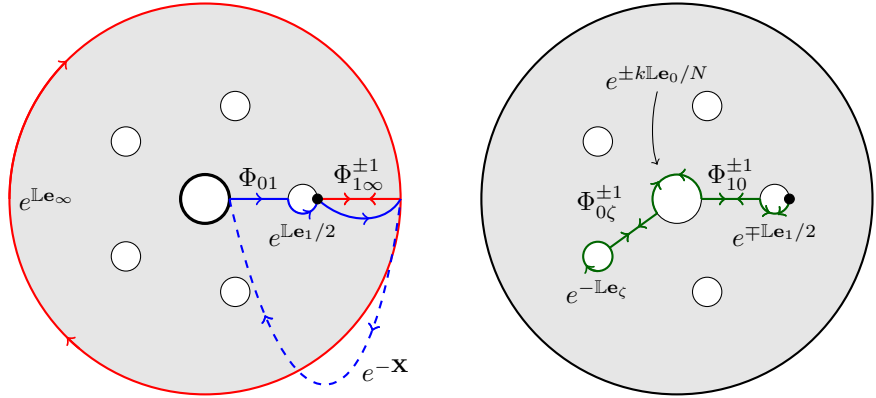

The associators $\Phi^\m$ satisfy relations compatible with the topology of $E_\uu$. The loop $\alpha$ in Figure \ref{fig:E} is homotopy equivalent to the loop $\alpha'$ based at $\vv_1$ in Figure \ref{fig:star}. Set $\xi = e^{2\pi i/N}$. Then
\begin{multline*}
    \Theta(\alpha') = e^{-\Le \ee_1/2} \Phi_{01}^\m e^{-\Le\ee_0/N} (\Phi_{0\xi}^\m)^{-1}   e^{-\Le \ee_\xi}\Phi_{0\xi}^\m\cdots \cr
    \cdots e^{-\Le \ee_{\xi^{N-1}}}\Phi_{0\xi^{N-1}}^\m e^{-\Le\ee_0/N}(\Phi_{01}^\m)^{-1} e^{- \Le\ee_1/2}.
\end{multline*}
Since the $\KZ_N$ connection is flat, we know $\per \Theta(\alpha) = \per \Theta(\alpha')\in \C \otimes A_{E_\uu'}$. We shall call this the ``star relation.'' It is a cyclotomic version of Drinfeld's ``hexagon relation'' \cite{drinfeld}.

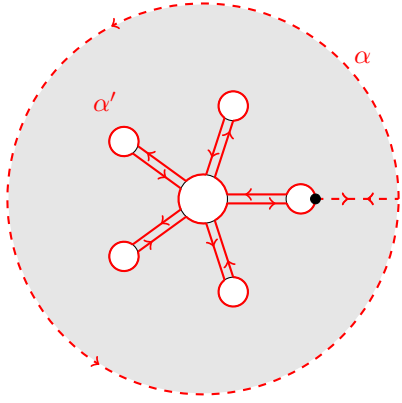
\begin{figure}[!ht]
\centering
    \begin{tikzpicture}[scale=1.3]
        \draw[white, thick,fill=gray!20] (0, 0) circle (2cm);
        \draw[fill=white] (0, 0) circle (.25cm);
        \draw[fill=white] (1, 0) circle (.15cm);
        \draw[fill=white] (.309, .951) circle (.15cm);
        \draw[fill=white] (-.809, .588) circle (.15cm);
        \draw[fill=white] (-.809, -.588) circle (.15cm);
        \draw[fill=white] (.309, -.951) circle (.15cm);
        
        \draw[thick, red] (.857, -.046) arc (198:360+162:.15);
        \draw[thick, red, -<-=.3] (.857, -.046) -- (.246, -.046);
        \draw[thick, red, ->-=.7] (.857, .046) -- (.246, .046);
        
        \draw[thick, red] (.309, .801) arc (270:360+270-36:.15);
        \draw[thick, red, -<-=.3] (.309, .801) -- (.12, .219);
        \draw[thick, red, ->-=.7] (.221, .83) -- (.032, .248);
        
        \draw[thick, red] (-.666, .541) arc (360-18:720-54:.15);
        \draw[thick, red, -<-=.3] (-.666, .541) -- (-.172, .182);
        \draw[thick, red, ->-=.7] (-.72, .467) -- (-.226, .107);
        
        \draw[thick, red] (-.72, -.467) arc (54:360+18:.15);
        \draw[thick, red, ->-=.3] (-.666, -.541) -- (-.172, -.182);
        \draw[thick, red, -<-=.7] (-.72, -.467) -- (-.226, -.107);
        
        \draw[thick, red] (.221, -.83) arc (90+36:450:.15);
        \draw[thick, red, ->-=.3] (.309, -.801) -- (.12, -.219);
        \draw[thick, red, -<-=.7] (.221, -.83) -- (.032, -.248);
        
        \draw[thick, red] (.246, .046) arc (10.26:72-10.26:.25);
        \draw[thick, red] (.032, .248) arc (72+10.26:144-10.26:.25);
        \draw[thick, red] (-.226, .107) arc (144+10.26:216-10.26:.25);
        \draw[thick, red] (-.172, -.182) arc (216+10.26:288-10.26:.25);
        \draw[thick, red] (.12, -.219) arc (288+10.26:360-10.26:.25);
        
        \draw[dashed, thick, red, ->-=.4, -<-=.7] (1.15, 0) -- (2, 0);
        \draw[dashed, ->, thick, red] (-1, 1.732) arc (120:238:2);
        \draw[dashed, ->, thick, red] (2, 0) arc (0:118:2);
        \draw[dashed, thick, red] (-1, -1.732) arc (240:360:2);
        
        \draw[fill=black] (1.15, 0) circle (.05cm);
        \draw (-1, 1) node {{\color{red} $\alpha'$}};
        \draw (1.45, 1.45) node[above,right] {{\color{red} $\alpha$}};
        
    \end{tikzpicture}
\caption{The star relation}
\label{fig:star}
\end{figure}

\begin{lemma}
The star relation is motivic, that is $\Theta(\alpha) =\Theta(\alpha')$ in $\cP^\m_N \otimes A_{E_\uu'}$.
\end{lemma}
\begin{proof}
The motivic comparison isomorphism $c_\m$ \eqref{eqn:cVm} is natural with respect to morphisms in $\MTM_N$. Thus, by Theorem \ref{thm:DG}, we have the commutative diagram
$$
\xymatrix{\pi_1(U_N,\vv_1) \times \pi_1(U_N,\vv_1) \ar[r] \ar[d]_{\Theta} & \pi_1(U_N,\vv_1) \ar[d]^{\Theta} \cr \cP_N^\m\ll\ee_0,\ee_\zeta\rr \otimes \cP_N^\m\ll\ee_0,\ee_\zeta\rr \ar[r] & \cP_N^\m\ll\ee_0,\ee_\zeta\rr,}
$$
where the top horizontal map is composition of loops and the bottom horizontal map is multiplication. Since $\alpha^{-1}\alpha'$ is homotopic to the constant path, the image of $\alpha^{-1}\alpha'$ in the bottom right must be 1. Hence, $\Theta(\alpha) = \Theta(\alpha')$.
\end{proof}
There is also a motivic ``cylinder relation'' involving the factor of automorphy.
\begin{lemma}
The group algebra $\cP^\m_N \otimes A_{E_\uu'}$ has the relation
$$
e^\bX e^{\Le\ee_0} e^{-\bX}e^{\Le \ee_\infty} = 1.
$$
\end{lemma}
\begin{proof}
Apply the factor of automorphy $e^{-\bX}$ to identify the circle at $\infty$ to the circle at zero to observe 
$$
e^\bX \ee_0 e^{-\bX} = -\ee_\infty.
$$ 
Then multiply both sides by $\Le$ and exponentiate. Alternatively, see \cite[Lemma 18.1]{hain:kzb} for a computational proof.
\end{proof}

\begin{cor}
Inverse monodromy $\Theta : \p(E_{\uu}',\vv_1) \to \cP^\m_N \otimes A_{E_\uu'}$ is well-defined. 
\end{cor}
\begin{proof}
Apply the previous lemma to confirm 
$$
\Theta(\alpha\beta\alpha^{-1}\beta^{-1}) = \Theta(\gamma_1\gamma_{\xi}\gamma_{\xi^2} \cdots \gamma_{\xi^{N-1}}).
$$
\end{proof}
\begin{theorem}
\label{thm:obj}
The generalized Hodge realization 
$$
\left(\p(E_\uu',\vv_1), A_{E_\uu'}, \Theta\right) \in \Ob(\cH)
$$
is the Hodge realization of a pro-object of $\MTM_N$, and the Hain map $\Psi_N$ is motivic.
\end{theorem}
\begin{proof}
By Proposition \ref{prop:pobj}, $\p(U_N,\vv_1)$ is an object of $\MTM_N$. The Hodge realization 
$$
\left(\p(U_N,\vv_1),\Q\ll\ee_0,\ee_\zeta\mid \zeta \in \bmu_N\rr,\Theta_{\vv_1,\vv_1}^\KZ\right)
$$
is an object of Brown's category $\cH$. The action of $\pi_1(\cH,\omega_\cH^\betti)$ on $\p(U_N,\vv_1)$ factors through $\cG^\betti_N = \pi_1(\MTM_N,\omega^\betti)$. 
The Hain map
$$
\Psi_N : \left(\p(U_N,\vv_1),\Q\ll\ee_0,\ee_\zeta\mid \zeta \in \bmu_N\rr,\Theta_{\vv_1,\vv_1}^\KZ\right) \longmapsto \left(\p(E_\uu',\vv_1), A_{E_\uu'}, \Theta\right)
$$
is a morphism in $\cH$ and thus equivariant with respect to the action of $\pi_1(\cH,\omega^\betti_\cH)$. This is summarized in the following commutative diagram, where the horizontal maps are induced by $\Psi_N$, the bending vertical maps are the $\pi_1(\cH,\omega^\betti)$-action, and upper straight vertical maps are induced by the Hodge realization functor $\MTM_N \to \cH$.
$$
\xymatrix{\pi_1(\cH, \omega^\betti) \times \p(U_N,\vv_1) \ar[d] \ar[r]^{\Psi_N} \ar@/_4pc/[dd] & \pi_1(\cH, \omega^\betti) \times \p(E_\uu',\vv_1) \ar[d] \ar@/^4pc/[dd] \cr 
\cG^\betti \times \p(U_N,\vv_1) \ar[d] \ar[r]^{\Psi_N} & \cG^\betti \times \p(E_\uu',\vv_1) \ar@{.>}[d] \cr 
\p(U_N,\vv_1) \ar[r]^{\Psi_N} & \p(E_\uu',\vv_1)}
$$
To prove the claim, it is sufficient to establish the existence of the dotted bottom right vertical action $\cG^\betti_N \times \p(E_\uu',\vv_1) \to \p(E_\uu',\vv_1)$ such that the diagram commutes. Since the Hain map is injective, this action can easily be defined on the image of $\Psi_N$. To see the action of $\pi_1(\cH,\omega^\betti_\cH)$ on the entirety of $\p(E_\uu',\vv_1)$ factors through $\cG_N^\betti$, it remains to show $\cG_N^\betti$ respects the cylinder relation. However, we know this is the case since $\ee_0$ and $\ee_\infty$ each span copies of $\Q(1)$ and $\bX$ spans a copy of $\Q(0)$ in $\p(E_\uu',\vv_1)$. Since $\cG_N^\betti$ acts trivially on $\Q(0)$, the action respects the cylinder relation.  Thus, there is a well-defined $\cG^\betti_N$-action on $\p(E_\uu',\ww_1)$. Hence, $\p(E_\uu',\ww_1)$ is the Betti realization of an object of $\MTM_N$. 
\end{proof}


\section{Depth filtrations and polylogarithm quotients}
\label{sec:depth}

Theorem \ref{thm:obj} establishes the existence of representations $\phi_\cyc$ and $\phi_\elp$ of $\cG^\DR_N$ that fit into a commutative diagram
$$
\xymatrix{\cG^\DR_N \ar[r]^-{\phi_\cyc} \ar[rd]_-{\phi_\elp} & \Aut \p(U_N,\vv_1)^\DR \ar[d]^{\Psi_N} \cr & \Aut \p(E_\uu',\vv_1)^\DR,}
$$
where the vertical map is induced by the Hain map. Restricting to the unipotent radical $\mathcal{K}_N^\DR$\label{not:KN} of $\cG_N^\DR$ \eqref{eqn:Gses} induces representations of the motivic Lie algebra $\k_N^\DR$.\label{not:kN}
\begin{equation}
\label{eqn:k}
\xymatrix{\k_N^\DR \ar[r]^-{\phi_\cyc} \ar[rd]_-{\phi_\elp} & \Der \p(U_N,\vv_1)^\DR \ar[d]^{\Psi_N} \cr & \Der \p_N(E_\uu',\vv_1)^\DR.}
\end{equation}
As discussed in the introduction, relations between motivic periods of $\MTM_N$ are closely related to relations in the associated depth graded of the Lie algebra $\k_N^\DR$. We now define this depth filtration. 

\subsection{Filtrations and compatibility}

Fix a (possibly tangential) base point $x$ of $U_N$. The inclusion $U_N \hookrightarrow \Gm$ induces a map $\r_\cyc : \p(U_N,x)^\bullet \to \p(\Gm,x)^\bullet$ for $\bullet \in \{B,\DR\}$. Define the {\em depth filtration} on $\p(U_N,x)^\bullet$ by 
$$
D^d \p(U_N,x)^\bullet = \begin{cases}
\p(U_N,x)^\bullet & d = 0 \cr
\ker\r_\cyc & d = 1 \cr
[D^1,D^{d-1}] & d > 1.
\end{cases}
$$
We may identify $\p(U_N,x)^\DR$ with $\bL( \ee_0,\ee_\zeta \mid \zeta \in \bmu_N)^\wedge$ and $\p(\Gm,x)^\DR$ with $\bL(\ee_0)^\wedge$ via the KZ connection \eqref{eqn:KZ} and \eqref{eqn:chen}. Then $\r(\ee_0) = \ee_0$ and $\r(\ee_\zeta) = 0$. Thus, $D^d\p(U_N,w)^\DR$ is spanned by the Lie words with degree at least $d$ in the $\ee_\zeta$ terms.

There is an analogous filtration in the elliptic case. 
Let $E_q = \Gm/q^\Z$ be a smooth elliptic curve and let $E_q' = E_q - \bmu_N$. Fix a (possibly tangential) base point $x \in E_q'$. The inclusion $E_q' \hookrightarrow E_q$ induces a map $\r_\elp : \p(E_q',x)^\bullet \to \p(E_q,x)^\bullet$.  The depth filtration on $\p(E_q',x)$ is then
$$
D^d \p(E_q',x)^\bullet = \begin{cases}
\p(E_q',x)^\bullet & d = 0 \cr
\ker\r_\elp & d = 1 \cr
[D^1,D^{d-1}] & d > 1.
\end{cases}
$$
Use the KZB connection to identify $\p(E_q',x)^\DR$ with 
$$
\bL(\bx,\by,\bt_\zeta)^\wedge/\left(\sum\bt_\zeta = [\bx,\by]\right)
$$
and $\p(E_q,x)^\DR$ with $\bL(\bx,\by)^\wedge$. Then $\r_\elp(\bx) = \bx$, $\r_\elp(\by) = \by$, and $\r_\elp(\bt_\zeta) = 0$.
Thus, $D^d\p(E_q',x)^\DR$ is spanned by the Lie words with degree at least $d$ in the $\bt_\zeta$ terms. 


\def\f{\mathfrak{f}}

\begin{lemma}
\label{lem:free}
Suppose $\varphi : \f \to \f'$ is a homomorphism of free Lie algebras. If the induced map $\varphi_\ast : H_1(\f) \to H_1(\f')$ is injective, then $\varphi$ is injective and strict with respect to the lower central series. That is,
$$
\varphi(L^j \f') = \varphi(\f)\cap L^j\f'.
$$ 
To prove the result, it suffices to show that $\Gr^\bdot_L \f \to \Gr^\bdot_L\f'$ is injective.
\end{lemma}

\begin{proof}
Since $\f$ and $\f'$ are free, there are canonical isomorphisms $\Gr^\bdot_L \f \cong \bL(H_1(\f))$ and $\Gr^\bdot_L \f' \cong \bL(H_1(\f'))$ \cite{serre}. We then have the commutative diagram
$$
\xymatrix{\Gr_L^\bullet \f \ar[r]^{\Gr_L^\bullet\varphi} \ar[d] & \Gr_L^\bullet \f' \ar[d] \cr \bL(H_1(\f)) \ar[r]^-{\varphi_\ast} & \bL(H_1(\f')).}
$$
The result follows immediately.
\end{proof}


\begin{prop}
The depth filtration on $\p(U_N,x)$ (resp. $\p(E_q',x))$ induces filtrations $D^\bullet$ on the Betti local system $\bP_N^{\KZ,\topo}$ (resp. $\bP_{\G_1(N)}^\topo$) and de~Rham holomorphic vector bundle $\bP_N^\KZ$ (resp. $\bP_{\Gamma_1(N)}$) with flat connection $\nabla_{\KZ_N}$ (resp. $\nabla_{\KZB_{\G_1(N)}}$). The subbundles $D^d\bP^{\KZ,\topo}_N$ and $D^d\bP^\KZ_N$ (resp. $D^d\bP_{\G_1(N)}^\topo$ and $D^d\bP_{\G_1(N)}$) are isomorphic as flat vector bundles. 
\end{prop}

\begin{proof}
The Lie algebras $\Gr^W_\bullet\p_N$ and $\Gr^W_\bullet\p_N^\KZ$ are free and thus by the Shirishov--Witt theorem so are $D^1\Gr^W_\bullet\p_N$ and $D^1\Gr^W_\bullet\p_N^\KZ$. Both the KZ and KZB connections take values in $D^0$ and thus transport along any path is injective on $\Gr^1_D\Gr^W_\bullet$. Since $D^d = L^d(D^1)$, the result follows from the lemma and taking inverse limits with respect to the weight filtration.  
\end{proof}

The choice of notation $D^\bullet$ is intentionally ambiguous since the filtrations are compatible with the respect to the Hain map.
\begin{lemma}
\label{lem:D1}
The Hain map induces an injection 
$$
H_1(\Gr^W_\bullet D^1 \p(U_N,\vv_1)^\DR) \to H_1(\Gr^W_\bullet D^1 \p(E_\uu',\vv_1)^\DR). 
$$
\end{lemma}
\begin{proof}
Observe
$$
H_1(\Gr^W_\bullet D^1\p(U_N,\vv_1)^\DR) = \bigoplus_{\substack{n \ge 0 \\ \zeta \in \bmu_N}} \Q\ee_0^n\cdot \ee_\zeta
$$
and 
$$
H_1(\Gr^W_\bullet D^1\p(E_\uu',\vv_1)^\DR) = \bigoplus_{\substack{n \ge 0 \\ \zeta \in \bmu_N}} S^m H \cdot \bt_\zeta,
$$
where $H = \Q\bX \oplus \Q\bY$. The Hain map sends $\Psi_N : \ee_0^n \cdot \ee_\zeta \mapsto \bY^n \cdot \bt_\zeta \bmod D^2\p(E_\uu',\vv_1)^\DR$. This is clearly injective.
\end{proof}
\begin{cor}
The Hain map is strict with respect to the depth filtrations.
\end{cor}
\begin{proof}
It follows from the Lemmas \ref{lem:free} and \ref{lem:D1} that 
$$
\Gr^W_\bullet \Psi_N : \Gr^W_\bullet\p(U_N,\vv_1)^\DR \to \Gr^W_\bullet \p(E_\uu',\vv_1)^\DR
$$
is strict with respect to $D^\bullet$. Apply inverse limits with respect to the weight filtration for the result.
\end{proof}
The depth filtration $D^\bullet$ pulls back along $\phi_\cyc$ and $\phi_\elp$ in \eqref{eqn:k} to induce a depth filtration on $\k_N$. 

\subsection{Polylog variations}

Next we use these compatible depth filtrations to construct elliptic and cyclotomic polylog VMHS over $\E_{\G_1(N)}'$ and $U_N$, respectively. In the next section, we will use the Hain map to relate their limit MHS at $\uu + \vv_1$. 

\subsubsection{Elliptic case}
\label{sec:ellPol}

The {\em elliptic polylog variation} $\Pol_{\elp,N}$ over $\E_{\G_1(N)}'$ is the quotient $\mathbfcal{P}_{\G_1(N)}/D^2\mathbfcal{P}_{\G_1(N)}$. We have the obvious short exact sequence of VMHS
$$
\xymatrix{0 \ar[r] & \Gr_D^1\bP_{\G_1(N)} \ar[r] & \Pol_{\elp,N} \ar[r] & \Gr_D^0\bP_{\G_1(N)} \ar[r] & 0.}
$$
Define $\H := \Gr_D^0 \bP_{\G_1(N)} = \Gr^W_{-1}.$\label{not:H} The KZB connection on this subquotient variation $\H$ is simply $d +2\pi i\bY\partial/\partial\bX$. It follows that
$$
\Gr^W_\bullet\Gr_D^1 \bP_{\G_1(N)} = \Gr^W_\bullet\bigoplus_{\zeta \in \bmu_N} \Sym \H \cdot \bt_\zeta \cong \Gr^W_\bullet\bigoplus_{\zeta \in \bmu_N} (\Sym \H)(1) 
$$
since each $\bt_\zeta$ generates a constant variation MHS $\Q(1)$ modulo $D^2$. Thus, the polylog variation $\Pol_{\elp,N}$ is a direct sum of extensions of $\H$ by symmetric powers $S^m\H(1)$. This is a level $N$ or cyclotomic generalization of Beilinson and Levin's elliptic polylog variation \cite{BL}. 

\subsubsection{Cyclotomic case}

The {\em cyclotomic polylog variation} is defined analogously except with the $\KZ_N$ VMHS. Define $\Pol^\cyc_N$ to be the quotient $\bP_N^\KZ/D^2\bP_N^\KZ$. As in the elliptic case, there is a short exact sequence 
$$
\xymatrix{0 \ar[r] & \Gr_D^1 \bP_N^\KZ \ar[r] & \Pol^\cyc_N \ar[r] & \Gr_D^0 \bP^\KZ_N \ar[r] & 0.}
$$
We know the holomorphic vector bundle $\Gr^1_D \bP_N^\KZ \otimes \cO_{U_N}$ is spanned by $\ee_0^m \cdot \ee_\zeta$. In the next section, we will show $\Gr^1_D \bP_N^\KZ$ in fact splits as a VMHS
$$
\Gr^M_\bullet\Gr^1_D \bP_N^\KZ \cong \Gr^M_\bullet\bigoplus_{\substack{m \geq 0 \\ \zeta \in \bmu_N}} \Q(m+1)^\DR,
$$
where $\ee_0^m\cdot \ee_\zeta$ each span a copy of $\Q(m+1)^\DR$. Meanwhile, the quotient $\Gr_D^0 \bP_N^\KZ$ is spanned by the class of $\ee_0$, hence $\Gr_D^0 \bP_N^\KZ$ is isomorphic to the constant VMHS $\Q(1)$ over $U_N$. 

\begin{remark}
One can also obtain the variation $\Pol^\cyc_N$ by restricting $\Pol^\elp_N$ to a neighborhood of the cusp $q = 0$ and taking monodromy invariants. We know from \eqref{eqn:KZBrest} that the residue of the connection on $\Pol^\elp_N$ is $\bY\partial/\partial\bX$, and thus monodromy around the cusp is given by $\exp(-\bY\partial/\partial\bX)$. The invariants are generated by $\bY$ and $\bt_\zeta$, which is precisely the image of $\Pol^\cyc_N$ under the Hain map. 
\end{remark}

\section{Limit MHS of polylogarithm variations}
\label{sec:polMHS}

To compute the action \eqref{eqn:k} of $\k_N$, we must know the limit MHS of the cyclotomic and elliptic polylogarithm variations at $\vv_1$. We proceed by first computing the limit MHS of the cyclotomic polylog at $\ww_1$ and then use parallel transport of the $\KZ_N$ connection along $\dch$ to compute the limit MHS of the cyclotomic polylog at $\vv_1$.

\subsection{Preliminaries}

If $X$ is a topological space, denote the free abelian group on the set of homotopy classes of paths in $X$ from $x$ to $y$ by $H_0(P_{x,y}X)$. These groups form a local system
\begin{equation}
\label{eqn:H0}
\{H_0(P_{x,y}X)\}_{(x,y)} \to X \times X.
\end{equation}
If $x = y$, then $H_0(P_{x,y}X)$ is canonically isomorphic to $\Q\pi_1(X,x)$. Powers of the augmentation ideal $I$ of $\Q\pi_1(X,x)$ define a filtration on $H_0(P_{x,y}X)$, which extends to a flat filtration of the local system \eqref{eqn:H0}. Each fiber has a completion $H_0(P_{x,y}X,\Q)^\wedge$ in the $I$-adic topology. The corresponding variation of completed groups
$$
\{H_0(P_{x,y}X,\Q)^\wedge\}_{(x,y)} \to X \times X
$$
is an admissible VMHS whose fiber over $(x,x)$ is canonically isomorphic to the MHS on $\Q\pi_1(X,x)^\wedge$ \cite{HZ}.

\subsection{The limit MHS of $\Pol^\cyc_N$ at $\ww_1$} 

The variation \eqref{eqn:H0} can be described explicitly when $X=U_N$. Restrict the variation $\{H_0(P_{x,y}U_N,\Q)^\wedge\}_{(x,y)}$ to $U_N \times \{\ww_1\}$:
$$
\{H_0(P_{w,\ww_1}U_N,\Q)^\wedge\}_{w\in U_N} \to U_N.
$$
By \eqref{eqn:chen}, the restricted variation is isomorphic, as a flat vector bundle, to 
$$
\C\ll\ee_0,\ee_\zeta \mid \zeta \in \bmu_N \rr \times U_N \to U_N
$$
with connection $\nabla = d + \ee_0 \omega_0 + \sum \ee_\zeta \omega_\zeta$, where each $\ee_r$ acts via left multiplication. This is a cyclotomic version of the classical polylogarithm variation \cite{hain:polylog}.

Recall the augmentation ideal $J$ of $\C\ll \ee_0,\ee_\zeta \mid \zeta \in \bmu_N\rr$. Define the left ideal $\cJ$
of $H_0(P_{\ww_1,\ww_1}U_N)^\wedge$ generated by $\{J\ee_0, \ee_\zeta J\mid \zeta \in \bmu_N\}$. Set $\cJ_m = \cJ + J^m$ for $m \ge 2$. Then the quotient $J/\cJ_m$ has basis $\big\{\ee_0,\ee_\zeta, \ee_0 \ee_\zeta, \ldots, \ee_0^{m - 2}\ee_\zeta \mid \zeta \in \bmu_N\big\}$.
Let $\J_{N,m}$ be the variation over $U_N$ isomorphic to the trivial bundle $J/\cJ_m \times U_N \to U_N$ with connection $\nabla$. Set $\J_N := \varprojlim\limits_m \J_{N,m}$. Then $\J_N$ isomorphic to the trivial bundle $J/\cJ \times U_N \to U_N$ with connection $\nabla$. 
\begin{remark}
As written, $\J_N$ and $\J_{N,m}$ are holomorphic vectors bundles with flat connection and Hodge and weight filtrations. In order to view them as VMHS, we consider the rational structure induced by Chen's transport formula as in \S\ref{sec:piMHS}. 
\end{remark}
\begin{prop}
\label{prop:polisom}
There is an isomorphism of variations of MHS $\Pol^\cyc_N \to \J_N$ given by the expected formula 
\begin{align*}
    \ee_0 &\longmapsto \ee_0  \\
    \ee_0^n\cdot\ee_\zeta &\longmapsto \ee_0^{n} \ee_\zeta.
\end{align*}
\end{prop}
\begin{proof}
Both $\Pol^\cyc_N$ and $\J_N$ are sub-quotients of $\{H_0(P_{w,\ww_1}U_N,\Q)^\wedge\}_{w\in U_N}$. Moreover, we know $\ee_0^n\cdot \ee_\zeta \equiv \ee_0 \ee_\zeta \bmod \cJ$. Thus, the map preserves the connection and the rational structure. 
\end{proof}

Next, for each $\zeta \in \bmu_N$, set $U_\zeta := \Gm - \{\zeta\}$. Define $J_\zeta$ to be the augmentation ideal of $\C\ll \ee_0,\ee_\zeta\rr$ viewed as a subalgebra of $\C\ll \ee_0,\ee_\zeta: \zeta \in \bmu_N\rr$. Define the left ideal $\cJ_{m,\zeta} = \cJ_m \cap \C\ll \ee_0,\ee_\zeta\rr$. Let $\J_{N,m}^\zeta$ be the variation over $U_\zeta$ isomorphic to the trivial bundle $(J_\zeta/\cJ_{m,\zeta}) \times U_\zeta \to U_\zeta$ with connection $\nabla = d + \ee_0\omega_0 + \ee_\zeta\omega_\zeta$.
The rational structure of $\J_{N,m}^\zeta$ is pulled back from $\J_{1,m} = \J_{N,m}^1$ along the map $z \mapsto z/\zeta$. 

The inclusion $U_\zeta \hookrightarrow \Gm$ induces a surjection on completed group algebras $\Q\pi_1(U_\zeta,\ww_1)^\wedge \to \Q\pi_1(\Gm,\ww_1)^\wedge$. This induces a surjection on the quotients
$$
\Upsilon_{N,m}^\zeta : \J_{N,m}^\zeta \big|_{\ww_1} \to \frac{J\cap\C\ll\ee_0\rr}{\cJ_m\cap\C\ll\ee_0\rr} \cong \C\ee_0.
$$
The inclusion $\D^\ast \hookrightarrow U_\zeta$ induces a splitting of map $\Upsilon_{N,m}^\zeta$ above and thus the limit MHS at $\ww_1$ splits as
\begin{equation}
    \label{eqn:Vsplit}
    \J_{N,m}^\zeta \big|_{\ww_1} \cong \Q(1) \oplus \ker \Upsilon_{N,m}^\zeta \big|_{\ww_1},
\end{equation}
where the de~Rham realization of the $\Q(1)$ summand is spanned by $\ee_0$. 
The other summand $\ker \Upsilon_{N,m}^\zeta$ is spanned by $\{\ee_\zeta,\ee_0\ee_\zeta,\ldots,\ee^{m-2}_0\ee_\zeta\}$. 

\begin{lemma}
The limit MHS on $\ker \Upsilon_{N,m}^\zeta$ at $\ww_1$ and $\vv_1$ are isomorphic to $\Q(1) \oplus \Q(2) \oplus \cdots \oplus \Q(m - 1)$, where the de~Rham generator of each summand is $\ee_0^{n-1}\ee_\zeta$. 
\end{lemma}
\begin{proof}
 The VMHS $\ker \Upsilon_{N,m}^\zeta$ is isomorphic to the $(m-1)$th symmetric power of the logarithm variation (see Example \ref{ex:Emhs}). Thus, if the limit MHS of $\ker \Upsilon_{N,m}^\zeta$ splits at $\ww_1$, it does at $\vv_1$ as well. To show the limit MHS splits at $\ww_1$, first observe  
\begin{equation}
    \label{eqn:zetaChange}
    H_0(P_{w,\ww_1}U_\zeta,\Q)^\wedge \cong H_0(P_{w\zetabar,\ww_\zetabar}U_1,\Q)^\wedge
\end{equation}
via the change of variables $w \mapsto w/\zeta$ and $\ee_\zeta \mapsto \ee_1$. Let $\gamma$ be the loop in $U_1$ based at $\ww_\zeta$ which travels counterclockwise to $\ww_1$, via $\dch$ to $-\vv_1$, around 1 counterclockwise, and back to to $\ww_\zetabar$.  
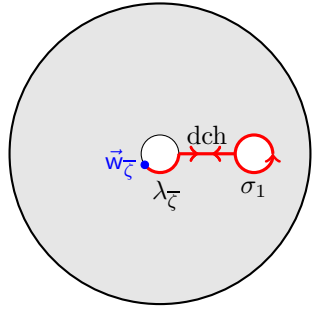
\begin{center}
\begin{tikzpicture}[scale=1]
    \draw[thick,fill=gray!20] (0, 0) circle (2cm);
    \draw[fill=white] (0, 0) circle (.25cm);
    \draw[fill=white] (1.25, 0) circle (.25cm);
    \draw[very thick,red,->-=.35,-<-=.75] (.25,0) -- (1,0);
    
    \draw[very thick, red] (-.202, -.147) arc (216:360:.25);
    \draw[very thick, red, ->] (1.5,0) arc (0:360:.25);
        
        
    
    \draw[fill=blue,blue] (-.202, -.147) circle (.05cm);
        
    \draw (-.202, -.147) node[left] {{\color{blue} $\ww_\zetabar$}};
    \draw (.625,0) node[above] {dch};
    \draw (1.25,-.25) node[below] {$\sigma_1$};
    \draw (.077,-.238) node[right,below] {$\lambda_\zetabar$};

    \end{tikzpicture}
\end{center}
For convenience, denote the path from $\ww_\zetabar$ to $\ww_1$ by $\lambda_\zetabar$ and the loop about 1 by $\sigma_1$. We will use Chen's transport formula \eqref{eqn:chen} along $\gamma$ to compute an element of the rational structure of $H_0(P_{\ww_\zetabar,\ww_\zetabar}U_1,\Q)^\wedge$ which descends to a rational vector in $J/\cJ_{N,m}^1$. We first compute some {\em regularized} iterated integrals needed in the proof. These are
\begin{enumerate}
    \item $\displaystyle \int_{\lambda_\zetabar} \underbrace{\omega_0 \cdots \omega_0}_n = \frac{1}{n!}\frac{(2\pi i)^n}{k^n}$ where $0 \leq k < N$ and $\zeta = e^{2\pi ik/N}$,
    \item $\displaystyle \int_\dch \underbrace{\omega_0 \cdots \omega_0}_n = \frac{1}{n!}\left(\int_\dch \omega_0\right)^n = 0$,
    \item $\displaystyle \int_{\sigma_1} \omega_{r_1} \cdots \omega_{r_n} = \frac{1}{n!}\prod_{k = 1}^n \int_{\sigma_1} \omega_{r_k} = 0$ unless $r_k = 1$ for all $k$,
    \item $\displaystyle \int_{\sigma_1} \omega_1 = 2\pi i$.
\end{enumerate}
Chen's transport along the loop based at $\ww_1$ yields
\begin{align*}
    \Theta_{\ww_1}&(\dch \cdot \sigma_1 \cdot \dch^{-1}) \cr
    &\equiv 1 + \sum_{\substack{n \ge 1 \\ r_s \in \{0,1\}}} \ee_{r_1} \cdots \ee_{r_n} \int_\gamma \omega_{r_1} \cdots \omega_{r_n} \cr
    &\equiv 1 + \sum_{\substack{n \ge 1 \\ j + k + \ell = n \\ r_s \in \{0,1\}}} \ee_{r_1} \cdots \ee_{r_n} \int_{\dch} \omega_1 \cdots \omega_{r_j}\int_{\sigma_1} \omega_{r_{j + 1}} \cdots \omega_{r_{j+k}}\int_{\dch^{-1}} \omega_{r_{j+k + 1}} \cdots \omega_{r_n} \cr
    &\equiv 1 + 2\pi i \ee_1 \bmod \cJ_{N,m}^1
\end{align*}
since every product vanishes except when $j = \ell = 0$, $k = 1$, and $r_1 = 1$. Let $1_{\ww_\zetabar}$ be the constant path at $\ww_\zetabar$. It follows that
\begin{align*}
    \Theta_{\ww_\zetabar}(\gamma - 1_{\ww_\zetabar}) &\equiv \left(1 + \sum_{n \ge 1}\int_{\lambda_\zetabar} \underbrace{\omega_0 \cdots \omega_0}_n\right)\left(1 + 2\pi i\ee_1\right)\left(1 + \sum_{n \ge 1}\int_{\lambda_\zetabar^{-1}} \underbrace{\omega_0 \cdots \omega_0}_n\right) - 1 \cr
    &\equiv \sum_{n = 0}^{m-2} \frac{1}{n!} \frac{(2\pi i)^{n+1}}{k^n}
    \ee_0^n\ee_1 \bmod \cJ_{N,m}^1
\end{align*}
is a $\Q$-Betti vector of the limit MHS on $H_0(P_{\ww_\zetabar,\ww_\zetabar}U_1,\Q)^\wedge$ at $\ww_\zetabar$. We also know from \eqref{eqn:Vsplit} and \eqref{eqn:zetaChange} that $2\pi i\ee_0$ is also $\Q$-Betti. Thus, all products of $2\pi i\ee_0$ and $\Theta_{\ww_\zetabar}(\gamma)$ are also $\Q$-Betti. Since the $\Q$-Betti structure is closed under addition and multiplication, we conclude $(2\pi i)^{n + 1} \ee_0^n \ee_1$ is $\Q$-Betti modulo $\cJ^1_{N,m}$ for $n = 0, 1, \ldots, (m - 2)$. The result follows from the change of variables in \eqref{eqn:zetaChange}. 
\end{proof}
\begin{cor}
The limit MHS of $\Pol^\cyc_N$ at $\ww_1$ splits completely as 
$$
\Pol^\cyc_N\big|_{\ww_1} \cong \Q(0) \oplus \bigoplus_{\substack{n \ge 0 \\ \zeta \in \bmu_N}} \Q(n + 1),
$$
where $\ee_0$ spans $\Q(0)^\DR$ and each $\ee_0^n \cdot \ee_\zeta$ spans a copy of $\Q(n + 1)^\DR$. 
\end{cor}
\begin{proof}
The inclusions $U_N \hookrightarrow U_\zeta$ induce morphisms of VMHS $\J_{N,m} \to \J_{N,m}^\zeta$ for each $\zeta \in \bmu_N$. The inclusion $U_N \hookrightarrow \Gm$ induces a split surjection $\Upsilon_{N,m} : \J_{N,m} \big|_{\ww_1} \to \Q(1)$. The variation $\J_{N,m}$ is the pullback
$$
\xymatrix{
0 \ar[r] & \displaystyle\bigoplus_{\zeta \in \bmu_N} \ker\Upsilon_{N,m}^\zeta \ar[r] & \displaystyle\bigoplus_{\zeta \in \bmu_N} \J_{N,m}^\zeta \ar[r] & \displaystyle\bigoplus_{\zeta \in \bmu_N} \Q(1) \ar[r] & 0 \cr
0 \ar[r] & \ker \Upsilon_{N,m} \ar[r] \ar@{=}[u] & \J_{N,m} \ar[u] \ar[r] & \Q(1) \ar[r] \ar[u]_-\Delta & 0,
}
$$
where $\Delta$ is the diagonal map. After taking inverse limits, it follows from \eqref{eqn:Vsplit} and the previous lemma that the top row (and hence the bottom row) splits. The result follows from Proposition \ref{prop:polisom}.
\end{proof}

\subsection{The limit MHS of $\Pol^\cyc_N$ at $\vv_1$}
\label{sec:cycpoly}

We now compute the limit MHS of $\Pol^\cyc_N$ at $\vv_1$ by transporting the $\Q$-Betti structure of the limit MHS at $\ww_1$ to $\vv_1$. We first transport along the straight line path $\dch$ to $-\vv_1$ and then counter-clockwise around the puncture at $1 \in \Gm$ to $\vv_1$. 

Observe
\begin{align*}
    T(\dch) &\equiv T(\dch^{-1})^{-1} \cr
    &\equiv 1 + \sum_{\substack{m \geq 1 \\ \zeta \in \bmu_N}} \ee_0^m \cdot \ee_\zeta \int_{\dch^{-1}} \underbrace{\omega_0 \cdots \omega_0}_{m}\omega_\zeta \cr
    &\equiv 1 + \sum_{\substack{m \geq 1 \\ \zeta \in \bmu_N}} (-1)^{m+1}\ee_0^m \cdot \ee_\zeta \int_{\dch} \omega_\zeta\underbrace{\omega_0 \cdots \omega_0}_{m} \cr
    &\equiv 1 + \sum_{\substack{m \geq 1 \\ \zeta \in \bmu_N}} (-1)^m\Li_{m+1}(\zetabar)\, \ee_0^{m}\cdot \ee_\zeta \bmod D^2. 
\end{align*}
The third and fourth lines follow from \eqref{eqn:inversion} and \eqref{eqn:Li}, respectively. Applying $T(\dch)$ to the $\Q$-Betti basis of the limit MHS at $\ww_1$ yields the $\Q$-Betti basis of the limit MHS at $-\vv_1$: 
\begin{equation*}
    T(\dch) : \left\{ 
    \begin{array}{rcl}
        (2\pi i)^{n+1}\ee_0^n \cdot \ee_\zeta & \longmapsto & (2\pi i)^{n+1}\ee_0^n \cdot \ee_\zeta  \cr
        & & \cr
        2\pi i\ee_0 & \longmapsto & \displaystyle 2\pi i\ee_0 + 2\pi i\sum_{\substack{m \geq 2 \\ \zeta \in \bmu_N}} (-1)^{m}\Li_m(\zetabar)\,\ee_0^{m}\cdot \ee_\zeta.
    \end{array}
    \right.
\end{equation*}
Finally, transport from $-\vv_1$ to $\vv_1$ is simply $\exp(2\pi i\ee_1/2) \equiv 1 + \pi i \ee_1 \bmod D^2$. Applying this to the image of $T(\dch)$ yields a $\Q$-Betti basis of $\p(U_N,\vv_1)$ given in the following table. 
\begin{center}
    \begin{tabular}{c | l}
        Weight & $\Q$-Betti generators \\
        \hline & \\
        $-2$ & $\displaystyle 2\pi i \ee_0 - \frac{(2\pi i)^2}{2}\ee_0\cdot \ee_1 + 2\pi i \sum_{\substack{m \ge 2 \\ \zeta \in \bmu_N}} (-1)^{m}\Li_m(\zetabar)\, \ee_0^m\cdot \ee_\zeta$ \\
        $-2m-2$ & $(2\pi i)^{m+1} \ee_0^m \cdot \ee_\zeta$
    \end{tabular}
\end{center}
For later use, we define a slightly different basis. 
\begin{lemma}
\label{lem:Li}
If $m \ge 2$ and $\zeta \in \bmu_N$, then
$$
\Li_m(\zeta) + (-1)^m\Li_m(\zetabar) \in (2\pi i)^m\Q
$$
\end{lemma}
\begin{proof}
The $m$th Bernoulli polynomial $B_m(x)$ has Fourier expansion
$$
B_m(x) = -\frac{m!}{(2\pi i)^m} \sum_{k \neq 0} \frac{e^{2\pi i kx}}{k^m}.
$$
Thus, 
\begin{align*}
-\frac{(2\pi i)^m}{m!}B_m(x) &= \sum_{k < 0} \frac{(e^{2\pi i x})^k}{k^m} + \sum_{k > 0} \frac{(e^{2\pi i x})^k}{k^m} \cr
&= \sum_{k > 0} \frac{(e^{-2\pi i x})^k}{(-k)^m} + \sum_{k > 0} \frac{(e^{2\pi i x})^k}{k^m} \cr
&= (-1)^m \Li_m(e^{-2\pi ikx}) + \Li_m(e^{2\pi ikx}).
\end{align*}
Set $x = \frac{1}{2\pi i}\log \zeta$ with choice of logarithm such that $x \in [0,1)$. The result follows from the fact both $x$ and the coefficients of the polynomial $B_m(x)$ are in $\Q$.
\end{proof}
It follows from the lemma that the following generators also form a $\Q$-Betti basis of $\p(U_N,\vv_1)/D^2$. 
\begin{center}
    \begin{tabular}{c | l}
    \label{tab:cycB}
        Weight & $\Q$-Betti generators \\
        \hline & \\
        $-2$ & $\displaystyle 2\pi i \ee_0 - 2\pi i \sum_{\substack{m \ge 2 \\ \zeta \in \bmu_N}} \Li_m(\zeta)\, \ee_0^m\cdot \ee_\zeta$ \\
        $-2m-2$ & $(2\pi i)^{m+1} \ee_0^m \cdot \ee_\zeta$
    \end{tabular}
\end{center}

Let $E_{m,\zeta} \in \Ext_\MHS^1(\Q,\Q(m))$ \label{not:Eext} denote the extension of $\Q$ by $\Q(m)$ with period $\Li_m(\zeta)$. The sub-MHS of $\p(U_N,\vv_1)/D^2$ \label{not:cyc} generated by $\ee_0$ and $\ee_0^m \cdot \ee_\zeta$ is an extension of $\Q(m+1)$ by $\Q(1)$ with period $-\Li_m(\zeta)$, hence isomorphic to $-E_{m,\zeta}(1)$. Thus, $\p(U_N,\vv_1)/D^2$ is the pull-back of the direct product of all extensions $-E_{m,\zeta}(1)$ via the diagonal map $\Delta : \Q(1) \to \prod \Q(1)$. 
$$
\xymatrix{
0 \ar[r] & \displaystyle \prod_{\substack{m \ge 0 \\ \zeta \in \bmu_N}}\Q(m+1) \ar[r] \ar@{=}[d] & \p(U_N,\vv_1)/D^2 \ar[r] \ar[d] & \Q(1) \ar[r] \ar[d]^-\Delta & 0 \cr
0 \ar[r] & \displaystyle\prod_{\substack{m \ge 0 \\ \zeta \in \bmu_N}}\Q(m+1) \ar[r] & \mathbf{E} \ar[r] & \displaystyle\prod_{\substack{m \ge 2 \\ \zeta \in \bmu_N}}\Q(1) \ar[r] & 0
}
$$
The extension $\mathbf{E}$ in the diagram is 
$$
\mathbf{E} = \bigoplus_{\zeta \in \bmu_N} (\Q(1) \oplus \Q(1)) \oplus \bigoplus_{\zeta \in \bmu_N} (\Q(2) \oplus \Q(1)) \oplus \prod_{\substack{m \ge 2 \\ \zeta \in \bmu_N}} -E_{m,\zeta}(1),
$$
where each copy of $\Q(1)\oplus \Q(1)$ is a trivial extension spanned by $2\pi i \ee_0$ and $2\pi i \ee_\zeta$ each copy of $\Q(2) \oplus \Q(1)$ is a trivial extension spanned by $(2\pi i)^2 \ee_0 \cdot \ee_\zeta$ and $2\pi i \ee_0$.

\subsection{The limit MHS of $\Pol^\elp_N$ at $\vv_1$}
\label{sec:lmhsPolEll}

Recall from \S\ref{sec:ellPol} that $\Pol^\elp_N$ is an extension of VMHS
$$
\xymatrix{0 \ar[r] & \displaystyle\bigoplus_{\substack{m \ge 0 \\ \zeta \in \bmu}} S^m \H \cdot \bt_\zeta \ar[r] & \Pol^\elp_N \ar[r] & \H \ar[r] & 0.} 
$$
Limit MHS at $\uu + \vv_1$ fit in the short exact sequence
$$
\xymatrix{0 \ar[r] & \displaystyle\prod_{\substack{m \ge 0 \\\zeta\in\bmu_N}} S^mH_\uu \cdot \bt_\zeta \ar[r] & \p(E_\uu',\vv_1)/D^2 \ar[r] & H_\uu \ar[r] & 0,}
$$
where $H_\uu$ is the limit MHS of $\H$ at $\uu$. Note that $H_\uu \cong \Q(0) \oplus \Q(1)$ with de~Rham generators $\bX$ and $\bY$ (see \cite[Ex.~11.3]{hopper}). 

Let $V_{m,\zeta} \in \Ext_\MHS^1(H_\uu,S^mH_\uu(1))$ be the sub-quotient of $\p(E_\uu',\vv_1)/D^2$ generated by $H_\uu$ and $S^mH_\uu\cdot\bt_\zeta$. As with the cyclotomic polylog, we observe that $\p(E_\uu',\vv_1)/D^2$ is the pull-back of the direct product of the extensions $V_{m,\zeta}$ via the diagonal map $\Delta : H_\uu \to \prod H_\uu$.
$$
\xymatrix{
0 \ar[r] & \displaystyle\prod_{\substack{m \ge 0 \\\zeta\in\bmu_N}} S^mH_\uu \cdot \bt_\zeta \ar[r] \ar@{=}[d] & \p(E_\uu',\vv_1)/D^2 \ar[r] \ar[d] & H_\uu \ar[r] \ar[d]^-\Delta & 0 \cr
0 \ar[r] & \displaystyle\prod_{\substack{m \ge 0 \\\zeta\in\bmu_N}} S^mH_\uu \cdot \bt_\zeta \ar[r] & \displaystyle\prod_{\substack{m \ge 0 \\\zeta\in\bmu_N}} V_{m,\zeta} \ar[r] & \displaystyle\prod_{\substack{m \ge 0 \\\zeta\in\bmu_N}} H_\uu \ar[r] & 0
}
$$
The Hain map preserves depth filtrations and thus induces a map on the polylog quotients. 
\begin{equation}
    \label{eqn:pMHS}
    \xymatrix{
    0 \ar[r] & \displaystyle\prod_{\substack{m \ge 0 \\\zeta\in\bmu_N}} \C\ee_0^m\cdot \ee_\zeta \ar[r] \ar[d] & \p(U_N,\vv_1)/D^2 \ar[r] \ar[d]_{\Psi_N} & \C\ee_0 \ar[r] \ar[d] & 0 \cr
    0 \ar[r] & \displaystyle\prod_{\substack{m \ge 0 \\\zeta\in\bmu_N}} S^mH_\uu \cdot \bt_\zeta \ar[r] & \p(E_\uu',\vv_1)/D^2 \ar[r] & H_\uu \ar[r] & 0
    }
\end{equation}
Our task is to deduce the periods of the lower extension from the periods of the upper extension. First, we consider a decomposition.

\begin{lemma}
The $m$th symmetric power of the limit MHS $H_\uu$ splits as
$$
S^mH_\uu = \Q(0) \oplus \cdots \oplus \Q(m).
$$
with de~Rham generators $\bX^{m-r}\bY^r$ for $r = 0,1,\ldots,m$.
\end{lemma}

Recall $-E_{m,\zeta}(1)$ is the extension of $\Q(1)$ by $\Q(m+1)$ with period $-\Li_m(\zeta)$. The previous diagram decomposes into a direct product of maps $-E_{m,\zeta}(1) \to V_{m,\zeta}$ for $m \ge 2$ (there are also analogous maps $m = 0$ and $m = 1$, but these extensions split and thus have no periods of interest).
$$
\xymatrix{
0 \ar[r] & \C \ee_0^m \cdot \ee_\zeta \ar[r] \ar[d] & -E_{m,\zeta}(1) \ar[r] \ar[d] & \C\ee_0 \ar[r]\ar[d] & 0 \cr
0 \ar[r] & \displaystyle\bigoplus_{r = 0}^m \C\bX^{m-r}\bY^r\cdot \bt_\zeta \ar[r] & V_{m,\zeta} \ar[r] & \C\bX \oplus \C\bY \ar[r] & 0
}
$$
The Betti $\Q$-basis of the top extension $-E_{m,\zeta}(1)$ maps to a Betti $\Q$-basis of the bottom extension $V_{m,\zeta}$.
\begin{center}
    \begin{tabular}{c | l}
        $M$-weight & $\Q$-Betti generator of $V_{m,\zeta}$\\
        \hline & \\
        $-2$ & $2\pi i \bY - 2\pi i \,\Li_m(\zeta)\bY^m\cdot \bt_\zeta$ \\
        & \\
         $-2r - 2$ & $ (2\pi i)^{r+1}\bX^{m-r}\bY^r\cdot\bt_\zeta$
    \end{tabular}
\end{center}
There must also be a $\Q$-Betti vectors
$$
\bX^\betti_\zeta = \bX - \sum_{r = 0}^mc_r\bX^{m-r}\bY^r \cdot \bt_\zeta
$$
in $V_{m,\zeta}$ that projects onto $\bX$ in $H_\uu$. To compute the constants $c_r$, we use that fact that monodromy preserves the rational structure. It follows from \eqref{eqn:KZBrest} and \eqref{eqn:chenT} that the local monodromy operator on $\Pol^\elp_N$ at $q = 0$ is $h_q := \exp(2\pi i \bY \partial/\partial \bX)$. Thus, 
\begin{align*}
(h_q - 1)\bX^\betti_\zeta &\equiv 2\pi i \bY - h_q \left( \sum_{r = 0}^mc_r\bX^{m-r}\bY^r \cdot \bt_\zeta\right) \cr 
&\equiv 2\pi i \bY - 2\pi i \sum_{r= 0}^{m-1} c_r\bX^{m-r-1}\bY^{r+1} \cdot \bt_\zeta \bmod D^2
\end{align*}
is also rational. Given the $\Q$-Betti vectors in the table above, we may choose $c_{m-1} = \Li_m(\zeta)$ and $c_r = 0$ for $r \neq m-1$. We have proven the following.
\begin{prop}
When $m \ge 2$, the extension $V_{m,\zeta}$ has a de~Rham splitting given by 
$$
\left\{
\begin{array}{l}
\bX \longmapsto \bX -\Li_m(\zeta)\bX\bY^{m-1}\cdot \bt_\zeta \cr 
\bY \longmapsto \bY -\Li_m(\zeta)\bY^m \cdot \bt_\zeta.
\end{array}
\right.
$$
\end{prop}
\begin{cor}
\label{cor:Vsplit}
When $m \ge 2$ and $\zeta \neq 1$, the pull-back of $V_{m,\zeta} \oplus V_{m,\zetabar}$ along the diagonal map $H_\uu \to H_\uu \oplus H_\uu$ has de~Rham splitting
$$
\left\{
\begin{array}{l}
\bX \longmapsto \bX-\Li_m(\zeta)\bX\bY^{m-1}\cdot (\bt_\zeta + (-1)^{m+1}\bt_\zetabar) \cr 
\bY \longmapsto \bY-\Li_m(\zeta)\bY^m\cdot (\bt_\zeta + (-1)^{m+1}\bt_\zetabar).
\end{array}
\right.
$$
\end{cor}
\begin{proof}
Use the splitting in the previous statement for $\zeta$ and $\zetabar$. Then recall $\Li_m(\zeta) + (-1)^m \Li_m(\zetabar) \in (2\pi i)^m\Q$.
\end{proof}

This splitting may be described in terms of the derivations which appear in the KZB connection. For $m \ge 2$, the derivations $\e_{m+1,\zeta}$\label{not:ep} in \eqref{eqn:KZBrest} are given by 
$$
\e_{m+1,\zeta} \equiv \bX^{m-1}\cdot(\bt_\zeta+(-1)^{m+1}\bt_\zetabar) \bmod D^2 \Der \p_N
$$
\cite[\S7.1]{hopper}. Then define
\begin{align}
\label{eqn:eop}
\begin{split}
\e^\op_{m+1,\zeta} &= \frac{1}{(m-1)!} \left(\bY\frac{\partial}{\partial \bX}\right)^{m-1} \cdot \e_{m+1,\zeta} \cr 
&\equiv \bY^{m-1}\cdot(\bt_\zeta+(-1)^{m+1}\bt_\zetabar) \bmod D^2 \Der \p_N.
\end{split}
\end{align}
\begin{cor}
The de~Rham splitting of the pull-back of $V_{m,\zeta} \oplus V_{m,\zetabar}$ along the diagonal map $H_\uu \to H_\uu \oplus H_\uu$ is
$1 + \Li_m(\zeta)\,\e^\op_{m+1,\zeta} \bmod D^2 \Der \p_N$.
\end{cor}
\begin{remark}
In the case $\zeta = 1$, we see $V_{m,1}$ splits when $m$ is even and has de~Rham splitting $1+\frac12\zeta(m)\e^\op_{m+1,1}$ when $m$ is odd.
\end{remark}
\begin{remark}
It follows from Lemma \ref{lem:Li} and \eqref{eqn:eop} that $\Li_m(\zeta)\,\e^\op_{m,\zeta}$ is invariant under complex conjugation.
\end{remark}

\section{Hodge realizations of \texorpdfstring{$\Ext^1_{\MTM_N}(\Q,\Q(m))$}{TEXT}}
\label{sec:Eext}

We observed in \S\ref{sec:cycpoly} that the Hodge realization of $\p(U_N,\vv_1)/D^2$ is the sum of extensions $-E_{m,\zeta}(1)$. Since $E_{m,\zeta}$ is a subquotient of the mixed Tate motive $\p(U_N,\vv_1)$, it is the realization of a rational class in $\Ext_{\MTM_N}^1(\Q,\Q(m))$ \cite[\S 5]{DG}. 
Consider the composition 
\begin{multline}
    \label{eqn:extLi}
    \Span_\Q\{E_{m,\zeta}\mid\zeta \in \bmu_N\} \hookrightarrow \Ext^1_{\MTM_N}(\Q,\Q(m))
    \cr
    \to \Ext^1_\MHS(\Q,\Q(m)) \xrightarrow{\sim} \C/(2\pi i)^m\Q,
\end{multline}
where the middle map is the Hodge realization functor and the last map is the period map.
The image of each extension $E_{m,\zeta}$ in $\C/(2\pi i)^m\Q$ is $\Li_m(\zeta)$. When $m \ge 2$, Goncharov \cite{gonch:poly} proved there is an isomorphism 
$$
\Span_\Q \{\Li_m(\zeta) \mid \zeta \in \bmu_N\} \cong K_{2m - 1}(\cO_N) \otimes \Q. 
$$
Thus, by \eqref{eqn:K} the inclusion in \eqref{eqn:extLi} is an isomorphism. Therefore, $\Q$-linear relations in $\Ext^1_{\MTM_N}(\Q,\Q(m))$ correspond exactly to $\Q$-linear relations between the periods $\Li_m(\zeta)$ in $\C/(2\pi i)^m\Q$. 

\begin{remark}
When $m = 1$, the extensions $E_{1,\zeta}$ do {\em not} span $\Ext^1_{\MTM_N}(\Q,\Q(1))$. One must add extensions with period $\log p$ for all primes $p$ that divide $N$. This is essentially equivalent to Dirichlet's unit theorem.
\end{remark}



We recall a useful fact about polylogarithms. 
\begin{lemma}
Suppose $\ell$ is a positive integer. Then the {\em distribution relation} of polylogarithms is
\begin{equation}
    \label{eqn:dist}
    \ell^{m-1}\sum_{w^\ell = z} \Li_m(w) = \Li_m(z).
\end{equation}
\end{lemma}
The proof follows directly from the power series expansion of the polylogarithm function. Details can be found in \cite[\S 2]{gonch:poly}. 

\begin{cor}
Suppose $\zeta$ is a primitive $d$th root of unity. If $p$ is prime, then $\Li_m(\zeta)$ is a $\Q$-linear combination of values of $\Li_m$ at primitive $(dp)$th roots of unity. 
\end{cor}
\begin{proof}
There exists a primitive $(dp)$th root of unity $\xi$ such that $\xi^p = \zeta$. Then by the distribution relation,
$$
\Li_m(\zeta)  = p^{m-1}\sum_{w^p = \zeta} \Li_m(w) = p^{m-1}\sum_{k = 0}^{p - 1} \Li_m(\xi^{kd + 1}).
$$
If $p\mid d$, then $\gcd(kd+1,p) = 1$ for all $k$; hence, $\gcd(kd+1,dp) = 1$ and each $\xi^{kd+1}$ is primitive, completing the proof. On the other hand, if $p\nmid d$, then there is exactly one $\ell \in \{0,1,\ldots,p-1\}$ such that $p\mid \ell d+1$. Then $\xi^{\ell d +1}$ is a $d$th primitive root of unity. Every root $\xi^{kd +1}$ with $k \neq \ell$ is a primitive $(dp)$th root of unity. One can then use the distribution relation again to rewrite $\Li_m(\xi^{\ell d+1})$ as a sum of values of $\Li_m$. After no more than $p$ iterations, the original primitive $d$th root of unity $\zeta$ will appear in the sum with a coefficient not equal to 1. One can then solve for $\Li_m(\zeta)$ in terms of $\Li_m(\eta)$ with $\eta \in \bmu_{dp}$ primitive. 
\end{proof}
\begin{prop}
\label{prop:extbasis}
When $m \ge 2$, the extensions $E_{m,\zeta}$ with $\zeta \in \bmu_N$ primitive and $\Im \zeta > 0$ form a basis of $\Ext_{\MTM_N}^1(\Q,\Q(m))$. 
\end{prop}
\begin{proof}
Suppose $d \mid N$ and $\eta$ is a primitive $d$th root of unity. There are primes $p_1,\ldots,p_n$ (not necessarily distinct) such that $N = dp_1\cdots p_n$. One can apply the result of the previous lemma $n$ times to write $\Li_m(\eta)$ as $\Q$-linear combination of $\Li_m(\zeta)$ with $\zeta \in \bmu_N$ primitive. Thus, the extensions $E_{m,\zeta}$ with $\zeta \in \bmu_N$ primitive generate $\Ext^1_{\MTM_N}(\Q,\Q(m))$. It follows from Lemma \ref{lem:Li} that $E_{m,\zetabar} \equiv (-1)^{m+1} E_{m,\zeta} \bmod \Q(m)$. Thus, we need only consider those $E_{m,\zeta}$ with positive imaginary part. The result follows from counting dimension and the fact that the inclusion in \eqref{eqn:extLi} is an isomorphism. 
\end{proof}

\begin{example}
\label{ex:E1}
Suppose $N = p^n$ is a prime power. Let $\zeta \in \bmu_N$ be a primitive $(p^k)$th root of unity with $0 < k \le n$. It follows from the distribution relation that
\begin{equation}
    \label{eqn:ppdist}
    \Li_m(\zeta) = (p^{n-k})^{m-1} \sum_{\substack{\eta \in \bmu_N \\ \eta^{p^{n - k}} = \zeta}} \Li_m(\eta).
\end{equation}
Note that each of the roots $\eta$ in the sum are primitive. Then 
\begin{align*}
\zeta(m) = \Li_m(1) &= N^{m-1}\sum_{\zeta \in \bmu_N} \Li_m(\zeta) \cr
&= N^{m-1} \zeta(m) + N^{m-1}\sum_{k = 1}^n \sum_{\substack{\zeta \in \bmu_N \\ \text{$p^k$-primitive}}} \left((p^{n-k})^{m-1} \sum_{\substack{\eta \in \bmu_N \\ \eta^{p^{n - k}} = \zeta}} \Li_m(\eta)\right) \cr
&= N^{m-1} \zeta(m) + N^{m-1} \sum_{\substack{\zeta \in \bmu_N \\ \text{primitive}}} \sum_{k = 1}^n (p^{n-k})^{m-1}\Li_m(\zeta) \cr
&= N^{m-1} \zeta(m) + N^{m-1}  \frac{1-N^{m-1}}{1-p^{m-1}}\sum_{\substack{\zeta \in \bmu_N \\ \text{primitive}}} \Li_m(\zeta).
\end{align*}
Rearranging, 
$$
\zeta(m) = \frac{N^{m-1}}{1 - p^{m-1}} \sum_{\substack{\zeta \in \bmu_N \\ \text{primitive}}} \Li_m(\zeta).
$$
Thus, as extensions in $\Ext^1_{\MTM_N}(\Q,\Q(m))$,
\begin{equation}
    \label{eqn:1dist}
    E_{m,1} = \frac{N^{m-1}}{1 - p^{m-1}} \sum_{\substack{\zeta \in \bmu_N \\ \text{primitive}}} E_{m,\zeta},
\end{equation}
where summation is the Baer sum of extensions. Note that by Lemma \ref{lem:Li} this last statement is trivial when $m$ is even.
\end{example}

\section{The action of \texorpdfstring{$H_1(\k_N)$}{TEXT}}
\label{sec:gen}

For the remainder of the paper, we are concerned only with the de~Rham fiber functor of $\MTM_N$, so we will omit the $\DR$ decoration from $\cG_N^\DR$, $\mathcal{K}_N^\DR$, and $\k_N^\DR$.

The Hochschild--Serre spectral sequence associated to the short exact sequence \eqref{eqn:Gses} converges:
$$
E_2^{s,t} = H^s(\Gm,H^t(\mathcal{K}_N,\Q(m))) \Rightarrow H^{s+t}(\cG_N,\Q(m)).
$$
The group $\Gm$ is reductive, and thus $H^s(\Gm,-) = 0$ for $s > 0$. It follows that
\begin{align*}
H^t(\mathcal{K}_N,\Q(m))^\Gm &= H^0(\Gm, H^t(\mathcal{K}_N,\Q(m))) \cr 
&= H^t(\cG_N,\Q(m)) \cr
&= \Ext^t_{\MTM_N} (\Q,\Q(m)).
\end{align*}
Therefore,
\begin{align*}
    H_1(\k_N,\Q) = \prod_{m > 0} \Ext^1_{\MTM_N}(\Q,\Q(m))^\vee \otimes \Q(m).
\end{align*}

When $m \ge 2$, we have a basis of $\Ext^1_{\MTM_N}(\Q,\Q(m))$ from Proposition \ref{prop:extbasis}. One can then specify a dual basis \label{not:sigbar}
$$
\{\bsigmabar_{m,\zeta} \mid \zeta \in \bmu_N \text{ primitive}, \Im \zeta > 0\}
$$
of $H_1(\k_N)$. We can also define $\bsigmabar_{m,\zeta}$ for $\zeta$ primitive and $\Im \zeta < 0$ by the relations induced by Lemma \ref{lem:Li}: $\bsigmabar_{m,\zetabar} = (-1)^{m+1} \bsigmabar_{m,\zeta}$. We then choose {\em non-canonical} lifts $\bsigma_{m,\zeta} \in \k_N$ \label{not:sig} of each $\bsigmabar_{m,\zeta} \in H_1(\k_N)$ satisfying the same relation
\begin{equation}
    \label{eqn:sigma}
    \bsigma_{m,\zetabar} = (-1)^{m+1}\bsigma_{m,\zeta}.
\end{equation}

Suppose $\{\bv_0^\DR,\bv_m^\DR\}$ is a $\Q$-de~Rham basis of $E_{m,\zeta}$ with $\bv_m \in M_{-2m} E_{m,\zeta}$. The action of $\exp \bsigma_{m,\zeta} \in \mathcal{K}_N$ on $E_{m,\zeta}$ stabilizes the associated graded $\Gr^M_\bullet E_{m,\zeta}$. Thus, 
$$
\exp \bsigma_{m,\zeta} : 
\begin{cases}
\bv_0^\DR \longmapsto \bv_0^\DR + c \bv_m^\DR \cr
\bv_m^\DR \longmapsto \bv_m^\DR
\end{cases}
$$
for some $c \in \Q$. Scale $\bsigma_{m,\zeta}$ such that $c = 1$ and thus $\bsigma_{m,\zeta}$ acts on $E_{m,\zeta}$ by
$$
\bsigma_{m,\zeta} : 
\begin{cases}
\bv_0^\DR \longmapsto \bv_m^\DR \cr
\bv_m^\DR \longmapsto 0.
\end{cases}
$$
The action of $\bsigma_{m,\zeta}$ on basis elements $E_{m,\xi}$ with $\Im \xi > 0$ and $\xi \neq \zeta$ is trivial. 
\begin{prop}
Suppose $V_1,\ldots,V_n \in \Ext_{\MTM_N}^1(\Q,\Q(m))$. The action of $\bsigma_{m,\zeta}$ on the pull-back of $\bigoplus_k V_k$ along the diagonal map $\Delta : \Q \to \Q^n$ is the sum of the actions of $\bsigma_{m,\zeta}$ on each $V_k$ individually.
\end{prop}
\begin{proof}
Consider the diagram of de~Rham realizations below. 
$$
\xymatrix{
0 \ar[r] & (\Q(m)^\DR)^n \ar[r] \ar@{=}[d] & \bigoplus_kV_k^\DR  \ar[r] & \Q^n \ar[r] & 0 \cr
0 \ar[r] & (\Q(m)^\DR)^n \ar[r] & \Delta^{\ast}(\bigoplus_kV_k^\DR) \ar[u] \ar[r] & \Q \ar[u]_\Delta \ar[r] & 0
}
$$
The vertical maps are reaizations of morphisms in $\MTM_N$ and thus $\k_N$-equivariant. Thus, the direct sum of the actions of $\bsigma_{m,\zeta}$ on each $V_k$ pulls back to the their sum. 
\end{proof}
\begin{cor}
The action of $\bsigma_{m,\zeta}$ on $E_{m,\zetabar}$ is $(-1)^{m+1}$ times the action of the $\bsigma_{m,\zeta}$ on $E_{m,\zeta}$.
\end{cor}
\begin{proof} Consider the Baer sum $E_{m,\zeta} + E_{m,\zetabar}$ in the diagram below.
$$
\xymatrix{0 \ar[r] & \Q(m)^\DR \oplus \Q(m)^\DR \ar[r] \ar[d]_\Sigma & \Delta^{\ast}(E_{m,\zeta}\oplus (-1)^mE_{m,\zetabar}) \ar[d] \ar[r] & \Q \ar@{=}[d] \ar[r] & 0 \cr
0 \ar[r] & \Q(m)^\DR \ar[r] & E_{m,\zeta}+(-1)^mE_{m,\zetabar} \ar[r] & \Q \ar[r] & 0
}
$$
Since $E_{m,\zeta} + (-1)^mE_{m,\zetabar}$ splits as a MHS and thus also in $\MTM_N$, the action of $\k_N$ is trivial. The result follows from $\k_N$-equivariance of the vertical maps.
\end{proof}

\section{Galois representation of \texorpdfstring{$\k_N$}{TEXT}}
\label{sec:rep}

Recall action of $\k_N$ on $\p(U_N,\vv_1)$ and $\p(E_\uu',\vv_1)$ 
$$
\xymatrix{\k_N \ar[r]^-{\phi_\cyc} \ar[rd]_-{\phi_\elp} & \Der \p(U_N,\vv_1) \ar[d]^{\Psi_N} \cr & \Der \p_N(E_\uu',\vv_1),}
$$
where the Hain map $\Psi_N$ \eqref{eqn:hain} is $\k_N$-equivariant. For $\delta \in \Gr_{-2m}^M \Der \p(E_\uu',\vv_1)$, define its {\em head} to be the class $\hat{\delta} \in \Gr_{-k}^W \Gr_{-2m}^M \Der \p(E_\uu',\vv_1)$ for the least $k$ such that $\hat{\delta}$ is nonzero. The objective of this section is to compute the head of $\phi_\elp(\bsigma_{m,\zeta}) \in \Gr_{-2m}^M \Der \p(E_\uu',\vv_1)$ for each generator $\bsigma_{m,\zeta} \in \k_N$.
\begin{lemma}
\label{lem:low}
The image of $\phi_\elp$ commutes with the residue of the $\G_1(N)$ KZB connection \eqref{eqn:KZBrest} along $q = 0$.
\end{lemma}
\begin{proof}
The residue of the KZB connection form \eqref{eqn:KZBrest} along $q = 0$ spans a copy of $\Q(1)$ in $\Der \p(E_\uu',\vv_1)$. 
The action of $\k_N$ on semisimple objects of $\MTM_N$ is trivial. 
\end{proof}
Define the subalgebra of {\em special derivations}\label{not:SDer} to be
$$
\SDer \p_N := \{\delta \in \Der \p_N \mid \delta(\bt_1) = 0, \delta(\bt_\zeta) = [u_\zeta,\bt_\zeta]\text{ for some $u_\zeta \in \p_N$}\}.
$$
This is the Lie algebra of the {\em special automorphisms} of $\p_N$ given by 
$$
\SAut \p_N := \{\sigma \in \Aut\p_N \mid \sigma(\bt_1) = \bt_1,\sigma(\bt_\zeta) = e^{u_\zeta}\bt_\zeta e^{-u_\zeta} \text{ for some $u_\zeta \in \p_N$}\}.
$$
\begin{lemma}
\label{lem:ad}
The sub-VMHS of $\bP_{\G_1(N)}$ with fiber $\SDer \p_N$ restricted to the first-order neighborhood of the zero section of $\E_{\G_1(N)} \to Y_1(N)$ is admissible.
\end{lemma}
\begin{proof}
The derivations $\e_{m,\zeta}$ are special \cite[\S7.1]{hopper}, and thus the residues of the KZB connection at each cusp are also special. The residue along the zero section $w = 1$ is $\ad \bt_1$ is clearly special as well. It follows from the admissibility of KZB (Theorem \ref{thm:avmhs}) that the sub-VMHS of interest is admissible as well. 
\end{proof}
\begin{lemma}
The image of $\phi_\elp$ is contained in $\SDer \p(E_\uu',\vv_1)$.
\end{lemma}
\begin{proof}
Fix $\zeta \in \bmu_N \subset \Gm \subset E_\uu$. Let $\ell$ denote a path from $\vv_1$ to $\vv_\zeta$, and let $\kappa$ denote a small loop based at $\vv_\zeta$ traveling around $\zeta$ once counter-clockwise. Then, up to conjugation, the natural inclusion $\pi_1(E_\uu',\vv_1) \to \exp\p(E_\uu',\vv_1)$ maps $\ell\kappa\ell^{-1} \mapsto \exp\bt_\zeta$.

Let $D$ be a contractible neighborhood of $\zeta$ containing no other roots of unity. Then $\p(D - \{\zeta\},\vv_\zeta) \cong \mathbb{L}(\log \kappa)^\wedge$, which is a pro-object of $\MTM^\ss_N$. Thus, the action of $\mathcal{K}_N$ fixes $\log\kappa$ and also $\kappa$. Therefore, if $\exp \bsigma \in \mathcal{K}_N$, we have
\begin{align*}
    (\exp \bsigma) (\exp \bt_\zeta) &= \ell^\bsigma \kappa (\ell^\bsigma)^{-1} \cr
    &= \ell^\bsigma(\ell^{-1}\ell)\kappa(\ell^{-1}\ell)(\ell^{\bsigma})^{-1} \cr
    &= (\ell^\bsigma \ell^{-1}) (\exp\bt_\zeta) (\ell^\bsigma\ell^{-1})^{-1}.
\end{align*}
Note that $\ell^\bsigma\ell^{-1}$ is a loop based at $\vv_1$ and thus is identified with $e^u$ for some $u \in \p(E_\uu',\vv_1)$. Applying logarithms yields 
\begin{align*}
(\exp \bsigma) (\bt_\zeta) &= \log(e^ue^{\bt_\zeta}e^{-u}) \cr
&= \log(\Ad(e^u) \exp \bt_\zeta) \cr
&= \log(\exp \Ad(e^u) \bt_\zeta) \cr
&= \Ad(e^u) \bt_\zeta.
\end{align*}
If $\zeta = 1$ and $\ell$ is the trivial path, then $u = 0$ and $(\exp\bsigma)(\bt_1) = \bt_1$. Thus, $\mathcal{K}_N$ acts on $\p(E_\uu',\vv_1)$ via special automorphisms. Again applying logarithms shows 
$$
\bsigma(\bt_\zeta) = \log (\Ad(e^u))(\bt_\zeta) = \log(\exp (\ad u)) \cdot \bt_\zeta = [u,\bt_\zeta].
$$
Thus, $\phi_\elp(\bsigma) \in \SDer \p(E_\uu',\vv_1)$. 
\end{proof}
\begin{lemma}
$W_{-2}\SDer \p_N = M_{-2}W_{-2}\SDer\p_N$.
\end{lemma}
\begin{proof}
Suppose $\delta \in W_{-2} \SDer \p_N$. The nontrivial terms of $\delta(\bX)$ are Lie words with degree at least one in $\bY$ or $\bt_\zeta$, and thus $\delta(\bX) \in M_{-2} \p_N$. We also have $\delta(\bt_\zeta) = [u_\zeta,\bt_\zeta]$ for some $u_\zeta \in \p_N$ since $\delta$ is special. It follows from \eqref{eqn:filtrations} that if $u_\zeta \in W_{-2}$, then $u_\zeta \in M_{-2}$ as well. Thus, $\delta(\bt_\zeta) \in M_{-4}\p_N$. It remains to show $\delta(\bY) \in M_{-4}\p_N$. We know 
$$
[\bX,\delta(\bY)] = \delta([\bX,\bY]) - [\delta(\bX),\bY)].
$$
Since $[\bX,\bY] = \sum\bt_\zeta$, we have $\delta([\bX,\bY]) \in M_{-4}\p_N$. Since $\delta(\bX) \in M_{-2}\p_N$, we also know $[\delta(\bX),\bY] \in M_{-4}\p_N$. It follows that $\delta(\bY) \in M_{-4}\p_N$; hence, $\delta \in M_{-2}\SDer\p_N$.
\end{proof}
\begin{prop}
If $m \ge 2$, then $\phi_\elp(\bsigma_{m,\zeta}) \in W_{-m-1} \SDer \p(E_\uu',\vv_1)$.
\end{prop}
\begin{proof}
Set $\delta = \phi_\elp(\bsigma_{m,\zeta})$, and choose $k \in \Z$ such that $\delta \in W_{-k} \SDer \p(E_\uu',\vv_1)$ but $\delta \notin W_{-k-1} \SDer \p(E_\uu',\vv_1)$. Since $m \ge 2$, we know $\delta \in  M_{-4}\SDer \p(E_\uu',\vv_1) \subset W_{-2} \SDer \p(E_\uu',\vv_1)$; that is, $k \ge 2$. Since the sub-VMHS of special derivations is admissible (Lemma \ref{lem:ad}), it follows from Lemma \ref{lem:low} that $\delta$ is a lowest weight vector in the $\sl_2$-representation $\Gr^W_{-k} \SDer \p(E_\uu',\vv_1)$. Since $\delta \in M_{-2m} \SDer \p(E_\uu',\vv_1)$ is nontrivial (\S\ref{sec:gen}), the isomorphism \eqref{eqn:sl2} implies the corresponding {\em nonzero} highest weight vector is in $\Gr^M_{2m - 2k}\Gr_{-k}^W \SDer \p(E_\uu',\vv_1)$. Since $k \ge 2$, it follows from the previous lemma that $2m - 2k \le -2$, or equivalently, $k \ge m + 1$. 

\end{proof}
\begin{theorem}
\label{thm:head}
When $N \ge 3$ and $\zeta \in \bmu_N$ is primitive, the head of $\phi_\elp(\bsigma_{m,\zeta})$ is congruent to 
$$
\e^\op_{m+1,\zeta} + \sum_{\eta \text{ not primitive}} c_\eta \e^\op_{m+1,\eta}
$$
in $\Gr_W^{-m-2} \Der (\p(E_\uu',\vv_1)/D^2)$ for some constants $c_\eta \in \Q$.
\end{theorem}
\begin{proof}
Recall from \S\ref{sec:cycpoly} the limit MHS of the cyclotomic polylog $\p(U_N,\vv_1)/D^2$ is the pull-back of the direct product of extensions $-E_{m,\zeta}(1)$. For $\zeta \in \bmu_N$ primitive, the action of $\bsigma_{m,\zeta}$ on $-E_{m,\zeta}(1)$ is given by 
$$
\bsigma_{m,\zeta} : \ee_0 \longmapsto -\ee_0^m\cdot \ee_\zeta.
$$
(see \S\ref{sec:gen}). Since $-E_{m,\zetabar} \equiv (-1)^mE_{m,\zeta}$ as extensions of MHS, the action of $\bsigma_{m,\zeta}$ on $-E_{m,\zetabar}(1)$ is 
$$
\bsigma_{m,\zeta} : \ee_0 \longmapsto (-1)^m \ee_0^m \cdot \ee_\zetabar.
$$
Applying the Hain map, the action of $\bsigma_{m,\zeta}$ on the pull-back of the direct sum $V_{m,\zeta}\oplus V_{m,\zetabar}$ is equivalent to $\e^\op_{m+1,\zeta}$ (this is essentially the same argument as in Corollary \ref{cor:Vsplit}). Moreover, $\bsigma_{m,\zeta}$ acts trivially on $-E_{m,\xi}(1)$, and therefore $V_{m,\zeta}$, for all primitive $\xi \notin \{\zeta,\zetabar\}$. 
The constants $c_\eta$ are determined by the decomposition of $E_{m,\eta}$ for $\eta$ not primitive into a sum of extensions $E_{m,\xi}$ with $\xi$ primitive (as in Proposition \ref{prop:extbasis}). 
Finally, since each derivation $\e^\op_{m+1,\zeta}$ has $W$-weight $-m-1$, it follows from the previous result that this is the head of $\phi_\elp(\bsigma_{m,\zeta})$.
\end{proof}
Since the derivations $\e_{m+1,\zeta}$ are linear independent, it follows that the representation $\phi_\elp$ is injective in depth 1. Moreover, the map is remains injective after mapping to the quotient where $\bt_\zeta = 0$ for all $\zeta \in \bmu_N$ not primitive. Thus, the quadratic relations between the generators $\bsigma_{m,\zeta}$ (with $\zeta \in \bmu_N$ primitive) in depth 2 correspond exactly to quadratic relations between the derivations $\e_{m+1,\zeta}$ with $\zeta$ primitive. 

Finally, we include calculations of the constants $c_\eta$ in Theorem \ref{thm:head} for values of $N$ where the formulas are relatively simple. 
\begin{theorem}
When $N = p^n$ is a prime power and $\zeta$ is a primitive $N$th root of unity, the head of $\phi_\elp(\bsigma_{m,\zeta})$ is congruent to 
$$
\e^\op_{m+1,\zeta} + \frac{N^{m-1}}{1-p^{m-1}} \e^\op_{m+1,1} + \sum_{k=1}^{n-1} p^{k(m-1)}\e^\op_{m+1,\zeta^{p^k}} 
$$ 
in $\Gr_W^{-m-2} \Der (\p(E_\uu',\vv_1)/D^2)$.
\end{theorem}
\begin{proof}
By the previous statement, $\e^\op_{m+1,\zeta}$ is the only derivation indexed by a primitive root of unity to appear in the head. For $1 \leq k < n$, one deduces from \eqref{eqn:ppdist} that $\bsigma_{m,\zeta}$ acts on $-E_{m,\zeta^{p^k}}(1)$ by
$$
\bsigma_{m,\zeta} : \ee_0 \longmapsto -p^{k(m-1)} \ee_0^m\cdot \ee_{\zeta^{p^k}}.
$$
Similarly, the action on $-E_{m,\zetabar^{p^k}}(1)$ is
$$
\bsigma_{m,\zeta} : \ee_0 \longmapsto (-1)^m p^{k(m-1)} \ee_0^m\cdot \ee_{\zetabar^{p^k}}.
$$
The extension $E_{m,1}$ decomposes as a $\Q$-linear combination of $E_{m,\zeta}$ with $\zeta$ primitive (see Example \ref{ex:E1}). It follows that $\bsigma_{m,\zeta}$ acts on $-E_{m,1}(1)$ by 
$$
\bsigma_{m,\zeta} : \ee_0 \longmapsto -\frac{N^{m-1}}{1-p^{m-1}}\left(1 + (-1)^m\right)\ee_0^m \cdot \ee_1. 
$$
Applying the Hain map yields the result.
\end{proof}

\begin{cor}
When $N$ is prime, the head of $\phi_\elp(\bsigma_{m,\zeta})$ reduces to 
$$
\phi_\elp(\bsigma_{m,\zeta}) \equiv \e^\op_{m+1,\zeta} + \frac{1}{N^{1-m}-1} \e^\op_{m+1,1}.
$$
\end{cor}
\begin{remark}
When $N = 1$, one simply has $\phi_\elp(\bsigma_{m,1}) \equiv \frac{1}{2}\e^\op_{m+1,1}$. This is consistent with \cite[Prop. 29.4]{hain:MEM}. When $N = 2$, one can compute 
$$
\phi_\elp(\bsigma_{m,-1}) \equiv \frac12 \e^\op_{m+1,-1} + \frac{1}{2(N^{1-m}-1)} \e^\op_{m+1,1}.
$$
Note that these images are trivial when $N = 1,2$ and $m$ is even.
\end{remark}

To determine the head of each $\phi_\elp(\bsigma_{m,\zeta})$ when $N$ is not a prime power, one must decompose each $E_{m,\zeta}$ in terms of the basis of $\Ext_{\MTM_N}^1(\Q,\Q(m))$ in Proposition \ref{prop:extbasis}. The computation can be arduous for large $N$, so we will only compute the heads of $\phi_\elp$ when $N = 6$ (the only one of Deligne's exceptional values \cite{deligne:23468} that is not a prime power).

\begin{example}
Fix a primitive sixth root of unity $\zeta$. It follows from Lemma \ref{lem:Li} that $\Li_m(\zeta^5) \equiv (-1)^{m+1} \Li_m(\zeta)$ and $\Li_m(\zeta^4) \equiv (-1)^{m+1}\Li_m(\zeta^2) \bmod (2\pi i)^m\Q$. One can use the distribution relation \eqref{eqn:dist} to verify $\Li_m(-1) = (2^{1-m} - 1)\Li_m(1)$ and $\Li_m(\zeta) \equiv (2^{1-m} + (-1)^m)\Li_m(\zeta^2) \bmod (2\pi i)^m \Q.$  Thus, 
\begin{align*}
    \zeta(m) &= 6^{m-1} \sum_{\eta \in \bmu_6} \Li_m(\eta) \cr
    &\equiv 6^{m-1} \left(2^{1-m}\zeta(m) + \frac{1 + (-1)^{m+1}}{1-2^{m-1}}\Li_m(\zeta) \right) \bmod (2\pi i)^m\Q.
\end{align*}
Rearranging yields
$$
\zeta(m) \equiv (1 + (-1)^{m+1})\frac{6^{m-1}}{(1-3^{m-1})(1 - 2^{m-1})}\Li_m(\zeta) \bmod (2\pi i)^m\Q. 
$$
 Thus, the action of $\phi_\cyc(\bsigma_{m,\zeta})$ on $\p(U_6,\vv_1)$ is
\begin{align*}
\phi_\cyc(\bsigma_{m,\zeta}) : \ee_0 \longmapsto -\ee_0^m & \cdot (\ee_\zeta + (-1)^{m + 1} \ee_{\zeta^5}) \cr
&- \frac{1}{2^{1-m} + (-1)^m} \ee^m_0 \cdot (\ee_{\zeta^2} +(-1)^{m+1} \ee_{\zeta^4}) \cr
&- (1 + (-1)^{m+1})\frac{6^{m-1}}{(1-3^{m-1})(1 - 2^{m-1})}\ee_0^m \cdot \ee_1 \cr
&- (1 + (-1)^{m+1})\frac{1}{(3^{1-m}-1)}\ee_0^m \cdot \ee_{-1}.
\end{align*}
Finally, apply the Hain map to see
\begin{align*}
    \phi_\elp(\bsigma_{m,\zeta}) \equiv \e^\op_{m+1,\zeta} &+ \frac{1}{2^{1-m} + (-1)^m} \e^\op_{m+1,\zeta^2} \cr
    &+ \frac{6^{m-1}}{(1-3^{m-1})(1 - 2^{m-1})} \e^\op_{m+1,1} \cr
    &+ \frac{1}{(3^{1-m}-1)} \e^\op_{m+1,-1}.
\end{align*}
mod $W_{-m-2}\Der (\p(E_\uu',\vv_1)/D^2)$. 
\end{example}

\section{Hecke action on $K_{2m - 3}(\cO_N) \otimes \Q$}

The Galois representation of $H_1(\k_N)$ together with the KZB connection induces a Hecke action on the odd $K$-groups $K_{2m-3}(\cO_N) \otimes \Q$. 

\subsection{Coefficients of KZB}

Let $\Pol^\elp_{N,\vec{1}}$ denote the linearization of $\Pol^\elp_N$ along the section $z = 0$ (see Appendix \ref{sec:transport}). The connection on $\Pol^\elp_{N,\vec{1}}$ is given by 
$$
\nabla = d + \left(\bY \frac{\partial}{\partial \bX} - \frac12 \sum_{\substack{m \ge 2\\\zeta \in \bmu_N}} \frac{m - 1}{(2\pi i)^{m}} G_{m,\zeta}(\tau) \e_{m,\zeta}\right)\frac{dq}{q} + \bt_1 \frac{dz}{z},
$$
where the coefficients are Eisenstein series
$$
G_{m,\zeta}(\tau) = \sum_{\substack{k,\ell \in \Z \\ (k,\ell) \neq (0,0)}} \frac{\zeta^k}{(k\tau + \ell)^m}
$$
(see \cite[\S12]{hopper} for background on the formula for $\nabla$). Set $\E_{m,N} = \Span_\Q \{G_{m,\zeta} \mid \zeta \in \bmu_N \text{ primitive}\}$. When $m \ge 3$, the KZB connection determines (up to constant multiple) a canonical map
\begin{align}
\begin{split}
    \label{eqn:heckeMap}
    \psi_{m,N} : \mathcal{E}_{m,N} &\longrightarrow \Ext^1_{\MTM_N}(\Q,\Q(m-1)) \cr 
    G_{m,\zeta} &\longmapsto E_{m-1,\zeta}
\end{split}
\end{align}
by restricting $\Pol^\elp_{N,\vec{1}}$ to the subquotient spanned by $\bY$ and $\bY^{m-1}\cdot \bt_\zeta$ and computing the limit MHS at $q = 0$ (as in \S\ref{sec:lmhsPolEll}). The maps $\psi_{m,N}$ respect shifts in level via
\begin{align*}
    [d] : \mathcal{E}_{m,N} &\longrightarrow \mathcal{E}_{m,dN} \cr 
    f(\tau) &\longmapsto d^{m-1}f(d\tau).
\end{align*}
This is because
$$
[d](G_{m,\zeta}) = d^{m-1}G_{m,\zeta}(d\tau) = d^{m-2}\sum_{\xi^d = \zeta} G_{m,\xi}(\tau),
$$
which is analogous to the distribution relation
$$
\Li_{m-1}(\zeta) = d^{m-2} \sum_{\xi^d = \zeta} \Li_{m-1}(\zeta).
$$
That is, the diagram 
$$
\xymatrix{\E_{m,N} \ar[r]^-{\psi_{m,N}} \ar[d]_{[d]} & \Ext^1_{\MTM_N}(\Q,\Q(m-1)) \ar@{^{(}->}[d] \cr \E_{m,dN} \ar[r]_-{\psi_{m,dN}} & \Ext^1_{\MTM_{dN}}(\Q,\Q(m-1))}
$$
commutes.

\subsection{Hecke action}
\label{sec:hecke}

For a sublattice $\Lambda \subset \Lambda_\tau = \Z\tau \oplus \Z$, set 
$$
G_{m,\zeta}(\Lambda) = \sum_{\substack{(k,\ell) \neq (0,0) \\ k\tau + \ell \in \Lambda}} \frac{\zeta^k}{(k\tau + \ell)^m}.
$$
Fix a prime $p \nmid N$. Then 
\begin{align*}
    T_pG_{m,\zeta}(\tau) &= p^{m-1}\sum_{[\Lambda_\tau : \Lambda] = p} G_{m,\zeta}(\Lambda) \cr 
    &= p^{m-1}\left(G_{m,\zeta}(\tau) + pG_{m,\zeta}(p\Lambda_\tau)\right) \cr 
    &= p^{m-1}G_{m,\zeta}(\tau) + G_{m,\zeta^p}(\tau).
\end{align*}
Observe that $\E_{m,N}$ is $T_p$-invariant and that $T_p$ and $[d]$ commute when $p \nmid dN$. 

Define an analogous operator on $\Ext_{\MTM_N}(\Q,\Q(m))$ given by 
\begin{equation}
    \label{eqn:heckeLi}
    T_p: \Li_{m}(\zeta) \longmapsto p^{m}\Li_{m}(\zeta) + \Li_{m}(\zeta^p)
\end{equation}
so that $T_p$ and $\psi_{m,N}$ commute. Since $p\nmid N$, the operator $T_p$ respects all distribution relations between the periods $\Li_m(\zeta)$ with $\zeta \in \bmu_N$ and thus is well-defined.
\begin{theorem}
    When $m \ge 3$, there is a natural identification between $\E_{m,N}^\vee$ and $K_{2m - 3}(\cO_N) \otimes \Q$ compatible with changes in level, which induces a natural action of the prime to $N$ Hecke operators on $K_{2m - 3}(\cO_N) \otimes \Q$.
\end{theorem}
\begin{proof}
From Deligne and Goncharov, we know
\begin{equation}
\label{eqn:Kisom}
\Ext^1_{\MTM_N}(\Q,\Q(m-1)) \cong \Gr^W_{2m-2} H^1(\k_N) \cong K_{2m - 3}(\cO_N) \otimes \Q. 
\end{equation}
It follows that there is a commutative diagram 
$$
\xymatrix{\Gr_{-2m + 2}^WH_1(\k_N) \ar[r]^-{\psi_{m,N}^\vee} \ar[d]_{\phi_\mathrm{ell}} & \mathcal{E}_{m,N}^\vee \ar[d] \cr \SDer \p_N/D^2 \ar[r]_\op & \SDer \p_N/D^2,}
$$
where $\psi_{m,N}^\vee$ is the dualization of \eqref{eqn:heckeMap}, the right vertical map is induced by KZB, and the lower horizontal map is \eqref{eqn:eop}. The map $\psi_{m,N}^\vee$ is injective since $\phi_\mathrm{ell}$ and $\op$ are injective. Thus, $\psi_{m,N}$ is an isomorphism since $\Gr^W_{-2m+2} H_1(\k_N)$ and $\E^\vee_{m,N}$ have the same dimension. Compatibility with changes in level follows from the fact that $[d]$, $\psi_{m,N}$, and $T_p$ commute when $p\nmid N$ is prime. Finally, the natural action of $T_p$ with $p \nmid N$ on $K_{2m - 3}(\cO_N) \otimes \Q$ is given by \eqref{eqn:heckeLi} via the isomorphism \eqref{eqn:Kisom}.
\end{proof}


\appendix

\part*{Appendices}


\section{Tannakian categories}
\label{sec:tannaka}
We will give a very brief description of tannakian categories. For a full discussion including proof of Theorem \ref{thm:tannaka}, we refer the reader to \cite{delmilne}.

Suppose $F$ is a field of characteristic zero and $G$ is an affine group scheme over $F$. Denote the category of $F$-linear representations of $G$ by $\Rep_F(G)$ and the sub-category of finite-dimensional representations by $\Rep_F^\fte(G)$.
\begin{definition}
A {\em neutral tannakian category $\mathcal{C}$} over $F$ is a category equivalent to one of the form $\Rep_F^\fte(G)$ for some affine $F$-group $G$.
\end{definition}
Suppose $(\mathcal{C},\otimes)$ is a rigid abelian $F$-linear tensor category with identity object $\mathds{1}$ such that $\End_{\mathcal{C}}(\mathds{1}) = F$. Suppose $\omega : \mathcal{C} \to \Vec_F$ is a faithful $F$-linear exact tensor functor, which we call a {\em fiber functor}. Given fiber functors $\omega_1$ and $\omega_2$, a {\em natural isomorphism} $\eta$ from $\omega_1$ to $\omega_2$ is a family of isomorphisms $\eta_V : \omega_1(V) \to \omega_2(V)$ indexed by $V \in \Ob (\mathcal{C})$ such that $\eta_{V_2}  \circ \omega_1(f) = \omega_2(f) \circ \eta_{V_1}$ for all morphisms $f \in \Hom_{\mathcal{C}}(V_1,V_2)$. A natural isomorphism from a fiber functor $\omega$ to itself is called a {\em natural automorphism}. Let $\Aut^\otimes(\omega)$ denote the set of natural automorphisms of $\omega$ compatible with the tensor product.
\begin{theorem}[{{\cite[Thm 3.2]{delmilne}}}]
\label{thm:tannaka}
If $\omega$ is a fiber functor of a neutral tannakian category, then $\Aut^\otimes(\omega)$ is represented by an affine group scheme $G$ over $F$ and there is an equivalence of categories $\mathcal{C} \to \Rep_F^\fte(G)$ corresponding to $\omega$.
\end{theorem}
We refer to affine group scheme $\Aut^\otimes(\omega)$ as the {\em fundamental group of $\mathcal{C}$ with respect to $\omega$}, denoted $\pi_1(\mathcal{C},\omega)$. When the choice of fiber functor is clear, we sometimes omit $\omega$ from the notation. Given two fiber functors $\omega_1$ and $\omega_2$, we define $\Isom^\otimes(\omega_1,\omega_2)$ as the set of natural isomorphisms from $\omega_1$ to $\omega_2$. It has the structure of a right $\pi_1(\mathcal{C},\omega_1)$-torsor and left $\pi_1(\mathcal{C},\omega_2)$-torsor.

\section{Admissible variations of MHS}
\label{sec:avmhs}

We assume familiarity with mixed Hodge structures (MHS). Here we give the definition of admissible variations of MHS and a standard example. The reader can find relevant background in \cite{SZ,kashiwara}.

Suppose $X$ is a smooth projective variety over $\C$ and $D$ is a divisor with normal crossings in $X$. Let $Y = X - D$. Relevant examples are
\begin{itemize}
    \item where $Y$ is a modular curve, $X$ is its natural compactification, and $D$ is the set of cusps; and
    \item where $Y$ is the universal elliptic curve $\E_\G'$ over the modular curve $Y_\G$ with single-valued $N$-torsion removed, $X$ is the compactification $\overline{\E}_\G$ over $X_\G$, and $D$ is the union of the set of single-valued $N$-torsion sections and singular fibers over the cusps of $Y_\G$.
\end{itemize}

Let $\V$ be a $\Q$-local system of finite rank over $Y$ with unipotent local monodromy at every smooth point of $D$. Let $\cV = \V \otimes_\Q \cO_X$ be the associated flat vector bundle. Denote Deligne's canonical extension of $\cV$ to $X$ by $\overline{\cV}$. Then $\overline{\cV}$ has natural connection 
$$
\nabla : \overline{\cV} \to \overline{\cV} \otimes \Omega_X^1(\log D)
$$
with logarithmic singularities along $D$. Since the local monodromy operators are unipotent, the residues of $\nabla$ at each smooth point of $D$ are nilpotent.
\begin{definition}
A {\em variation of MHS} (VMHS) $\V$ over $Y$ consists of a local system $\V_\Q$ over $Y$ of finite dimensional rational vector spaces, together with
\begin{enumerate}[(i)]
    \item a finite increasing filtration $W_\bullet$ of $\V_\Q$ by $\Q$-local systems 
    $$
    0 \subseteq W_a \V \subseteq \cdots \subseteq W_{r - 1}\V \subseteq W_r\V \subseteq \cdots \subseteq W_b\V = \V,
    $$
    and
    \item a finite decreasing filtration $F^\bullet$ of $\cV$ by holomorphic subbundles.
\end{enumerate}
These are required to statisfy
\begin{enumerate}[(i)]
    \item {\em Griffiths' transversality}
    $$
    \nabla(F^p\cV) \subseteq F^{p-1}\cV\otimes \Omega^1_Y = F^p(\cV \otimes \Omega^1_Y).
    $$
    and
    \item The fiber $V_y$ above any point $y \in Y$ is a MHS with weight and Hodge filtrations cut out by $W_\bullet$ and $F^\bullet$, respectively.
\end{enumerate}
\end{definition}
\begin{definition}
\label{def:admissible}
Suppose that $X$ is a curve. A VMHS $\V$ over $Y$ is {\em admissible} if the following additional conditions hold. 
\begin{enumerate}[(i)]
    \item The subbundles $F^p\cV$ extend to holomorphic subbundles of the canonical extension $\overline{\cV}$. 
    \item For $P \in D$, let $L_P = -\Res_P \nabla$ and $V_P$ be the fiber of $\overline{\cV}$ above $P$. There exists an increasing {\em relative weight filtration} $M_\bullet$ such that 
    \begin{enumerate}[(a)]
        \item $L_P(M_rV_P) \subseteq M_{r - 2}V_P$ and $L_P(W_mV_P) \subseteq W_mV_P$ for all $m$ and $r$, and
        \item $L_P^r$ induces an isomorphism 
        \begin{equation}
            \label{eqn:sl2}
            L_P^r : \Gr^M_{m + r} \Gr^W_mV_P \to \Gr_{m - r}^M \Gr^W_mV_P
        \end{equation}
        for all $m$ and $r$. 
    \end{enumerate}
\end{enumerate}
\end{definition}
If $X$ is a curve and $\V$ is admissible, the fibers $V_P$ over $P \in D$ have canonical {\em limit MHS} for each choice of tangent vector $\vv \in T_PX$. We typically denote this MHS by $V_{P,\vv}$ (or $V_\vv$ if the choice of point is clear). The weight and Hodge filtrations of $V_{P,\vv}$ are $M_\bullet$ and the restriction of $F^\bullet$ to $V_P$, respectively. Meanwhile, the $\Q$-structure of $V_{P,\vv}$ is determined by the elements 
\begin{equation}
    \label{eqn:Qstructure}
    \lim_{t \to 0} t^{-L_P} v(t) \in V_P, 
\end{equation}
where $t$ is a local holomorphic coordinate of $X$ centered at $P$ such that $\vv = \partial/\partial t$ and $v(t)$ is a local flat section of $\V_\Q$.


\begin{remark}
$W_m V_{P,\vv}$ is a sub-MHS of $V_{P,\vv}$ for all $m$.
\end{remark}
\begin{remark}
When $\dim X > 1$, a VMHS $\V$ is said to be admissible if its restriction to any curve is admissible \cite{kashiwara}.
\end{remark}

\section{Iterated integrals}

Suppose $X$ is a manifold, $\omega_1,\ldots,\omega_m$ are smooth 1-forms on $X$, and $\beta : [0,1] \to X$ is a piecewise smooth path. Define the {\em iterated integral} 
$$
\int_\beta \omega_1 \cdots \omega_m := \int\limits_{0 \le t_1 \le \cdots \le t_m \le 1} f_1(t_1) \cdots f_m(t_m) \, dt_1 \cdots dt_m,
$$
where $\beta^\ast\omega_k = f_k(t) \, dt$. These integrals satisfy an inversion property 
\begin{equation}
    \label{eqn:inversion}
    \int_{\beta^{-1}} \omega_1 \cdots \omega_m = (-1)^m \int_\beta \omega_m\cdots \omega_1
\end{equation}
and the {\em shuffle product}
\begin{equation}
    \label{eqn:shuffle}
    \int_\beta \omega_1 \cdots \omega_m\int_\beta \omega_{m+1} \cdots \omega_{m + n} = \sum_\sigma \int_\beta \omega_{\sigma(1)} \cdots \omega_{\sigma(m + n)},
\end{equation}
where $\sigma$ are permutations of $\{1,\ldots,m+n\}$ that preserve the order of $\{1,\ldots, m\}$ and $\{m+1,\ldots,m+n\}$. By convention, the iterated integral with empty integrand is 1. See \cite{chen,hain:bowdoin} for more details. 


\subsection{Cyclotomic multiple $\zeta$-values}
\label{sec:Nmzv}

The {\em multiple polylogarithms} are analytic functions given by the power series \label{not:polylog}\eqref{eqn:Lisum}, where $n_1,\ldots,n_m$ are positive integers and $n_m \geq 2$. The series converges in the closed polydisk $|z_i| \le 1$ for all $i  = 1,\ldots,m$. Evaluated at $z_1 = \cdots = z_m = 1$, the multiple polylogarithms yield the classical {\em multiple $\zeta$-values} (MZVs)
$$
\zeta(n_1,\ldots,n_m) := \sum_{0<k_1<k_2<\cdots<k_m} \frac{1}{k_1^{n_1}k_2^{n_2}\cdots k_m^{n_m}},
$$
where $m$ is the {\em depth} and the sum $n_1 + \cdots + n_m$ is the {\em weight} of the MZV. The values of the multiple polylogarithms at $N$th roots of unity are called {\em $N$-cyclotomic MZVs} and have similar convention for length and depth. Notably, the multiple polylogarithms may be written as iterated integrals.
\begin{prop}
\begin{multline}
    \label{eqn:Li}
    \Li_{n_1,\ldots,n_m}(z_1,\ldots,z_m) = \cr  \int_0^{z_1 \cdots z_m} \frac{dt}{1 - t} \underbrace{\frac{dt}{t} \cdots \frac{dt}{t}}_{\substack{(n_1 - 1) \\ \text{times}}}\frac{dt}{z_1 - t} \underbrace{\frac{dt}{t} \cdots \frac{dt}{t}}_{\substack{(n_2 - 1) \\ \text{times}}} \cdots \frac{dt}{z_1\cdots z_{m-1} - t} \underbrace{\frac{dt}{t} \cdots \frac{dt}{t}}_{\substack{(n_m - 1) \\ \text{times}}}.
\end{multline}
\end{prop}
\begin{proof}
Express each term $\frac{1}{z - t}$ as the power series $\frac{1}{z}\sum_{r \ge 0}(t/z)^r$ and evaluate directly to recover the sum \eqref{eqn:Lisum}.
\end{proof}

\subsection{Chen's transport formula} 

Suppose $V \times X$ is a trivial vector bundle over a manifold $X$ with connection $\nabla = d + \omega$, where $\omega$ is a smooth 1-form on $X$ taking values in $\End V$. 
\begin{prop}
Suppose $\beta : [0,1] \to X$ is a piecewise smooth path. Then the {\em inverse parallel transport} of $V \times X \to X$ with respect to $\nabla$ along $\beta$ is given by 
\begin{equation*}
    \label{eqn:chenT}
    T(\beta)^{-1} = 1 + \int_\beta \omega + \int_\beta \omega\omega + \int_\beta \omega \omega \omega + \cdots.
\end{equation*}
\end{prop}
This is the inverse of Chen's formula \cite[\S 3]{chen}. See \cite[Lemma 2.5]{hain:bowdoin} and \cite[Lemma 5.4]{hain:kzb} for a proof.

If the connection $\nabla$ is flat, then $T(\beta)^{-1}$ depends only on the homotopy class of $\beta$. Using the topologist's convention for path multiplication, inverse transport is multiplicative: $T(\alpha \beta)^{-1} = T(\alpha)^{-1}T(\beta)^{-1}$. This induces the following property of iterated integrals. 
\begin{prop}
\label{prop:shuffle}
Suppose $\omega_1,\ldots,\omega_m \in E^1(X)$ and $\alpha$ and $\beta$ are paths in $X$ such that $\alpha(1) = \beta(0)$. Then 
$$
\int_{\alpha\beta}\omega_1 \cdots \omega_m = \sum_{r = 0}^m \int_\alpha \omega_1 \cdots \omega_r \int_\beta \omega_{r+1}\cdots \omega_m.
$$
\end{prop}

\subsection{Regularized iterated integrals}
\label{sec:reg}

Suppose $X$ is a smooth curve. Fix a point $P \in X$ and suppose $\omega_1,\ldots,\omega_r \in H^0(X,\Omega^1_X(\log P))$ are holomorphic 1-forms on $X-P$ with at worst logarithmic singularities at $P$. Fix a tangent vector $\vv \in T_PX$. We would like to define iterated integrals of the form
$$
\int_\vv^Q \omega_1 \cdots \omega_r,
$$
where $Q \in X$. Integrals of this form are said to be {\em regularized at $\vv$}. 

For simplicity, we will let $X = \C$ and $P$ be the origin as this is the only case needed in \S\ref{sec:polMHS}. Moreover, the definition is easily generalized to any smooth curve $X$ by choosing a holomorphic coordinate centered at $P$. Consider the trivial bundle $V \times \C \to \C$ with flat connection $\nabla = d + \Omega$ where $\Omega \in H^0(\C,\Omega^1(\log 0)) \otimes \End V$ and $L = -\Res_0\Omega$ is nilpotent (e.g. $V$ underlies an admissible VMHS over $\C$). 

Set $\Omega = \omega_1A_1 + \cdots +\omega_rA_r$, where each $\omega_j \in H^0(\C,\Omega^1(\log 0))$ and $A_j \in \End V$. 
Meanwhile, set $\overline{\omega}_j = \Res_0 \omega_j \frac{dz}{z}$ and $\overline{\Omega} = \overline{\omega}_1 A_1 + \cdots + \overline{\omega}_rA_r$. Observe that transport with respect to the linearized connection $d + \overline{\Omega}$ from $z\in \C - \{0\}$ to $\lambda \in \C - \{0\}$ is given by 
\begin{equation}
\begin{aligned}
\label{eqn:Tmono}
    T_{\overline{\Omega}}([z,\lambda]) = T_{\overline{\Omega}}([\lambda,z])^{-1} &= 1 + \int_\lambda^z \overline{\Omega} + \int_\lambda^z \overline{\Omega}\,\overline{\Omega} + \cdots \cr
    &= 1 + \sum_{m \ge 1} \frac{(\log z/\lambda)^m}{m!}(-L)^m \cr
    &= (z/\lambda)^{-L}.
\end{aligned}
\end{equation}
Fix a point $Q \in \C -\{0\}$. The {\em regularized iterated integrals} from $\lambda\partial/\partial z$ to $Q$ are the coefficients of the transport \eqref{eqn:Qstructure} from $\lambda\partial/\partial z$ to $Q$:
$\lim_{z \to 0} (z/\lambda)^{-L} T_\Omega([Q,z/\lambda])$. Using the computation \eqref{eqn:Tmono}, the coefficient of the $A_{\ell_1}\cdots A_{\ell_n}$ term is given by 
$$
\int_{\lambda\partial/\partial z}^Q \omega_{\ell_1} \cdots \omega_{\ell_n} := \lim_{z \to 0} \sum_{j = 1}^n \int_\lambda^z \overline{\omega}_{\ell_1} \cdots \overline{\omega}_{\ell_j} \int_z^Q \omega_{\ell_{j+1}} \cdots \omega_{\ell_n}.
$$
This is the same formula as Brown's ``mortar board regularization'' \cite[\S4]{brown:MMV}. Concretely, the regularization of an iterated integral at $\partial/\partial z\in T_0\C$ has the effect of setting $\lim\limits_{z \to 0} \log z$ to zero. 

\begin{example}
\label{ex:Emhs}
The fibers of the logarithm variation $\bE$ over $\C - \{0\}$ are extensions of $\Z$ by $\Z(1)$. The fiber over $z$ has rational basis $\{2\pi i\ee_1, \ee_0 - (\log z)\ee_1\}$ where $\ee_0$ and $\ee_1$ are the de~Rham generators of $\Z$ and $\Z(1)$, respectively. The fiber at $\pm 1$ splits over $\Z$ and the fibers at roots of unity split over $\Q$. The associated holomorphic vector bundle $\cV$ has connection $\nabla = d + L \frac{dz}{z}$, where $L$ is the nilpotent endomorphism
$$
L : \left\{ 
\begin{array}{rcl}
     \ee_0 & \longmapsto & \ee_1  \\
     \ee_1 & \longmapsto & 0.
\end{array}
\right.
$$
Transport from $Q \in \C -\{0\}$ to $\lambda \partial/\partial z$ is given by 
\begin{align*}
    T([Q,\lambda\partial/\partial z]) = T([\lambda \partial/\partial z,Q])^{-1} &= 1 + \sum_{n \ge 1} L^n\int_{\lambda\partial/\partial z}^Q \underbrace{\frac{dz}{z} \cdots \frac{dz}{z}}_{n} \cr
    &= 1 + L\int_{\lambda\partial/\partial z}^Q \frac{dz}{z} \cr
    &= 1 + L\lim_{z \to 0} \left(\int_\lambda^z \frac{dz}{z} + \int_z^Q \frac{dz}{z}\right) \cr
    &= 1 + L(\log Q - \log \lambda).
\end{align*}
Observe that the $\log z$ terms canceled. Setting $Q = 1$, we also observe the limit MHS at $\lambda\partial/\partial z$ splits over $\Z$ if $\lambda = \pm 1$ and over $Q$ if $\lambda$ is a root of unity. In fact, the limit variation over $T_0\C - \{0\}$ is isomorphic to $\bE$ via the natural identification $T_0\C \cong \C$.


\end{example}

\section{Regularized transport of linearized connections}
\label{sec:transport}

Let $\cV$ be a trivial vector bundle $V \times \C^2 \to \C^2$ with flat connection $$
\nabla = d +\Omega = d +  A(x,y)\frac{dx}{x} + B(x,y)\frac{dy}{y}.
$$
The functions $A$ and $B$ are holomorphic and take values in $\End V$. Let $\nabla_0$ denote the linearized connection along the divisor $y = 0$
$$
\nabla_0 = d + \Omega_0 = d + A(x,0)\frac{dx}{x} + B(x,0) \frac{d\lambda}{\lambda}.
$$
This defines a connection on the normal bundle of $y = 0$, where the point $(x,\lambda)$ is the vector $\lambda \partial/\partial y \in T_{(x,0)} \C^2$. 

Further, suppose the residues $A(0,y)$ and $B(x,0)$ are nilpotent. Then $\cV$ is Deligne's canonical extension of the local system of flat sections of $\cV$ over $\C^\times \times \C^\times$ \cite{deligne:can}. Thus, there is holomorphic change of gauge $h : \C^2 \to \Aut V$ with $h(0,0) = \id$ such that 
$$
\widetilde{\Omega} = h\Omega h^{-1} - (dh)h^{-1} = A(0,0)\frac{dx}{x} + B(0,0)\frac{dy}{y}.
$$
For convenience, set $A = A(0,0)$ and $B(0,0)$. Since the connection is flat, we have $[A,B] = 0$. Computing residues yields
$$
h(x,0)B(x,0)h(x,0)^{-1} = B(0,0).
$$
Thus, we also have 
$$
\widetilde{\Omega} = h(x,0)\Omega_0 h(x,0)^{-1} - (dh(x,0))h(x,0)^{-1}.
$$
\begin{prop}
Regularized transport with respect to $\nabla$ from $\v_0 = \lambda_0\partial/\partial y$ anchored at $(x_0,0)$ to $\v_1 = \lambda_1\partial/\partial y$ anchored at $(x_1,0)$ is equal to transport with respect to $\nabla_0$ from $(x_0,\lambda_0)$ to $(x_1,\lambda_1)$. 
\end{prop}

    
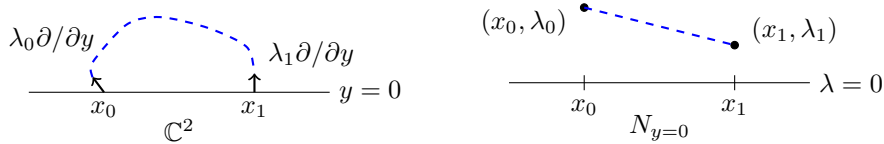
\begin{tikzpicture}[scale = 1]
        \draw (-4,0) -- (0,0);
        \draw (0,0) node[right] {$y = 0$};
        \draw (-2,-.5) node {$\C^2$}; 
        \draw[thick,->] (-3,0) -- (-3.15,.2);
        \draw[thick,->] (-1,0) -- (-1,.25);
        \draw (-3,0) node [below] {$x_0$};
        \draw (-1,0) node [below] {$x_1$};
        \draw (-3.7,.4) node[above] {$\lambda_0 \partial /\partial y$};
        \draw (-.25,.2) node[above] {$\lambda_1 \partial /\partial y$};
        \draw[dashed,blue,thick] plot [smooth] coordinates {(-3.15,.2) (-3.15, .4) (-2.7, .9) (-2.25, 1) (-1.6, .85) (-1.1, .6) (-1, .25)};
    \end{tikzpicture}
    \qquad
    \begin{tikzpicture}[scale = 1]
        \draw (-4,0) -- (0,0);
        \draw (0,0) node[right] {$\lambda = 0$};
        \draw (-2,-.6) node {$N_{y = 0}$}; 
        \draw[fill=black] (-3, 1) circle (.05cm);
        \draw[fill=black] (-1, .5) circle (.05cm);
        \draw (-3.1,.8) node[below,left] {$(x_0,\lambda_0)$};
        \draw (-.9,.7) node[above,right] {$(x_1,\lambda_1)$};
        \draw (-3,.1) -- (-3,-.1) node[below] {$x_0$};
        \draw (-1,.1) -- (-1,-.1) node[below] {$x_1$};
        \draw[dashed,blue,thick] (-3,1) -- (-1,.5);
    \end{tikzpicture}
\begin{proof}
Apply the definition of regularized transport and use the facts above.
\begin{align*}
    T_\Omega(\v_0,\v_1) &= \lim_{y \to 0} (y/\lambda_1)^{B(x_1,0)} T_\Omega((x_0,y),(x_1,y))(y/\lambda_0)^{-B(x_0,0)} \cr
    &= \lim_{y \to 0} (y/\lambda_1)^{B(x_1,0)} h^{-1}(x_1,y)T_{\widetilde{\Omega}}((x_0,y),(x_1,y))h(x_0,y)(y/\lambda_0)^{-B(x_0,0)} \cr 
    &= \lim_{y \to 0} (y/\lambda_1)^{B(x_1,0)} h(x_1,0)^{-1}(x_0/x_1)^{-A}h(x_0,0)(y/\lambda_0)^{-B(x_0,0)} \cr 
    &= h(x_1,0)^{-1}\left(\lim_{y \to 0} (y/\lambda_1)^{B} (x_1/x_0)^{-A}(y/\lambda_0)^{-B}\right)h(x_0,0) \cr
    &= h(x_1,0)^{-1}T_{\widetilde{\Omega}}((x_0,\lambda_0),(x_1,\lambda_1))h(x_0,0) \cr
    &= T_{\Omega_0}((x_0,\lambda_0),(x_1,\lambda_1))
\end{align*}
\end{proof}
\begin{cor}
Regularized transport from $\lambda_0\partial/\partial y \in T_{(x_0,0)}\C^2$ to 
$$
\lambda_1\partial/\partial y + \mu\partial/\partial x \in T_{(0,0)}\C^2
$$ 
with respect to $\nabla$ is equal to regularized transport from $(x_0,\lambda_0)$ to $\mu \partial/\partial \lambda \in T_{(0,\lambda_1)} \C$ with respect to $\nabla_0$.
\end{cor}
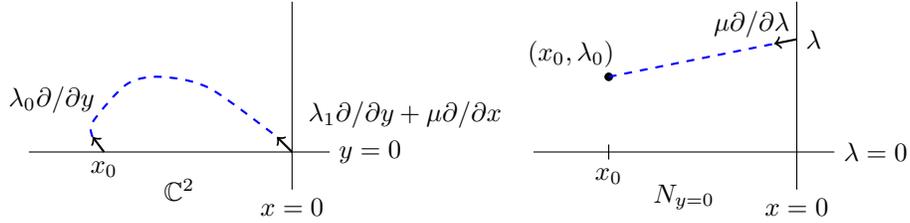
\begin{tikzpicture}[scale = 1]
        \draw (-4,0) -- (0,0);
        \draw (-.5,2) -- (-.5,-.5) node[below] {$x = 0$};
        \draw (0,0) node[right] {$y = 0$};
        \draw (-2,-.5) node {$\C^2$}; 
        \draw[thick,->] (-3,0) -- (-3.15,.2);
        \draw[thick,->] (-.5,0) -- (-.7,.2);
        \draw (-3,0) node [below] {$x_0$};
        \draw (-3.7,.4) node[above] {$\lambda_0 \partial /\partial y$};
        \draw (1,.2) node[above] {$\lambda_1 \partial /\partial y + \mu\partial/\partial x$};
        \draw[dashed,blue,thick] plot [smooth] coordinates {(-3.15,.2) (-3.15, .4) (-2.7, .9) (-2.25, 1) (-1.6, .85) (-.7,.2)};
    \end{tikzpicture}
    \begin{tikzpicture}[scale = 1]
        \draw (-4,0) -- (0,0);
        \draw (-.5,2) -- (-.5,-.5) node[below] {$x = 0$};
        \draw (0,0) node[right] {$\lambda = 0$};
        \draw (-2,-.6) node {$N_{y = 0}$}; 
        \draw[fill=black] (-3, 1) circle (.05cm);
        \draw (-2.8,1.3) node[above,left] {$(x_0,\lambda_0)$};
        \draw (-.5,1.5) node[right] {$\lambda$};
        \draw (-1.1,1.7) node {$\mu\partial/\partial\lambda$};
        \draw (-3,.1) -- (-3,-.1) node[below] {$x_0$};
        \draw[dashed,blue,thick] (-3,1) -- (-.8,1.44);
        \draw[thick,->] (-.5,1.5) -- (-.8,1.44);
    \end{tikzpicture}
\begin{proof}
From the previous statement, we know 
$$
T_{\nabla_0}((x_0,\lambda_0),(x_1,\lambda_1)) = \lim_{y \to 0} (y/\lambda_1)^{B(x_1,0)} T_\nabla((x_0,y),(x_1,y)) (y/\lambda_0)^{-B(x_0,0)}.
$$
Right multiplying by $(x_1/\mu)^{A(0,0)}$ and taking limits as $x_1 \to 0$ yields the standard regularized transport operators
\begin{multline*}
    \lim_{x_1\to 0} (x_1/\mu)^{A(0,0)} T_{\nabla_0}((x_0,\lambda_0),(x_1,\lambda_1)) \cr 
    = \lim_{\substack{ x_1 \to 0 \\ y \to 0}} (x_1/\mu)^{A(0,y)}(y/\lambda_1)^{B(x_1,0)} T_\nabla((x_0,y),(x_1,y)) (y/\lambda_0)^{-B(x_0,0)}.
\end{multline*}
The left hand side is transport in $N_{y=0}$ and the right hand side is transport in $\C^2$ since $\exp(-2\pi i A(0,y))$ is the monodromy operator about $x = 0$ at $(x_1,y)$. 
\end{proof}
\begin{theorem}
\label{thm:hainCan}
The Hain map sends the canonical limit MHS of $\p(U_N,\vv_1)$ to the canonical limit MHS of $\p(E_\uu',\vv_1)$.
\end{theorem}
\begin{proof}
By the previous statement, regularized transport in $\bP_N$ to $\p(E_\uu',\vv_1)$ is equal first transporting to a smooth point of the nodal cubic $E_0$ and then transporting with respect to the KZ connection to $\vv_1$.
\end{proof}

\section{Index of notation}

\begin{longtable}{ll @{\extracolsep{\fill}} r}
$\bmu_N$ & the set of $N$th roots of unity & p.~\pageref{not:bmu} \cr$\Gm$ & the multiplicative group of a field (usually $\C$) & p.~\pageref{not:Gm} \cr
$\Li_m$ & the $m$th polylogarithm function & p.~\pageref{not:Li} \cr
$\h$ & the upper half plane & p.~\pageref{not:h} \cr
$\piun(X,x)$ & the unipotent completion of $\pi_1(X,x)$ & p.~\pageref{not:piun} \cr
$\p(X,x)$ & the Lie algebra of $\piun(X,x)$ & p.~\pageref{not:p} \cr
$\omega^\bullet$ & rational fiber functor of $\MTM_N$ & p.~\pageref{not:DGw} \cr
$\cO_N$ & the ring of $S$-integers of $\Q(\bmu_N)$ with $S = \{p \mid p|N\}$ & p.~\pageref{not:ON} \cr
$\MTM_N$ & the category of mixed Tate motives over $\cO_N$ & p.~\pageref{not:MTMN} \cr
$\cG_N^\bullet$ & the fundamental group of $\MTM_N$ with respect to $\omega^\bullet$ & p.~\pageref{not:GbulletN}\cr
$\K_N^\bullet$ & the prounipotent radical of $\cG_N^\bullet$ & p.~\pageref{not:KNS} \cr
$\k_N^\DR$ & the Lie algebra of $\K_N^\DR$ & p.~\pageref{not:kOkS} \cr
$\Le$ & the Lefschetz period of $\Q(-1)$ & p.~\pageref{not:Le} \cr 
$\cH$ & Brown's category of Hodge realizations & p.~\pageref{not:HR} \cr
$\w_\cH^\bullet$ & fiber functors of $\cH$ & p.~\pageref{not:HRff} \cr
$\w_s$ & the 1-form $\frac{dw}{w-s}$ & p.~\pageref{not:ws} \cr
$U_N$ & $\Gm - \bmu_N$ & p.~\pageref{not:UN} \cr
$\ww_\zeta$ & the tangent vector $\zeta \partial/\partial w \in T_0\C$ & p.~\pageref{not:wz} \cr
$\vv_\zeta$ & the tangent vector $\zeta \partial/\partial w \in T_\zeta\C$ & p.~\pageref{not:vz} \cr
$\cP^\m_N$ & the ring of motivic periods of $\MTM_N$ & p.~\pageref{not:PN} \cr
$\dch$ & the straight line path from $\ww_1$ to $-\vv_1$ & p.~\pageref{not:dch} \cr
$\Phi^\m_{0\eta}$ & the motivic Drinfeld associator from $\ww_\zeta$ to $-\vv_\zeta$ & p.~\pageref{not:drin} \cr
$\p_N^\KZ$ & non-canonical fiber of KZ local system & p.~\pageref{not:KZ} \cr
$Y_\G$ & the modular curve $\G \bbs \h$ & p.~\pageref{not:YG} \cr
$X_\G$ & the compactification of $Y_\G$ & p.~\pageref{not:XG} \cr
$\E_\G$ & the universal elliptic curve over $Y_\G$ & p.~\pageref{not:EG} \cr
$\p_N$ & non-canonical fiber of KZB local system & p.~\pageref{not:pN} \cr
$\E_\G'$ & $\E_\G$ minus single-valued $N$-torsion & p.~\pageref{not:EGminus} \cr
$\bX, \bY, \bt_\zeta$ & generators of $\p_N$ & p.~\pageref{not:XY} \cr
$\uu$ & the tangent vector $\partial/\partial q$ at $q = 0$ in $Y_\G$ & p.~\pageref{not:u} \cr
$E_\uu$ & the first order smoothing of nodal cubic in the direction of $\uu$ & p.~\pageref{not:Eu} \cr
$E_\uu'$ & $E_\uu$ punctured at $\bmu_N \hookrightarrow E_\uu$ & p.~\pageref{not:Eup} \cr
$\H$ & the local system $R_1f_\ast \Q$ associated to $f : \E_\G \to Y_\G$ & p.~\pageref{not:H} \cr
$E_{m,\zeta}$ & the extension of $\Q$ by $\Q(m)$ with period $\Li_m(\zeta)$ & p.~\pageref{not:Eext} \cr
$\e_{m,\zeta}$ & derivations of $\p_N$ in KZB & p.~\pageref{not:ep} \cr
$\bsigmabar_{m,\zeta}$ & generator of $H_1(\k_N)$ dual to $E_{m,\zeta}$ & p.~\pageref{not:sigbar}\cr
$\bsigma_{m,\zeta}$ & non-canonical lift of $\bsigmabar_{m,\zeta}$ to $\k_N$ & p.~\pageref{not:sig} \cr
$\SDer$ & special derviations & p.~\pageref{not:SDer} \cr


\end{longtable}


\begin{thebibliography}{9}


\bibitem{andre}
Y.~Andr\'{e}: {\em Galois theory, motives, and transcendental numbers}, Renormalization and Galois Theories, IRMA Lectures in Math. Theor. Phys. 15 (2009),  165--177.

\bibitem{beilinson}
A.~Beilinson: {\em Higher regulators and values of L-functions}, Current problems in mathematics, Vol. 24, 181--238, Itogi Nauki i Tekhniki, Akad. Nauk SSSR, Vsesoyuz. Inst. Nauchn. i Tekhn. Inform., Moscow, 1984.


\bibitem{BL}
A.~Beilinson \& A.~Levin: {\em The Elliptic Polylogarithm}, Proc. Sympos. Pure Math., Vol. 55 (1994), Part 2, 126--196.


\bibitem{borel}
A.~Borel: {\em Cohomologie de $\SL_n$ et valeurs de fonctions z\^{e}ta aux points entiers}, Ann. Scuola Norm. Sup. Pisa. Cl. Sci. (4) 4 (1977), 613--636.

\bibitem{borel:rank}
A.~Borel: {\em Stable real cohomology of arithmetic groups}, Ann. Sci. \'Eco. Norm. Sup\'er. (4) 7 (1974), 235--272.



\bibitem{brown:MTM}
F.~Brown: {\em Mixed Tate Motives over $\Z$}, Annals of Math. 175 (2012), 949--976.

\bibitem{brown:P1}
F.~Brown: {\em Motivic periods and $\P^1 - \{0,1,\infty\}$}, Proceedings of the
International Congress of Mathematicians---Seoul 2014. Kyung Moon Sa, Seoul, Vol. II (2014), 295--318.

\bibitem{brown:MMV}
F.~Brown: {\em Multiple modular values and the relative completion of the fundamental group of $\M_{1,1}$}, 2017. {\sf [arXiv:1407.5167]}

\bibitem{brown:mzv}
F.~Brown: {\em Multiple zeta values and periods of moduli spaces $\overline{\M}_{0,n}$}, Ann. Sci. \'Eco. Norm. Sup\'er. (4) 42 (2009), 371--489.

\bibitem{brown:periods} 
F.~Brown: {\em Notes on motivic periods}, Commun. Number Theory Phys.
11, No. 3 (2017), 557--655.




\bibitem{CEE}
D.~Calaque, B.~Enriquez, P.~Etingof: {\em Universal KZB equations: the elliptic case}, in Algebra, arithmetic, and geometry: in honor of Yu. I. Manin. Vol. I, Progr. Math., 269, Birkh\"{a}user, Boston (2009), 165--266. 

\bibitem{CG}
D.~Calaque \& M.~Gonzalez: {\em On the universal ellipsitomic KZB connection}, Sel. Math. New Ser. 26, 73 (2020).


\bibitem{chen}
K.~T.~Chen: {\em Iterated path integrals}. Bull. Amer. Math. Soc. 83 (1977), No. 5, 831--879.

\bibitem{deligne:can}
P.~Deligne: {\em Equations diff\'erentielles \'a points singuliers r\'eguliers}, Lecture Notes in Mathematics, Vol. 163, Springer-Verlag, Berlin-New York, 1970.

\bibitem{deligne:P1}
P.~Deligne: {\em Le groupe fondamental de la droite projective moins trois points}, Galois groups over $\Q$ (Berkeley, CA, 1987), Math. Sci. Res. Inst. Publ., 16 (1989), 79--297.

\bibitem{deligne:23468}
P.~Deligne: {\em Le group fondamental unipotent motivique de $\Gm - \bmu_N$, pour $N = 2$, 3, 4, 6 ou 8}, Publications Math\'{e}matiques de l'IH\'{E}S, Volume 112 (2010), 101--141.

\bibitem{DG}
P.~Deligne \& A.~B.~Goncharov: {\em Groupes fondamentaux motiviques de Tate mixte}, Annales Scientifiques de l'\'{Ecole} Normale Sup\'{e}rieure, Vol. 38, Iss. 1 (2005), 1--56.

\bibitem{delmilne}
P.~Deligne \& J.~S.~Milne: {\em Tannakian categories} in {\em Hodge Cycles, Motives, and Shimura Varieties}, Lecture Notes in Mathematics, Vol. 900, Springer (1982), 101--228. Updated 2018, {\sf [https://www.jmilne.org/math/xnotes/tc2018.pdf]}.


\bibitem{drinfeld}
V.~Drinfeld: {\em On quasi-triangular quasi-Hopf algebras and some group closely related with
$\Gal(\Qbar/\Q)$}, Algebra i Analiz 2(4) (1990), 149--181.



\bibitem{furusho}
H.~Furusho: {\em The multiple zeta value algebra and the stable derivation algebra}, Publ. Res. Inst. Math. Sci. 39 (2003), 695--720. 

\bibitem{GKZ}
H.~Gangl, M.~Kaneko \& D.~Zagier: {\em Double Zeta Values and Modular Forms}, Max-Planck-Institut f\"{u}r Mathematik, Bonn (2005).

\bibitem{gonch:mzv}
A.~B.~Goncharov: {\em Multiple polylogarithms and mixed Tate motives}, unpublished preprint, 2001. {\sf [arXiv:math/0103059]}

\bibitem{gonch:sym}
A.~B.~Goncharov: {\em The dihedral Lie algebras and Galois symmetries of $\pi_1^{(l)}(\P^1 - (\{0,\infty\} \cup \mu_N))$}, Duke Math. J., Vol. 110, number 3 (2001), 397--487.



\bibitem{gonch:poly}
A.~B.~Goncharov: {\em Multiple polylogarithms, cyclotomy, and modular complexes}, Math. Res. Letters, Vol. 5, No. 4 (1998), 497--516.


\bibitem{grothendieck}
A.~Grothendieck: {\em On the de~Rham cohomology of algebraic varieties}, Pub. Math. de l'Institut des Hautes \'{E}tudes Scientifiques 29(1) (1966), 95--103.

\bibitem{hain:polylog}
R.~M.~Hain: {\em Classical Polylogarithms}, Motives, Seattle (1994), 3--42.

\bibitem{hain:drt}
R.~M.~Hain: {\em The de~Rham Homotopy Theory of Complex Algebraic Varieties I}, $K$-Theory 1 (1987), 271--324.


\bibitem{hain:bowdoin}
R.~M.~Hain: {\em The geometry of the mixed Hodge structure on the fundamental group}, Proc. Symp. Pure Math 46 Part 2 (1987), 247--282.




\bibitem{hain:kzb}
R.~M.~Hain: {\em Notes on the universal elliptic KZB connection}, Pure and Applied Math. Quarterly, Vol. 16, No. 2 (2020), 229--312. 



\bibitem{hain:MEM}
R.~M.~Hain \& M.~Matsumoto: {\em Universal mixed elliptic motives}, Journal of the Inst. of Math. of Jussieu 19(3) (2020), 663--766.

\bibitem{hainzucker}
R.~M.~Hain \& S.~Zucker: {\em A guide to unipotent variations of mixed hodge structure}. In {\em Hodge Theory}, E.~Cattani, A.~Kaplan, F.~Guill\'en, F.~ Puerta (eds.). Lecture Notes in Mathematics, vol 1246. Springer, Berlin, Heidelberg (1987), 92--106.

\bibitem{HZ}
R.~M.~Hain \& S.~Zucker: {\em Unipotent variations of mixed Hodge structure}, Invent. Math. 88 (1987), 83–124.

\bibitem{hopper}
E.~Hopper: {\em The universal elliptic KZB connection in higher level}, submitted, 2022. {\sf [arXiv:2107.14320]}

\bibitem{thesis}
E.~Hopper {\em Cyclotomic and elliptic polylogarithms and motivic extensions of $\Q$ by $\Q(m)$}, Ph.D. thesis, Duke University, 2021. 


\bibitem{HK}
A.~Huber \& G.~Kings: {\em Degeneration of l-adic Eisenstein classes and of the elliptic polylog}, Invent. Math. 135 (1999), no. 3, 545--594.

\bibitem{ihara}
Y.~Ihara: {\em Some arithmetic aspects of Galois actions in the pro-p fundamental group of $\P^1 - \{0, 1, \infty\}$}, Arithmetic fundamental groups and noncommutative algebra (Berkeley, CA, 1999), 247--273, Proc. Sympos. Pure Math., 70, Amer. Math. Soc., 2002.




\bibitem{KZe}
V.~Knizhnik \& A.~Zamolodchikov: {\em Current algebra and Wess–Zumino model in two dimensions}, Nuclear Phys. B 247 (1984), 83--103.

\bibitem{LR}
A.~Levin \& G.~Racinet: {\em Towards multiple elliptic polylogarithms}, unpublished preprint, 2007, {\sf [arXiv:math/0703237]}.

\bibitem{levine}
M.~Levine: {\em Tate motives and the vanishing conjectures for algebraic $K$-theory}, Algebraic $K$-theory and algebraic topology (Lake Louise, AB, 1991), 167–188, NATO Adv. Sci. Inst. Ser. C Math. Phys. Sci., 407, Kluwer, 1993.


\bibitem{kashiwara}
M.~Kashiwara: {\em A study of variation of mixed Hodge structure}, Publ. Res. Inst. Math. Sci. 22
(1986), no. 5, 991–1024.

\bibitem{KZ}
M.~Kontsevish \& D.~Zagier: {\em Periods}, Mathematics Unlimited --- 2001 and Beyond, Springer, Berlin, Heidelberg (2001), 771--808.

\bibitem{pollack}
A.~Pollack: {\em Relations between derivations arising from modular forms}, undergraduate thesis, Duke University, 2009.


\bibitem{serre}
J.-P.~Serre: {\em Lie algebras and Lie groups, Lectures given at Harvard University, 1964}, Benjamin, 1965. (Second edition, Springer-Verlag, 1992.)

\bibitem{SZ}
J.~Steenbrink \& S.~Zucker: {\em Variation of mixed Hodge structure, I}, Invent. Math. 80 (1985), 489--542.

\bibitem{tsunogai}
H.~Tsunogai: {\em On some derivations of Lie algebras related to Galois representations}, Publ. Res. Inst. Math. Sci. 31 (1995), 113--134.


\end{thebibliography}
\end{document}